\documentclass[11pt,a4paper,reqno]{amsart}
\usepackage{bm}

\usepackage{setspace}
\usepackage{indentfirst}

\usepackage{tabls}
\usepackage{url}
\usepackage{color}
\definecolor{refkey}{rgb}{1,0,0}
\definecolor{labelkey}{rgb}{0,0,1}
\definecolor{labelkey}{rgb}{1,1,1}
\usepackage[linktocpage = true]{hyperref}
\hypersetup{
 colorlinks = true,
 citecolor = blue,
}

\usepackage{times}
\usepackage{amsthm}
\usepackage{amsmath}
\usepackage{amsfonts}
\usepackage{amssymb}
\usepackage{amsrefs}
\usepackage{mathrsfs}
\usepackage{xypic}
\usepackage{enumerate}
\usepackage{extarrows}
\usepackage{graphicx}
\usepackage{subfigure}
\usepackage[toc,page]{appendix}
\usepackage[all,cmtip]{xy}

\allowdisplaybreaks

\hypersetup{pdfpagemode=UseOutlines,colorlinks=false,pdfpagelayout=SinglePage,pdfstartview=FitH,bookmarksopen=true}

\theoremstyle{plain}
\newtheorem{theorem}{Theorem}[section]
\newtheorem{lemma}[theorem]{Lemma}
\newtheorem{corollary}[theorem]{Corollary}
\newtheorem{definition}[theorem]{Definition}
\newtheorem{proposition}[theorem]{Proposition}

\newtheorem{example}[theorem]{Example}
\newtheorem{remark}[theorem]{Remark}

\theoremstyle{definition}

\newtheorem*{note*}{Note}

\theoremstyle{remark}

\DeclareMathOperator{\ad}{ad} \DeclareMathOperator{\Ad}{Ad}
 
\DeclareMathOperator{\id}{id}

\DeclareMathOperator{\Hom}{Hom}

 \DeclareMathOperator{\mult}{mult}

\newcommand{\ba}[2]{[#1,#2]}

\newcommand{\inserts}{\iota}

\newcommand{\galgebra}{\mathfrak{g}}

\newcommand{\crossedmoduletriple}[3]{(#1\stackrel{#2}{\rightarrow}#3)}

\newcommand{\moduleaction}{ \triangleright}

\newcommand{\rhopush}{D_{\rho}}

\newcommand{\minuspower}[1]{(-1)^{#1}}

\newcommand{\thetaalgebradegone}{\mathrm{\vartheta}}

\newcommand{\galgebradegzero}{\mathfrak{g}}

\newcommand{\actionmapaction}[2]{#1 \moduleaction #2}

\newcommand{\tobefilledin}{\,\stackrel{\centerdot}{}\,}

\newcommand{\jet}{\mathfrak{J}}
\newcommand{\heat}{\mathfrak{H}}
\newcommand{\core}{\mathfrak{h}}
\newcommand{\liftingd}{\jmath^1}

\newcommand{\Gpd}{\mathcal{G}}

\newcommand{\invstar}{\mathrm{inv}_*}

\newcommand{\CinfM}{C^\infty(M)}

\newcommand{\littlecirc}{}

\newcommand{\XkmultG}[1]{\mathfrak{X}^{#1}_{\mathrm{mult}}(\Gpd)}
\newcommand{\OmegakmultG}[1]{\Omega^{#1}_{\mathrm{mult}}(\Gpd)}
\newcommand{\Drhostarpush}{D_{\rho^*}}

\newcommand{\sign}{\mathrm{sgn}}

\newcommand{\piM}{\underline{P}}

\begin{document}

\title{
Multiplicative forms on Poisson groupoids
}

\author{Zhuo Chen}
\address{Department of Mathematics, Tsinghua University, Beijing 100084, China}
\email{\href{mailto:chenzhuo@tsinghua.edu.cn}{chenzhuo@tsinghua.edu.cn}}

\author{Honglei Lang$^\diamond$} 
\address{College of Science, China Agricultural University, Beijing 100083, China}
\email{\href{mailto:hllang@cau.edu.cn}{hllang@cau.edu.cn}(corresponding author)}

\author{Zhangju Liu}
\address{School of Mathematical and Statistical Sciences, Henan University, Kaifeng 475004, China}
\email{\href{mailto: liuzj@pku.edu.cn}{zhangju@henu.edu.cn}}

\makeatother

\begin{abstract} 
	Given a Lie groupoid $\Gpd$ over $M$,  $A$   the tangent Lie algebroid of $\Gpd$, and $\rho: A\rightarrow TM$   the anchor map, we provide a formula that decomposes an arbitrary multiplicative $k$-form $\Theta$ on $\Gpd$ into two parts. The first part is $e$, a $1$-cocycle of $\jet \Gpd$ valued in $\wedge^k T^*M$, and the second part is $\theta\in \Gamma(A^*\otimes (\wedge^{k-1} T^*M))$ which is $\rho$-compatible, meaning that $\iota_{\rho(u)}\theta(u)=0$ for all $u\in A$. We call this pair of data $(e,\theta)$ the $(0,k)$-characteristic pair of $\Theta$. Next, we prove that if $\Gpd$ is a Poisson Lie groupoid, then the space $\OmegakmultG{\bullet}$ of multiplicative forms on $\Gpd$ has a differential graded Lie algebra (DGLA) structure. Furthermore, when combined with $\Omega^\bullet(M)$, which is the space of forms on the base manifold $M$, $\OmegakmultG{\bullet}$ forms a canonical DGLA crossed module. This supplements a previously known fact that multiplicative multivector fields on $\Gpd$ form a DGLA crossed module with the Schouten algebra $\Gamma(\wedge^\bullet A)$ stemming from the tangent Lie algebroid $A$.

\end{abstract}

\maketitle
\tableofcontents



\section*{Introduction}

The motivation for this study derives from two primary sources. Firstly, we seek to build upon our prior research on multiplicative multi-vector fields on Lie groupoids \cite{CL1}; our current focus is directed towards multiplicative forms. Secondly, we aim to investigate Poisson Lie groupoids and, in particular, the constituting multiplicative forms of their induced graded Lie algebras.

 The concept of Lie groupoids was introduced by Ehresmann in the late 1950's \cite{a,b}  to describe smooth symmetries of a smooth family of objects. That is, the collection of arrows is a manifold $\Gpd$, the set of objects is a manifold $M$ called the base, and all the structure maps of the groupoid are smooth. Taking sources of arrows defines the source map $s:\Gpd\to M$, and similarly, one has the target map   $t:\Gpd\to M$, both being considered as part of the groupoid structures.  Let us denote by   $\Gpd \rightrightarrows M$ for such a Lie groupoid. Its infinitesimal counterpart, namely the tangent Lie algebroid of $\Gpd$, is defined and denoted by $A:=\ker(s_*) |_M$, i.e., vectors tangent to the $s$-fibers of $\Gpd$ along $M$.
The theory of Lie groupoids and   Lie algebroids has become a far reaching extension of the usual Lie theory, and it finds application in many areas of mathematics. The reader is referred to the texts \cites{Mackenzie, CF} for more useful information  on this subject.

Geometric structures compatible with the groupoid structure  are often called multiplicative. 
  Multiplicative objects have attracted widespread attention  because  they can be regarded as geometric structures on differentiable stacks \cite{Luca}. We refer to  \cite{Kosmann} for a thorough  survey of   different kinds of multiplicative structures on Lie groupoids   defined and studied in the past decades.  For this article, we hope that readers have some familiarities of   works on   multiplicative Poisson structures \cites{W1,Weinstein1987, LSX}.

A multiplicative vector field on a Lie groupoid is a vector field generating a flow of local groupoid automorphisms \cite{MX2}. In a recent work \cite{BCLX},  Bonechi,  Ciccoli,  Laurent-Gengoux, and  Xu have shown    a canonical graded Lie algebra (GLA for short) crossed module structure   on the space of multiplicative multi-vector fields of a groupoid --- Let $\Gpd \rightrightarrows M$ be a Lie groupoid over $M$ and $A$ its tangent Lie algebroid. The aforementioned GLA crossed module is actually composed of  a triple:  	
\begin{equation}\label{Introtemp1}\Gamma(\wedge^\bullet A) \xrightarrow{T} \XkmultG{\bullet} ,\qquad\mbox{where~}~  {T}(u):= \overleftarrow{u}-\overrightarrow{u},\qquad  \forall u\in \Gamma(\wedge^\bullet A).\end{equation}
Here $\Gamma(\wedge^\bullet A)$ is the Schouten algebra of $A$,  $\XkmultG{\bullet}$ stands for the space of multiplicative multi-vector fields of $\Gpd$, and the action of $\XkmultG{\bullet}$ on $\Gamma(\wedge^\bullet A)$ is intrinsic (see Example \ref{exBCLX}). 
 It is also proved in \cite{BCLX} that the homotopy equivalence class of the GLA crossed module $\Gamma(\wedge^\bullet A) \xrightarrow{T} \XkmultG{\bullet}$ is  invariant under Morita equivalence of Lie groupoids, and thus is considered as the multi-vector fields on the corresponding differentiable stack. 

In our previous work \cite{CL1}, we have established a   formula for multiplicative multi-vector fields (see Theorem \ref{Thm:mainexplicitformula}) --- any multiplicative $k$-vector field $\Pi$ on a Lie groupoid $\Gpd \rightrightarrows M$ can be decomposed into   
\begin{equation}\label{formula3intro}
\Pi_g=R_{g*}c([b_g])+L_{[b_g]}\big(\frac{1-e^{-\rhopush}}{\rhopush} (\pi)\big)_{s(g)}
\end{equation}
where $g\in \Gpd $,   $b_g$  is a bisection through $g$,   $c: \jet \Gpd \to \wedge^k A$ is a $1$-cocycle,  $\pi\in  \Gamma(TM\otimes (\wedge^{k-1} A))$  is a  $\rho$-compatible $(k,0)$-tensor (see Definition   \ref{Defn:properktensor}), and $D_\rho$  is    a  degree $0$ derivation on $ \wedge^\bullet (TM\oplus A)$. We call $(c,\pi)$ a $(k,0)$-characteristic pair on  $\Gpd$. More facts about multiplicative multi-vector fields are recalled in Section \ref{Sec:multikvectosrecalling}.

In duality to multiplicative multi-vector fields, differential forms on Lie groupoids suitably compatible with the groupoid structure are referred to as multiplicative forms, and are the main objects of
interest in this paper. After their first  
appearance with the advent of symplectic groupoids \cites{Karasev,Weinstein1987},
  a lot of interesting works on  multiplicative forms of Lie groupoids   have emerged. For example,  a one-to-one correspondence between  multiplicative forms (with certain coefficients) and  Spencer operators on Lie algebroids is established in \cite{C}, from which we find a lot of inspiration.  

Our first objective is to decompose multiplicative forms by drawing an analogy with Equation \eqref{formula3intro} of multiplicative multivector fields. Fortunately, Crainic, Salazar, and Struchiner \cite{C} have discovered an important result in which a multiplicative $k$-form $\Theta$ on $\Gpd$ can be characterized by a $(0,k)$-characteristic pair $(e,\theta)$. Here, $e$ is a 1-cocycle of the jet groupoid $\jet \Gpd$ valued in $\wedge^k T^*M$,  i.e. 
$e\in Z^1(\jet \Gpd , \wedge^k T^*M)$, and $\theta\in \Gamma(A^*\otimes (\wedge^{k-1} T^*M))$ is a $\rho$-compatible $(0,k)$-tensor (see Definition \ref{Def:rhocompatible0ktensor}). Using this tool, we can analyze the constituent elements of multiplicative forms and determine the relationship between multiplicative forms $\Theta$ and $(e,\theta)$. Our first main Theorem \ref{main2} states that for any bisection $b_g$ passing through $g\in \Gpd$, a nice decomposition can be obtained: 
	\begin{equation}\label{intro1}
	\Theta_g=R_{[b_g^{-1}]}^*\big(e[b_g]+\frac{e^{\Drhostarpush }-1}{\Drhostarpush }(\theta)_{t(g)}\big).\end{equation} 
	Here $\Drhostarpush$ is  a degree $0$ derivation  $\wedge^\bullet(T^*M\oplus A^*)\to \wedge^\bullet (T^*M\oplus A^*)$ (see  Equation \eqref{Eqt:Drhostarpush}).
	
	The infinitesimal counterparts of multiplicative forms on Lie groupoids are certain structures on Lie algebroids, and they are called    infinitesimally multiplicative (IM for short) forms \cite{BC}.
	Likewise, the infinitesimals of  $(0,k)$-characteristic pairs on Lie groupoids are $(0,k)$-characteristic pairs on Lie algebroids (see Definition \ref{2.16} and Proposition \ref{Prop:groupoidchartoalgebroidchar}), and we show that they are equivalent to IM-forms (see Proposition  \ref{Prop:LiepaircharpairtoIMform}).

	Our second goal is to create a new   object (in parallel to the aforementioned GLA crossed module) --- a triple involving  
	\begin{equation}\label{Introtemp2}\Omega^{\bullet}(M)\xrightarrow{~J~} \OmegakmultG{\bullet},\quad \mbox{ where }~
	J(\gamma):= s^*\gamma-t^*\gamma,\quad \forall \gamma\in \Omega^{\bullet}(M). \end{equation}

	Here, $\OmegakmultG{\bullet}$ refers to the space of multiplicative forms on $\Gpd$, and $s$ and $t$ are the source and target maps of the $\Gpd$ groupoid, respectively. However, as of now, the triple in \eqref{Introtemp2} is only a morphism of cochain complexes wherein the spaces $\Omega^{\bullet}(M)$ and $\OmegakmultG{\bullet}$ have the standard de Rham differentials. Our main result regarding this triple states that if $\Gpd$ is a Poisson Lie groupoid, i.e., equipped with a multiplicative Poisson structure $P$, then the triple $\Omega^{\bullet}(M)\xrightarrow{~J~} \OmegakmultG{\bullet}$ becomes a differential graded Lie algebra (DGLA) crossed module . 
	
	Poisson Lie groupoids, which unify Poisson Lie groups \cite{LuW} and symplectic groupoids, were introduced by Weinstein in \cite{Weinstein 1987,W1}. The Poisson structure $P$ on $\Gpd$ upgrades the triple in \eqref{Introtemp1} to a DGLA crossed module, and the two DGLA crossed modules \eqref{Introtemp1} and \eqref{Introtemp2} are related by a natural morphism. This is our second main Theorem \ref{multiform}.

Note that Ortiz and Waldron \cite{OW} had already discovered part of Theorem \ref{multiform}, which indicates that the data $\Omega^{1}(M)\xrightarrow{~J~}\OmegakmultG{1}$ forms a Lie algebra crossed module. This elegant fact is reasserted in Theorem \ref{Thm:LiealgebracrossedmodulePoissonLiegroupoid}, and to ensure comprehensiveness, we present a proof employing our theory of characteristic pairs.

 The paper is structured as follows. In Section \ref{Sec:Prelim}, we provide an overview of the fundamental concepts of Lie groupoids, Lie algebroids, their corresponding jets, and multiplicative multi-vector fields. We also establish the relation between $(k,0)$-characteristic pairs and multiplicative multi-vector fields. Next, in Section \ref{Section:multforms-charpairs}, we examine multiplicative forms, $(0,k)$-characteristic pairs on Lie groupoids and their interconnections, culminating in our main result, Theorem \ref{main2}. We also explore the infinitesimal theories of $(0,k)$-characteristic pairs and IM forms on Lie algebroids. We then focus on transitive Lie groupoids and Lie algebroids. In Section \ref{Sec:CartanComplex}, we derive a variety of essential formulas related to multiplicative multi-vector fields and forms. Finally, in Section \ref{Sec:PoissonGpd}, we investigate multiplicative forms on Poisson groupoids, and present our second main outcome, Theorem \ref{multiform}, along with its proof, which depends on propositions and lemmas developed earlier in the paper.


 To clarify and make more concise, we recommend consulting the following works related to Lie algebroid IM forms integrated to Lie groupoids: Calaque and Bursztyn \cite{BC}, Cabrera and Ortiz \cite{BCO}, and Cabrera, Marcut, and Salazar \cite{CMS0,CMS1}. For further information on multiplicative tensors and their infinitesimals, please refer to Bursztyn and Drummond's work \cite{BD}. Additionally, we suggest exploring a relevant piece of work by Jotz Lean,  Sti\'enon,    Xu \cite{JSX} on multiplicative generalized complex structures that could be a valuable resource for further research.

 \subsection*{List of conventions and notations}\
 
 Throughout the paper, $M$ stands for a smooth manifold and $k$ denotes a \textit{positive integer} (usually within the range $1\leqslant k\leqslant  \dim M+1$). Further, `GLA' stands for `graded Lie algebra', and `DGLA'   for `differential graded Lie algebra'. Some commonly used symbols are listed below.

 
 \smallskip
 
 \begin{enumerate}
 	\item $\mathrm{Sh}(p,   q)$ \quad  \textemdash  \quad  the set of $(p,   q)$-shuffles; 
 	A $(p,   q)$-shuffle is a permutation
 	$\sigma$ of the set
 	$\{1, 2, \cdots, p+q\}$
 	such that
 	$\sigma(1)<\cdots < \sigma(p)$
 	and
 	$\sigma(p+1)<\cdots < \sigma(p+q)$;
 	
 	\item $T^\sharp$ \quad \textemdash \quad   the contraction map $U^*\to V$, $u^*\mapsto \iota_{u^*}T$ for a given tensor $T\in U\otimes V$; 
 
	\item  $\Gpd \rightrightarrows M$ \quad \textemdash \quad     a Lie groupoid over $M$; 
	
	\item $\XkmultG{\bullet}$ $~(\OmegakmultG{\bullet})$ \quad \textemdash \quad    the space of multiplicative multi-vector fields (multiplicative forms) of $\Gpd$;
	
	\item $(A,[-,-],\rho)$ \quad \textemdash \quad   a Lie algebroid  with its bracket and anchor maps; Usually, $A$ is the tangent Lie algebroid of $\Gpd$;

		\item $b_g$ \quad \textemdash \quad     a (local) bisection  on  $\Gpd$ which passes through $g\in \Gpd$;  
		
		\item  $[b_g]$ \quad \textemdash \quad  the first jet of $b_g$ at $g\in \Gpd$; 

	\item $\jet \Gpd$ $~ (\jet A)$\quad \textemdash \quad     the   jet groupoid of $\Gpd$  (the jet Lie algebroid of $A$); 

	\item $\liftingd u$ \quad \textemdash \quad 
	the section of jet in  $\Gamma({\jet A})$ arising from   $u\in \Gamma(A)$; See Equation \eqref{Eqt:liftingd};
	
	\item  $\heat$ $~ (\core)$ \quad \textemdash \quad bundle of isotropy jet groups (bundle of  isotropy Lie algebras); 

 	\item $\Ad$ \quad \textemdash \quad     the adjoint action of the jet groupoid $\jet \Gpd$ on $A$ and $TM$ and also on $\wedge^k A$ and $\wedge^k TM$;
 	
 	\item $\Ad^\vee$ \quad \textemdash \quad     the coadjoint action of $\jet \Gpd$ on $A^*, T^*M$ and also on $\wedge^k A^*, \wedge^k T^*M$;

	\item   $\Drhostarpush$  \quad \textemdash \quad  a degree $0$ derivation  $\wedge^\bullet(T^*M\oplus A^*)\to \wedge^\bullet (T^*M\oplus A^*)$; See Equation \eqref{Eqt:Drhostarpush};

	\item  $\mathrm{CP}^\bullet (A)$  \quad \textemdash \quad the space of characteristic pairs on $A$;

	\item  $\mathrm{IM}^\bullet (A)$  \quad \textemdash \quad the space of IM-forms on $A$.

 \end{enumerate} 



\section{Preliminaries}\label{Sec:Prelim}
This section provides an overview of preliminary concepts such as Lie groupoids, Lie algebroids, their jets, and our notation and conventions regarding Lie groupoid and Lie algebroid modules. Additionally, we briefly recall multiplicative multi-vector fields and forms, with references to \cites{Mackenzie, CF} and our previous work \cite{CL1}.

\subsection{Lie algebroid and groupoid modules}\label{Sec:LALGPbasic}
  	A  {Lie  algebroid} is a $\mathbb{R}$-vector bundle $A \rightarrow M$ endowed with a Lie bracket $[{\tobefilledin},{\tobefilledin}]$ in its space of sections $\Gamma(A)$, together with a bundle map $\rho:  ~ A \rightarrow T M $ called {\bf anchor}, such that
	$\rho:~~ \Gamma(A) \rightarrow \mathfrak{X}^1(M) $ is a morphism of Lie algebras and
	\[
	[u, \, fv]=f[u, \,v]+(\rho(u)f)v
	\]
	holds for all $u, v \in \Gamma(A)$ and $f\in  C^\infty({M})$.

By  an {\bf $A$-module}, we mean a vector bundle $E\to M$ which is endowed with an $A$-connection:
$$
\nabla:~\Gamma(A)\times  \Gamma(E)\to \Gamma(E)
$$
which is flat:
$$
\nabla_{[u,v]}e=\nabla_u\nabla_v e-\nabla_v\nabla_u e,\qquad\forall u,v\in \Gamma(A), e\in \Gamma(E).
$$

Given an $A$-module $E$, we have the standard Chevalley-Eilenberg complex $(C^\bullet(A,E),d_A)$, where
$$
C^\bullet (A,E):=\Gamma(\Hom(\wedge^\bullet  A , E )) 
$$
and the coboundary operator  $d_A: ~C^{n} (A,E)\to C^{n+1} (A,E)$ is given by
\begin{eqnarray*}
	&&(d_A\lambda)(u_0,u_1,\cdots, u_{n})\\
	&=& \sum_{i}\minuspower{i}\nabla_{u_i} \lambda(\cdots, \widehat{u}_i, \cdots)
	+\sum_{i<j}\minuspower{i+j}\lambda([u_i,u_j],\cdots, \widehat{u}_i, \cdots, \widehat{u}_j, \cdots),
\end{eqnarray*}
for all $u_0,\cdots, u_n\in\Gamma(A)$.  


Let $\Gpd $  be a Lie groupoid over a smooth manifold $M$ with its   source and target maps
being denoted by $s:\Gpd\to M$ and  $t:\Gpd\to M$, respectively. We use the short notation  $\Gpd \rightrightarrows M$ in which $s$ and $t$ are omitted to denote such a Lie groupoid. We treat the set of identities $M\hookrightarrow \Gpd $   as a submanifold of $\Gpd $. The groupoid multiplication of two elements $g$ and ${r}$ is denoted by $g{r}$, provided that $s(g)=t({r})$. The collection of such pairs   $(g,{r})$,  called composable pairs, is denoted by $\Gpd ^{(2)}$. 
The groupoid inverse map of $\Gpd $ is denoted by $\mathrm{inv}: ~\Gpd \rightarrow \Gpd $. For $g\in \Gpd $, its inverse $\mathrm{inv}(g)$ is also denoted by $g^{-1}$.

A {\bf bisection} of $\Gpd $ is a smooth splitting $b: M\to \Gpd $ of the source map $s$ (i.e., $s{\littlecirc} b=\mathrm{id}_M$) such  that $$\phi_{b}:=t{\littlecirc} b:\quad M\to M$$ is a diffeomorphism. The set of bisections of $\Gpd $ forms a group which we denote by $\mathrm{Bis}(\Gpd )$ and its  identity element is   $\mathrm{id}_M$.  The multiplication $bb'$ of $b$ and $b'\in  \mathrm{Bis}(\Gpd )$ is given by
$$ (b  b')(x):=b(\phi_{b'}(x))b'(x),\qquad \forall x\in M.$$ 

A bisection $b$ defines a diffeomorphism of $\Gpd $ by left multiplication:
$$
L_b:\Gpd \to \Gpd, \quad L_b(g):=b(t(g))g,\qquad  \forall g\in \Gpd .
$$  We call $L_b$ the left translation by $b$.
Similarly, $b$ defines the right translation:
$$
R_b:\Gpd \to \Gpd, \quad R_b(g):=g b(\phi_b^{-1}s(g)),\qquad  \forall g\in \Gpd .
$$

In particular, there is an induced map $R_b^{!}: M\to \Gpd$ which is the restriction of $R_b$ on $M$:
\begin{equation}\label{Eqt:bt}
R_b^{!}(x):=R_b(x)=b(\phi_b^{-1} (x)),\qquad \forall x\in M.
\end{equation}
Clearly, $R_b^{!}$ is a section of the fibre bundle $\Gpd\stackrel{t}{\to}M$.



A local bisection $b$ passing through the point $b(x)=g$ on $\Gpd $, where $x=s(g)\in U$, is also denoted by $b_g$, to emphasise the particular point $g$.\emph{ In the sequel, by saying a   bisection through $g$, we mean a local bisection that passes through $g$.
}

A (left) {\bf $\Gpd $-module}   is a vector bundle $E\rightarrow M$  together with a smooth assignment: $g\mapsto \Phi_g$, where $g\in \Gpd $ and $\Phi_g\in \mathrm{GL}(E_{s(g)},E_{t(g)})$ (the set of isomorphisms from $E_{s(g)}$ to $E_{t(g)}$),
satisfying
\begin{enumerate}
	\item $\Phi_x=\mathrm{id}_{E_x}$, for all $x\in M$;
	\item $\Phi_{gr}= { \Phi_g{\littlecirc} \Phi_r}$, for all composable pairs $(g,r)$.
\end{enumerate}


 We recall the tangent Lie algebroid $(A,  \rho, [~,~]_A) $ of the Lie groupoid $\Gpd $. In fact, 
$A=\ker s_*|_M$ and the anchor map $\rho :~A\rightarrow TM$ is simply $t_*$. For  $u,v\in \Gamma({A})$,
the Lie bracket $[u,v]$ is determined by
$$\overrightarrow{[u,v]}=[\overrightarrow{u},\overrightarrow{v}].$$
Here $\overrightarrow{u}$ denotes the right-invariant vector field on $\Gpd $ corresponding to $u$. In the meantime, the left-invariant vector field corresponding to $u$, denoted by $\overleftarrow{u}$, is related to $\overrightarrow{u}$ via
$$
\overleftarrow{u}= - \invstar  (\overrightarrow{u})=-\overleftarrow{\invstar  u}.
$$

A $\Gpd $-module $E$ is also an $A$-module, i.e., there is an $A$-action on $E$,
$$
\nabla:~\Gamma(A)\times  \Gamma(E)\to \Gamma(E)
$$
defined by
$$
\nabla_{u}e=  \frac{d}{d\epsilon }|_{\epsilon =0} \Phi_{\exp{(-\epsilon u)}}e,\qquad \forall u\in \Gamma(A), e\in \Gamma(E).
$$

An {\bf $n$-cochain} on $\Gpd $ valued in the $\Gpd $-module $E$ is a smooth map $c: \Gpd ^{(n)}\to E$ such that $c(g_1,\cdots,g_n)\in E_{t(g_1)}$. 
 Denote by $C^n(\Gpd , E)$  the space of $n$-cochains.
The coboundary operator
\[ d_\Gpd: C^n(\Gpd , E)\to C^{n+1}(\Gpd , E),\]
is  standard:
\begin{itemize}
	\item[(1)] for $n=0$, $\nu\in C^0(\Gpd,E)=\Gamma(E)$, define 
	$$
	(d_\Gpd \nu)(g)=\Phi_g \nu_{s(g)} -\nu_{t(g)},\qquad\forall g\in \Gpd;
	$$ 	\item[(2)] for  $c\in C^n(\Gpd , E)$ $(n\geqslant  1)$ on $\Gpd$, define 
	\begin{eqnarray*}
		(d_\Gpd c)(g_0,g_1,\cdots,g_{n })&=&\Phi_{g_0} c(g_1,\cdots,g_{n })\\ &&+\sum_{i=0}^{n-1} (-1)^{i-1} c(g_0,\cdots,g_ig_{i+1},\cdots,g_{n})+(-1)^{n+1} c(g_0,\cdots,g_{n-1}).
	\end{eqnarray*}
\end{itemize}

A groupoid $1$-cocycle $c$ induces a Lie algebroid $1$-cocycle $\widehat{c}: A\to E$   by the following formula.  For each $u\in A_x$, choose a smooth curve $\gamma(\epsilon)$  in   the $s$-fibre $s^{-1}(x)$ such that $\gamma'(0)=u$. Then   $\widehat{c}(u)\in E_x$ is defined by
\begin{equation}\label{Eqt:hatcdefinition}
\widehat{c}(u):=  -\frac{d}{d\epsilon }|_{\epsilon =0}~ \Phi^{-1}_{\gamma(\epsilon)} c(\gamma(\epsilon))= \frac{d}{d\epsilon }|_{\epsilon =0}~ c(\gamma(\epsilon)^{-1}).
\end{equation}
We call $\widehat{c}$ the infinitesimal of $c$.
Note that our convention is slightly different from that in \cite{VE} (up to a minus sign).
Also, it is easily verified that, for $\nu\in \Gamma(E)$, we have
\begin{equation}\label{Eqt:dGpddA}
\widehat{d_\Gpd\nu}=d_A\nu.
\end{equation}



\subsection{The jet Lie algebroid and jet Lie groupoid}\label{Appendix:jetLiealgebroid}

The first jet space of a Lie algebroid $A$, denoted by $\jet A$, is also a Lie algebroid (see \cites{CLomni,C}), which fits into 
 an exact sequence of Lie algebroids:
\begin{equation*}\label{Eqt:jetalgebroidexactseq}
0\to\core  \xrightarrow{i}\jet A\xrightarrow{p} A\to 0.
\end{equation*}
Here $\core=T^*M\otimes A=\mathrm{Hom}(TM,A)$ is a bundle of Lie algebras. 
We will call $\core$ the bundle of \textbf{{isotropy} jet Lie algebras}.

We have a natural lifting map $\liftingd:\Gamma(A)\to \Gamma({\jet A})$ that sends $u\in \Gamma(A)$ to its jet $\liftingd u\in \Gamma({\jet A})$:
\begin{equation}\label{Eqt:liftingd}
(\liftingd u)_x=[u]_x,\qquad\forall x\in M.
\end{equation}
Moreover, the Lie bracket  of $\Gamma(\jet A)$ is determined by the  relation:
\begin{eqnarray*}\label{Eqt:jetbracket1}
[\liftingd u_1,\liftingd u_2]&=&\liftingd [u_1,u_2].
\end{eqnarray*}
The anchor of $\jet A$ is simply: \[\rho_{\jet A}(\liftingd u_1+df\otimes u_2)=\rho(u_1).\]

The vector bundles $A$   and $TM$ are modules of the jet algebroid $\jet A$ via adjoint actions:
\begin{eqnarray*}
	\ad_{\liftingd u}v=[u,v],\qquad \ad_{df\otimes u}v=-\rho(v)(f)u,
\end{eqnarray*}
and 
\[\ad_{\liftingd u} X=[\rho(u),X],\qquad \ad_{df\otimes u}X=-X(f)u,\]
where $u,v\in\Gamma(A)$, $f\in {\CinfM}$ and $X\in \mathfrak{X}^1(M)$.
Then $\jet A$ also naturally acts on $\wedge^k A, \wedge^k TM, \wedge^k A^*$ and $\wedge^k T^*M$.

Let $b_g,b'_{g}$ be two  bisections through $g\in \Gpd $. They are said to be equivalent at $g$ if $b_{*x}=b'_{*x}:~T_xM\to T_g\Gpd $, where $x=s(g)$. The equivalence class, denoted by $[b]_x$, or $[b_g]$, is called the first jet of   $b_g$.

The  {\bf (first) jet groupoid} $\jet \Gpd $ of a Lie groupoid $\Gpd $, consisting of such  jets  $[b_g]$, is a Lie groupoid over $M$.
The source and target  maps of $\jet \Gpd $ are given by \[s[b_g]=s(g), \qquad t[b_g]=t(g).\] The multiplication is given by
\[
[b_{g}][b'_{r}] =[(b   b')_{g{r}}],\]
for two   bisections $b_g$ and $b'_r$ through $g$ and $r$ respectively. 

{Given a   bisection $b_g$ through $g$, its first jet induces a left translation $$L_{[b_g]}:~ T \Gpd |_{t^{-1}(x)} \to T \Gpd |_{t^{-1}(y)},$$ where  $x=s(g)$ and  $y=\phi_b(x)=t(g)$. Similarly we have a right translation $$R_{[b_g]}: T \Gpd |_{s^{-1}(y)} \to T \Gpd |_{s^{-1}(x)}.$$ Note also that when  $L_{[b_g]}$ is restricted on $(\ker t_*)|_{t^{-1}(x)}$, it is exactly the left translation $L_{g*}$. Similarly, when  restricted on $(\ker s_*)|_{s^{-1}(y)}$, $R_{[b_g]}$  is exactly the right translation $R_{g*}$.}

Moreover, the restriction of $L_{[b_g]}$ to $T_x M$  is exactly the tangent map $b_{*x}: T_xM\to T_g\Gpd$, i.e., $L_{[b_g]}(X)=b_{g*}(X)$ for $X\in T_x M$.

A bisection $b$ of $\Gpd $ acts on $\Gpd $ by conjugation 
\[\mathrm{AD}_b (g)=R_{b^{-1}}{\littlecirc} L_b(g)=b(t(g)) g \bigl(b(s(g))\bigr)^{-1},\]
which maps units to units and $s$-fibres to $s$-fibres ($t$-fibres to $t$-fibres as well). There is an induced  action of the jet groupoid $\jet\Gpd $ on $A$ and $TM$.
\begin{definition}\label{Def:jetGpdadjoint} 
There is an action of $\jet \Gpd$, called the {\bf adjoint action}, on $A$ and $TM$:
	$$\Ad_{[b_g]}:\quad A_x\to A_{\phi_b(x)},\qquad T_x M\to T_{\phi_b(x)} M,  \quad\mbox{ where }x=s(g)$$
	defined by
	$$
	\Ad_{[b_g]} u=R_{g^{-1}*}{\littlecirc} L_{[b_g]} (u)=R_{g^{-1}*}(L_{g*}(u-\rho(u))+b_{*}(\rho(u))),\quad u\in A_x,$$
and 
 \[ \Ad_{[b_g]} X=R_{[b_g]^{-1}}{\littlecirc} L_{[b_g]}(X)=\phi_{b*} (X),\quad  X\in T_x M.\]
\end{definition}
Then, we also have the adjoint actions of $\jet \Gpd$ on $\wedge^k A$ and $\wedge^k TM$, all denoted by $\Ad$. In the meantime, we have the \textbf{coadjoint actions} of $\jet \Gpd$ on $\wedge^k A^*$ and $\wedge^k T^*M$ which we will denote by $\Ad^\vee$: 
$$
\Ad^{\scriptsize \vee}_{[b_g]}:=(\Ad_{[b_g]^{-1}} )^*.
$$
Note that the tangent Lie algebroid of $\jet \Gpd$ is the jet Lie algebroid $\jet A$. Taking infinitesimal, we get the adjoint and coadjoint actions of $\jet A$ introduced preciously.

The following exact sequence of groupoids can be easily established: 
\begin{equation*}\label{Eqt:jetgroupoidexactseq}
1\to\heat  \stackrel{I}{\rightarrow }\jet\Gpd \stackrel{P}{\rightarrow } \Gpd \to 1.
\end{equation*}
 The space  $\heat$ consists of  jets $[h]$,  where   
  $h$ is a bisection through $x\in M$. Let us call $\heat$ the bundle of \textbf{{isotropy} jet groups}.

For $[h ]\in \heat_x$, there exists $H: T_x M\rightarrow A_x$ such that
$$
h_{*}(X)= H(X)+X,\qquad X\in T_x M,
$$and \[\phi_{h*}=t_*{\littlecirc} h_{*}=\id+\rho{\littlecirc} H
\in \mathrm{GL}(T_x M).\]
Let us introduce  \[{\underline{\mathrm{Hom}}(TM,A)}:=\{H\in \mathrm{Hom}(TM,A);\id+\rho{\littlecirc} H\in \mathrm{GL}(T M)\}.\]
Then we have  $\heat \cong {\underline{\mathrm{Hom}}(TM,A)}$. In the sequel, for $H\in {\underline{\mathrm{Hom}}(T_xM,A_x)}$, we  write $[h ]=\id+H\in \heat_x$. See \cites{C,CL1,LL} for more details.

Now let us write down the explicit formulas for  the translation  of $\heat$ on $T\Gpd |_M=A\oplus TM$. See \cite{CL1} for its proof.

\begin{lemma}\label{ad for h}
	If $[h ]=\id+H \in \heat_x$, where $H\in {\underline{\mathrm{Hom}}(T_xM,A_x)}$, then the left and right translation maps
	\[~L_{[h ] },\quad R_{[h ] }:~ A_x\oplus T_x M \to A_x\oplus T_x M \] are given by
	\begin{eqnarray}\label{Eqn:Lhx}
	\nonumber L_{[h ]}(u+X )&=& H(X)+u+H \rho(u)+X,\\\label{Eqn:Rhx}
	R_{[h ]}(u+X  )&=&  H(\id+\rho H)^{-1} (X)+u +(\id+\rho H)^{-1} X,
	\end{eqnarray}
	where     $u\in A_x$, $X\in T_x M$.
\end{lemma} Consequently, the adjoint action $\Ad_{[h ]}:~ A_x\oplus T_x M \to A_x\oplus T_x M  $ is given by
	\begin{equation*}\label{Eqt:adhXu}
	\Ad_{[h ]}(u+X )= (\id+H\rho)(u)+ (\id+\rho H)(X),
	\end{equation*}
and the dual maps 
\[R_{[h ]}^*, \quad L_{[h ]}^*: A_x^*\oplus T^*_x M  \to A_x^*\oplus T_x^*M \]
are, respectively,
\begin{eqnarray}
\nonumber L_{[h ]}^*(\chi+\xi )&=&  \chi+\rho^*H^*\chi+\xi+H^*\chi,\\\label{Rhxichi} R_{[h ]}^*(\chi+\xi )&=&  \chi + (\id+\rho H)^{-1*} \xi+(\id+\rho H)^{-1*} H^*(\chi),
\end{eqnarray}
where $\xi\in T_x^* M, \chi\in A_x^*$. In particular, the induced coadjoint action of $\heat$ on $\wedge^k T^*M$ is given by
\begin{equation}\label{coadjoint}
	\Ad^\vee_{[h ]^{-1}}w=
( \Ad_{[h ]})  ^*w=(\id+\rho H)^{*\otimes k} (w),\qquad\forall w\in  \wedge^k T^*_x M.
\end{equation}



\subsection{Multiplicative $k$-vector fields and  characteristic pairs of  $( k,0)$-type}\label{Sec:multikvectosrecalling}
We first recall $\rho$-compatible $( k,0)$-tensors introduced in \cite{CL1}.

\begin{definition}\label{Defn:properktensor} Let $k\geqslant 1$ be an integer.
	A \textbf{$\rho$-compatible $(k,0)$-tensor}  is a section $\pi$   $\in \Gamma(TM\otimes (\wedge^{k-1} A))$ which satisfies \begin{equation*}\label{Eqt:fine-condition}
		\inserts_{\rho^*\xi}\iota_{\eta}\pi
		= -\inserts_{\rho^*\eta}\iota_{\xi}\pi,
		\qquad\forall\xi,\eta\in\Omega^1(M).
	\end{equation*}
\end{definition}

We need a particular operator $D_\rho$ which was   first studied in
\cite{CSX} \quad \textemdash \quad    a  degree $0$ derivation on $ \wedge^\bullet (TM\oplus A)$   
determined by its $(TM\oplus A)$-to-$ (TM\oplus A)$ part:
\begin{eqnarray*}\label{D}
	D_\rho(X+u)=\rho(u),\qquad \forall X\in TM, u\in A.
\end{eqnarray*}
Hence $D_\rho $ maps $(\wedge^pTM)\otimes(\wedge^q A)$ to $(\wedge^{p+1}TM)\otimes(\wedge^{q-1} A)$.
For
$\pi$   $\in   TM\otimes (\wedge^{k-1} A) $,
we introduce
\begin{equation}\label{Bpi}
	B\pi= \frac{1-e^{-\rhopush}}{\rhopush} (\pi)
	=
	\pi-\tfrac{1}{2!}\rhopush\pi+\tfrac{1}{3!}\rhopush^2
	\pi+\cdots+\tfrac{\minuspower{k-1}}{k!}\rhopush^{k-1}\pi\in \wedge^k (TM\oplus A) .
\end{equation}
Note that the term $\rhopush^{j}\pi\in  (\wedge^{j+1} TM)\otimes (\wedge^{k-1-j} A) $.

Recall that $A$, the tangent Lie algebroid of $\Gpd$, is a module of the jet groupoid $\jet\Gpd$ via the adjoint action (see Definition \ref{Def:jetGpdadjoint}).  Therefore, $\jet\Gpd$ also acts on $\wedge^k A$. 
Hence we have a coboundary operator
$d_{\jet \Gpd}: C^{n}(\jet\Gpd , \wedge^k A)\to C^{n+1}(\jet \Gpd ,\wedge^k A)$. 
  We denote by $Z^1(\jet\Gpd , \wedge^k A)$ the set of $1$-cocycles $c:\jet\Gpd\to \wedge^kA$.
\begin{definition}\label{Defn:groupoidcharpair} Let $k\geqslant 1$ be an integer.  A  \textbf{$(k,0)$-characteristic pair} on $\Gpd$ is a pair
	$(c,\pi)  \in Z^1(\jet \Gpd , \wedge^k A)\times \Gamma(TM\otimes (\wedge^{k-1} A))$, where $\pi$ is  $\rho$-compatible    and when $c$ is restricted to $\heat$, it satisfies   \begin{eqnarray}
		\label{formula1}c([h_x])&=&(B\pi)_x-L_{[h_x]} (B\pi)_x,\qquad \forall [h_x]\in  \heat_x,~x\in M.
	\end{eqnarray}
\end{definition}

This paper draws heavily on our previous work \cite{CL1} on multiplicative $k$-vector fields of Lie groupoids.   Before diving into the specific details of these fields, however, it is important to review some introductory information about the tangent and cotangent Lie groupoids. For a more comprehensive understanding, refer to \cite{CDW, Kosmann}. The Lie groupoid $\Gpd \rightrightarrows M$ has a corresponding tangent bundle $T\Gpd$ that is a Lie groupoid over $TM$, with its structure maps determined by the tangent maps of $\Gpd$'s structure maps. Similarly, the cotangent bundle $T^*\Gpd$ is a Lie groupoid over $A^*$, with source and target maps given by $s$ and $t:$ $T^*\Gpd\to A^*$, respectively,   
	\begin{eqnarray}\label{stco}
	\langle s(\alpha_g), u\rangle=\langle \alpha_g,\overleftarrow{u}_g\rangle=\langle \alpha_g,L_{g*}(u-\rho(u))\rangle, \quad\mbox{and}\qquad \langle t(\alpha_g), u\rangle=\langle \alpha_g,\overrightarrow{u}_g\rangle=\langle \alpha_g,R_{g*} u\rangle,
	\end{eqnarray} for all $\alpha_g\in T_g^*\Gpd$ and $u\in  A_{s(g)} $ (or $A_{t(g)}$). 
	The multiplication of composable elements $\alpha_g\in T_g^*\Gpd$ and $\beta_r\in T_r^*\Gpd$ is defined by 
	\[(\alpha_g\cdot \beta_r)(X_g\cdot Y_r)=\alpha_g(X_g)+\beta_r(Y_r),\qquad\forall (X_g,Y_r)\in (T \Gpd )^{(2)}.\] 
 	
	We remark that  both $T\Gpd\rightrightarrows TM$ and $T^*\Gpd \rightrightarrows A^*$ are special cases of VB-groupoids \cites{GM}. Taking $T^*\Gpd$ for example, we have the  law of interchange:
	\begin{eqnarray}\label{interlaw}
	(\gamma_1+\gamma_3)\cdot(\gamma_2+\gamma_4)=\gamma_1\cdot \gamma_2+\gamma_3\cdot \gamma_4,
	\end{eqnarray}
	for all $\gamma_1,\gamma_3\in T^*_g \Gpd$ and $\gamma_2,\gamma_4\in T^*_r \Gpd$ on the premise that $(g,r)\in \Gpd^{(2)}$ and $(\gamma_1,\gamma_2),(\gamma_3,\gamma_4)\in (T^*\Gpd)^{(2)}$.

	 In the  literature, various definitions or characterizations of multiplicative multi-vector fields can be found (see \cites{ILX, Kosmann}). In this article, we   use the following one. 
\begin{definition}\cite{ILX, Kosmann}  Let $k\geqslant 1$ be an integer. 
	A  $k$-vector field $\Pi\in \mathfrak{X}^k(\Gpd)$ on $\Gpd$  is called {\bf multiplicative} if   
	\begin{equation*}
	\xymatrix{
		\oplus^{k-1} T^*\Gpd\ar@{->}@< 2pt> [d]
		\ar @{ ->} @<-2pt> [d]  \ar[r]^{\quad \Pi^\sharp} &T\Gpd\ar@{->}@< 2pt> [d]
		\ar @{ ->} @<-2pt> [d] \\
		\oplus^{k-1} A^*\ar[r]^{\pi^\sharp} & TM \ .}
	\end{equation*}
	 is a Lie groupoid morphism. 	 
	 	Here the groupoid $\oplus^{k-1} T^*\Gpd\rightrightarrows \oplus^{k-1} A^*$ is the direct sum of $T^*\Gpd\rightrightarrows A^*$, 
	 $\Pi^\sharp$ is defined by contraction: $(\alpha^1,\cdots, \alpha^{k-1})\mapsto \Pi(\alpha^1,\cdots, \alpha^{k-1},\tobefilledin)$,   and $\pi^\sharp$ is similar: $(\eta^1,\cdots, \eta^{k-1})\mapsto (-1)^{k-1} \pi(\tobefilledin,\eta^1,\cdots, \eta^{k-1})$, where $\pi=\mathrm{pr}_{\Gamma(TM\otimes \wedge^{k-1} A)}(\Pi|_M)$.
	 
\end{definition}

One of the main results in \cite{CL1} is the following theorem.
\begin{theorem}\label{Thm:mainexplicitformula} There is a one-to-one correspondence between multiplicative $k$-vector fields $\Pi$ on a Lie groupoid $\Gpd \rightrightarrows M$, and   $(k,0)$-characteristic pairs $(c,\pi)$ on $\Gpd$
	such that
	\begin{equation}\label{formula3}
		\Pi_g=R_{g*}c([b_g])+L_{[b_g]}(B\pi)_{s(g)}
	\end{equation}
	holds for all $g\in \Gpd $ and   bisections $b_g$ through $g$, where $B\pi$ is given by Equation ~\eqref{Bpi}.
\end{theorem}
An equivalent form of this formula was found by   Iglesias-Ponte,   Laurent-Gengoux and  Xu as early as ten years ago. But in their final work \cite{ILX}, they did not write it.

The second element $\pi$ $\in \Gamma(TM\otimes (\wedge^{k-1} A))$  of the characteristic pair of $\Pi$ is indeed the $\Gamma(TM\otimes (\wedge^{k-1} A))$-component of $\Pi$'s restriction on $M$, i.e. $$\Pi|_M\in \sum_{i+j=k}  \Gamma((\wedge^{i}  TM)\otimes (\wedge^{j} A)) .$$
We shall call $\pi$ the \textbf{leading term} of $\Pi\in \XkmultG{k} $.
 {The first element of the pair, i.e. the $1$-cocycle $c:\jet\Gpd\to \wedge^kA$, can be determined by $\Pi$: 
\begin{eqnarray}\label{cocyclecp}
c([b_g])=R_{g*}^{-1}(\Pi_g-L_{[b_g]}(B\pi)_{s(g)}).
\end{eqnarray}}
\begin{example} \label{exact case}
The characteristic pair $(c,\pi)$ of an exact multiplicative $k$-vector field  $\Pi=\overleftarrow{u}-\overrightarrow{u}$ for $u\in \Gamma(\wedge^k A)$ is given by 
\[c=d_{\jet \Gpd} u,\qquad \pi=-D_\rho(u).\]
\end{example}

\section{Multiplicative $k$-forms and characteristic pairs of  $(0,k)$-type}\label{Section:multforms-charpairs}

\subsection{Multiplicative $k$-forms}\label{Sec:multkforms}

As usual, we denote by $\Gpd \rightrightarrows M$ a Lie groupoid over $M$ and $\Omega^k{(\Gpd )}=\Gamma(\wedge^k T^*\Gpd )$ the space of differential $k$-forms ($k$-forms for short, $k\geqslant 1$). 
A {\bf multiplicative} $k$-form on $\Gpd$ is an element $\Theta\in \Omega^k(\Gpd )$ satisfying   
\begin{eqnarray}\label{multform}
	m^*\Theta=\mathrm{pr}_1^*\Theta+\mathrm{pr}_2^*\Theta,
\end{eqnarray}
where the   maps $m$ and $\mathrm{pr}_i$  $: \Gpd ^{(2)}\to \Gpd$ are, respectively, the groupoid multiplication and       the projection from $\Gpd ^{(2)}$ to its $i$-th summand (see \cite{BC,BCO,C}).  We denote by $\OmegakmultG{k}$ the space of  multiplicative  $k$-forms on $\Gpd$. 

By default, a multiplicative ${0}$-form is one and the same as a smooth multiplicative function  on $\Gpd$, i.e.     $f\colon \Gpd\to \mathbb{R}$ which is a morphism of Lie groupoids:
$$
f(gr)=f(g)+f(r),\quad \forall (g,r) \in \Gpd^{(2)}.
$$

For $k\geqslant  1$, we have an equivalent characterization: $\Theta\in \Omega^k(\Gpd)$ is multiplicative if and only if 
	$\Theta^\sharp:\oplus^{k-1} T\Gpd\to T^*\Gpd$ is a Lie groupoid morphism.  {For more knowledge on this topic, please refer to \cite{Kosmann}.  }

\subsection{$\rho$-compatible $(0,k)$-tensors}   Let $A\to M$    be a vector  bundle which is equipped with a bundle map $\rho:A\to TM$. 
\begin{definition} \label{Def:rhocompatible0ktensor}Let $k\geqslant 1$ be an integer. 
	A {\bf $\rho$-compatible $(0,k)$-tensor} is vector bundle map $\theta: A\to \wedge^{k-1} T^*M$ such that
	$$\iota_{\rho(u)}\theta(u)=0,\qquad\forall u\in A.$$
	Equivalently, it is an element $\theta \in \Gamma(A^*\otimes (\wedge^{k-1} T^*M))$  subject to the following property: 
	\begin{eqnarray}\label{Eqn:rhocompatiblecondition}
	\iota_{\rho(v)}\iota_u \theta=-\iota_{\rho(u)}\iota_v \theta, \qquad \forall u,v\in A.
	\end{eqnarray} 
\end{definition} 
We also define \begin{equation}\label{Eqt:Drhostarpush} \Drhostarpush : \wedge^\bullet (A^*\oplus T^*M)\to  \wedge^\bullet (A^*\oplus T^*M)\end{equation} as a degree $0$-derivation such that
\begin{eqnarray*}\label{D_*}
	\Drhostarpush (\xi+\alpha)=\rho^*(\xi),\qquad \forall \alpha\in  A^*, \xi\in T^*M.
\end{eqnarray*}
In particular, $\Drhostarpush $ maps $ \wedge^i A^*\otimes (\wedge^{j} T^*M)$ to $ \wedge^{i+1} A^*\otimes (\wedge^{j-1} T^*M)$.

\begin{example}In the particular case that $k=1$, a $\rho$-compatible $(0,1)$-tensor is simply a section $\theta\in \Gamma(A^*)$. No other conditions are needed.  
 If $k=2$,  then a $\rho$-compatible $(0,2)$-tensor is an element $\theta\in \Gamma(A^*\otimes T^*M)$ subject to the condition $\theta(u,\rho(u))=0$, for all $u\in \Gamma(A)$. In other words, $(\id\otimes \rho^*)\theta\in \Gamma(A^*\otimes A^*)$ is skew-symmetric.
\end{example}

\begin{example}\label{2.4}
	Given any $\gamma\in \Omega^k(M)$, the tensor $\Drhostarpush  \gamma\in \Gamma(A^*\otimes (\wedge^{k-1} T^*M))$ is $\rho$-compatible.  In fact, we have
	$$
	\iota_{\rho(v)}\iota_u \Drhostarpush\gamma=\iota_{\rho(v)}\iota_{\rho(u)}\gamma=-\iota_{\rho(u)}\iota_{\rho(v)}\gamma=-\iota_{\rho(u)}\iota_v\Drhostarpush\gamma.
	$$
	
\end{example}
\begin{lemma}\label{formula}
	Given a $\rho$-compatible $(0,k)$-tensor $\theta \in \Gamma(A^*\otimes (\wedge^{k-1}  T^*M))$ and an integer $  j\geqslant  1 $, we have
	\begin{eqnarray}\label{eq}
		\iota_{\rho(u)} \Drhostarpush ^{j-1} \theta =\Drhostarpush ^j (\iota_u \theta )=\frac{1}{j+1}\iota_u \Drhostarpush ^j \theta ,\qquad \forall u\in A.
	\end{eqnarray}
	
\end{lemma}
\begin{proof}
	It is direct to check the following equations:
	\begin{eqnarray}
		\label{eq1}\iota_{u} {\littlecirc} \Drhostarpush -\Drhostarpush {\littlecirc} \iota_u&=&\iota_{\rho(u)},\\
		\label{eq2}\iota_{\rho(u)}{\littlecirc} \Drhostarpush &=& \Drhostarpush {\littlecirc} \iota_{\rho(u)}.
	\end{eqnarray}
	Now we prove Equation ~\eqref{eq} by induction. When $j=1$, it follows from the definition of $\rho$-compatibility, namely Equation ~\eqref{Eqn:rhocompatiblecondition}, that
	$\iota_{\rho (u)} \theta =\Drhostarpush (\iota_u \theta )$. By Equation ~\eqref{eq1}, we get
	\[(\iota_u{\littlecirc} \Drhostarpush -\Drhostarpush {\littlecirc} \iota_u)\theta =\iota_{\rho (u)} \theta =\Drhostarpush (\iota_u \theta )\]
	and thus $\Drhostarpush  (\iota_u \theta )=\frac{1}{2}\iota_u \Drhostarpush  \theta $. Assume that Equation ~\eqref{eq} holds for $j\geqslant  1$. Then using Equation ~\eqref{eq2}, we have
	\[\iota_{\rho(u)} (\Drhostarpush ^j \theta )=(\Drhostarpush {\littlecirc} \iota_{\rho(u)} )(\Drhostarpush ^{j-1}\theta )=\Drhostarpush {\littlecirc} \Drhostarpush ^j (\iota_u \theta )=\Drhostarpush ^{j+1} (\iota_u \theta ).\]
	Besides, applying Equations  ~\eqref{eq1} and \eqref{eq2}, we have
	\begin{eqnarray*}
		&&\Drhostarpush ^{j+1}(\iota_u \theta )=\frac{1}{j+1}(\Drhostarpush {\littlecirc} \iota_u{\littlecirc} \Drhostarpush ^j)\theta \\ &=&
		\frac{1}{j+1}(\iota_u{\littlecirc} \Drhostarpush -\iota_{\rho(u)})(\Drhostarpush  ^j \theta )= \frac{1}{j+1}(\iota_u(\Drhostarpush ^{j+1}\theta )-\Drhostarpush ^{j+1}(\iota_u \theta )),
	\end{eqnarray*}
	which implies that
	\[\Drhostarpush ^{j+1} (\iota_u \theta )=\frac{1}{j+2}\iota_u \Drhostarpush ^{j+1} \theta .\]
	Thus Equation ~\eqref{eq} also holds for $j+1$, and the assertion is proved by induction.
\end{proof}

For the convenience of future use, we define a bundle map $$B:~  A^*\otimes( \wedge^{k-1} T^*M)  \to    \bigoplus_{i=1}^k ~\wedge^iA^*\otimes (\wedge^{k-i} T^*M) $$ by
\begin{eqnarray}\label{Btheta}
	B\theta&:=&\frac{e^{\Drhostarpush }-1}{\Drhostarpush }(\theta)=\theta+\frac{1}{2} \Drhostarpush  \theta+\frac{1}{3!}\Drhostarpush ^2 \theta+\cdots+\frac{1}{k!} \Drhostarpush ^{k-1} \theta. 
\end{eqnarray}

Note that the term $\Drhostarpush ^{j-1} \theta$ sits in   $\Gamma(\wedge^{j}A^*\otimes \wedge^{k-j}T^*M)$ (for $1\leq j\leq k$).

\begin{proposition}If $\theta\in \Gamma(A^*\otimes( \wedge^{k-1} T^*M))$ is $\rho$-compatible, then the following statements are true.
	\begin{itemize}
		\item[(1)] For all $u\in A$, we have
	\begin{equation}\label{Eqt:Bthetawithuproperty}
	\iota_u (B\theta)	= \iota_{\rho(u)} (B\theta)+\iota_u \theta;
	\end{equation}
\item[(2)] For all $u_1,\cdots, u_j\in A_x$ ($1\leq j\leq k$), $X_{j+1},\cdots, X_k\in  T_xM $, we have
 \begin{eqnarray}\nonumber 
	&&(\Drhostarpush ^{j-1} \theta)  (u_1,\cdots,u_j,X_{j+1},\cdots,X_{k})\\\label{Eqn:Drhostarjminusone}
	&=& j! \theta(u_1,\rho(u_2), \rho(u_3),\cdots,\rho(u_j),X_{j+1},\cdots,X_{k}). \end{eqnarray}

\item[(3)] For all $u_1+X_1,\cdots, u_k+X_k\in  A\oplus TM $, we have
\begin{eqnarray}\nonumber 
	&&  (B\theta)(u_1+X_1,\cdots, u_k+X_k)\\\label{Eqt:Bthetatemp3}
	&=& \sum_{j=1}^{k}\sum_{\sigma\in \mathrm{Sh}(j,k-j)} \sign(\sigma)~ \theta\bigl(u_{\sigma_1},\rho(u_{\sigma_2}),\cdots,\rho(u_{\sigma_j}),X_{\sigma_{j+1}},\cdots,X_{\sigma_k}\bigr)\\\nonumber
	&=& \theta(u_1,\rho(u_2)+X_2,\rho(u_3)+X_3,\cdots,\rho(u_k)+X_k)\\\nonumber
	&& \quad + \theta(X_1,u_2,\rho(u_3)+X_3, \rho(u_4)+X_4,\cdots,\rho(u_k)+X_k)\\\nonumber &&\quad\quad+
	\theta(X_1, X_2,u_3,\rho(u_4)+X_4,\rho(u_5)+X_5\cdots,\rho(u_k)+X_k)
	\\\label{Eqt:Bthetatemp4}
	&&\quad\quad\quad\quad+\cdots+\theta(X_1, X_2,\cdots, X_{k-1},u_k).
	\end{eqnarray}
\end{itemize}
\end{proposition}
 \begin{proof}
		By the definition of $B\theta$ in Equation \eqref{Btheta}, we do direct  computations:\begin{eqnarray*}
\iota_u (B\theta)&=& \iota_u\theta+\iota_u \bigl( \frac{1}{2} \Drhostarpush  \theta+\frac{1}{3!} \Drhostarpush ^2 \theta+\cdots+\frac{1}{k!}  \Drhostarpush ^{k-1} \theta\bigr)\\
&=&\iota_u\theta+\iota_{\rho(u)} \bigl( \theta+\frac{1}{2} \Drhostarpush  \theta+\frac{1}{3!} \Drhostarpush ^2 \theta+\cdots+\frac{1}{(k-1)!}  \Drhostarpush ^{k-2} \theta \bigr).
	\end{eqnarray*}
In the last step,  we repetitively used Equation \eqref{eq}. The above terms turn to the desired right hand side of Equation \eqref{Eqt:Bthetawithuproperty}.

By the definition of $\Drhostarpush$,   we have
\begin{eqnarray*}
	&&(\Drhostarpush ^{j-1} \theta)  (u_1,\cdots,u_j,X_{j+1},\cdots,X_{k})\\
	&=&(j-1)!
	\sum_{i=1}^j
	(-1)^{i+1} \theta(u_i, \rho(u_1),\cdots,\widehat{\rho(u_i)},\cdots,\rho(u_j),X_{j+1},\cdots,X_{k}),
\end{eqnarray*}
where $u_i\in A_x$ and $X_i\in T_x M$.  Then by the $\rho$-compatibility condition of $\theta$, the term 
\begin{eqnarray*}
	&&(-1)^{i+1} \theta(u_i, \rho(u_1),\cdots,\widehat{\rho(u_i)},\cdots,\rho(u_j),X_{j+1},\cdots,X_{k})\\
	&=&(-1)^{i} \theta(u_1,\rho(u_i), \rho(u_2),\cdots,\widehat{\rho(u_i)},\cdots,\rho(u_j),X_{j+1},\cdots,X_{k})
	\\
	&=& \theta(u_1,\rho(u_2), \rho(u_3),\cdots,\rho(u_j),X_{j+1},\cdots,X_{k}) .
\end{eqnarray*}
This proves Equation \eqref{Eqn:Drhostarjminusone} and we then get \begin{eqnarray*}\nonumber 
	(B\theta)(u_1,\cdots,u_j,X_{j+1},\cdots,X_{k}) 
	&=& \frac{1}{j!} (\Drhostarpush ^{j-1} \theta)  (u_1,\cdots,u_j,X_{j+1},\cdots,X_{k})\\
	&=&  \theta(u_1,\rho(u_2), \rho(u_3),\cdots,\rho(u_j),X_{j+1},\cdots,X_{k}).
\end{eqnarray*} From these relations  it is immediate to   derive Equations \eqref{Eqt:Bthetatemp3} and \eqref{Eqt:Bthetatemp4}.
\end{proof}

 \subsection{From multiplicative forms to $\rho$-compatible tensors}

 Let $k\geqslant 1$ be an integer and $\Theta\in \OmegakmultG{k}$ be multiplicative. We
 first notice that   $M$ is isotropic with respect $\Theta$, i.e. 
 $\Theta_x(X_1,\cdots, X_k)=0$ for $X_i\in T_x M$. In fact, this can be easily seen   by  applying Equation ~\eqref{multform} to the tangent vectors $(X_1,X_1),\cdots, (X_k,X_k)$ (at $(x,x)\in \Gpd ^{(2)}$) and noticing that $X_i=m_*(X_i,X_i)$. We then know that $ \Theta|_M$ has no summand in the space $\Gamma(\wedge^{k} T^*M)$, or  
 \[\Theta|_M\in  \bigoplus_{i=1}^k\Gamma(\wedge^iA^*\otimes (\wedge^{k-i} T^*M)).\]

 We now claim that     all  the $\Gamma(\wedge^i A^*\otimes (\wedge^{k-i} T^*M))$-summands of $\Theta|_M$ ($ i=2,3,\cdots, k$) are decided by  $\theta\in \Gamma(A^*\otimes (\wedge^{k-1} T^*M))$ which is the $\Gamma(A^*\otimes (\wedge^{k-1} T^*M))$-component of 
 $  \Theta|_M $.  The element $\theta$ will be called the \textbf{leading term} of $\Theta\in \OmegakmultG{k}$.


 \begin{proposition}\label{Prop:ThetaalongM}
 	Let   $\Theta\in \OmegakmultG{k}$  be a multiplicative $k$-form  on $\Gpd $ ($k\geqslant 1$). Then we have the following facts.
 	\begin{itemize}
 		\item[(1)] Its leading term $\theta:=\mathrm{pr}_{\Gamma( A^*\otimes (\wedge^{k-1} T^*M))} \Theta|_M$ is a $\rho$-compatible $(0,k)$-tensor;
 		\item[(2)] The restriction of $\Theta$ on $M$ is given by \begin{eqnarray}\label{Btheta2}
 		\Theta|_M= B\theta.
 		\end{eqnarray}(See Equation \eqref{Btheta} for the definition of $B\theta$.)
 	\end{itemize}
  \end{proposition}

   In fact, this proposition   is covered by  Crainic-Salazar-Struchiner's result  \cite[Proposition 4.1]{C} (see also Proposition \ref{4.1}, Theorem \ref{main2} and Corollary \ref{Cor:fromThetatocharpair} in the sequel), which is highly nontrivial. However,   we shall give a direct  proof   using some basic techniques. Before that, let us   state the following fact.

 \begin{lemma}\label{lifting}
 	Given $k\geqslant 1$ and $\Theta\in \OmegakmultG{k}$, for any $u\in \Gamma(A)$, the term $\iota_{u-\rho(u)} (\Theta|_M) \in \Omega^{k-1}(\Gpd)|_M$ has no summand in the space $\Gamma(\wedge^i A^*\otimes (\wedge^{k-1-i} T^*M))$, $i=1,2,\cdots, k-1$. 
 	\end{lemma}

  \begin{proof} Consider the particular point $(g,s(g))\in \Gpd ^{(2)}$ and 
 	the following tangent vectors at $(g,s(g))$:
 	\[\bigl(0_g,(u-\rho(u))_{s(g)}\bigr),~(X_1,s_* X_1),~\cdots, ~(X_{k-1}, s_* X_{k-1}).\] 
 	Here $X_i\in T_g\Gpd$.  Notice the following facts:
 	\[m_*\bigl(0_g, (u-\rho(u))_{s(g)}\bigr)=L_{g*} (u-\rho(u))_{s(g)}=\overleftarrow{u}_g,\qquad m_*(X_i,s_* X_i)=X_i.\]
 	Substituting  these vectors to Equation \eqref{multform},  we obtain
 	\[\Theta_g(\overleftarrow{u}_g,X_1,\cdots, X_{k-1})=\Theta_{s(g)}((u-\rho(u))_{s(g)},s_* X_1,\cdots,s_* X_{k-1}).\]
 	Especially, if $g=x=s(g)\in M$, the above equation implies the desired assertion.\end{proof} 
 
 This lemma is indeed saying that $\iota_{u-\rho(u)} (\Theta|_M)$ is a $(k-1)$-form on $M$. Let us denote it by  $\omega=\iota_{u-\rho(u)} (\Theta|_M)\in \Omega^{k-1}(M)$, and then  the term $\iota_{\overleftarrow{u}} \Theta\in \Omega^{k-1}(\Gpd)$ is a left invariant $(k-1)$-form, and the following identity holds 
 \begin{equation}
 \label{Eqt:iotauTheta1}
 \iota_{\overleftarrow{u}} \Theta=s^*\omega.
 \end{equation}

   {Similar to the proof above, by considering the tangent vectors 
 	\[ (u _{t(g)},0_g),~(t_* X_1, X_1),~\cdots, ~( t_* X_{k-1},X_{k-1}) \]
 	at the point $(t(g),g)\in \Gpd ^{(2)}$, we get
 	$$\Theta_g(\overrightarrow{u}_g,X_1,\cdots, X_{k-1})=\Theta_{t(g)}( u_{t(g)},t_* X_1,\cdots,t_* X_{k-1}).$$ 
 	Then by the fact that $M$ is isotropic with respect $\Theta$, the right hand side of the above equation becomes
 	$$
 	\Theta_{t(g)}( (u-\rho(u))_{t(g)},t_* X_1,\cdots,t_* X_{k-1})
 	=\omega_{t(g)}(t_* X_1,\cdots,t_* X_{k-1}).
 	$$ }
  This shows that the $(k-1)$-form $\iota_{\overrightarrow{u}} \Theta$ is  right invariant, and the following identity holds 
 \begin{equation*}
 \label{Eqt:iotauTheta2}
 \iota_{\overrightarrow{u}} \Theta=t^*\omega.
 \end{equation*}

 We now turn to the  proof of Proposition \ref{Prop:ThetaalongM}.
 \begin{proof}
 	First, since $M$ is isotropic with respect to $\Theta$, it has no summand in $\wedge^k T^*M$. Thus $\Theta|_M$ is expressed as
 	\[\Theta|_M=\theta^{1,k-1}+\theta^{2,k-2}+\cdots+\theta^{k,0},\]
 	where $\theta^{i,k-i}\in \Gamma(\wedge^i A^*\otimes (\wedge^{k-i} T^*M))$. Next, Lemma \ref{lifting}  tells us that $\iota_{u-\rho(u)} \Theta|_M \in \Gamma(\wedge^{k-1} T^*M)$, which signifies that
 	\begin{eqnarray}\label{induction}
 	\iota_{\rho(u)} \theta^{i,k-i}=\iota_{u}\theta^{i+1,k-i-1},\qquad i=1,2,\cdots,k-1.
 	\end{eqnarray}
 	By Lemma \ref{lifting}, we also have $\iota_{u-\rho(u)}\iota_v\Theta|_M=0$ for $u,v\in \Gamma(A)$. This implies that $\iota_u\iota_v \Theta|_M=\iota_{\rho(u)}\iota_v \Theta|_M$. Thus we obtain
 	\[\iota_u\iota_{\rho(v)}\theta^{1,k-1}=-\mathrm{pr}_{\Gamma(\wedge^{k-2} T^*M)}(\iota_v\iota_u \Theta|_M)=\mathrm{pr}_{\Gamma(\wedge^{k-2} T^*M)}(\iota_u\iota_v\Theta|_M)=-\iota_{v}\iota_{\rho(u)} \theta^{1,k-1}.\]
 	That is, $\theta^{1,k-1}$ is $\rho$-compatible.
 	Then, we claim that
 	\[ \theta^{i,k-i}=\frac{1}{i!} \Drhostarpush ^{i-1} \theta,\qquad \mbox{where~}~ \theta:=\theta^{1,k-1},\] which yields the desired Equation \eqref{Btheta2}.
 	The proof is by induction for $i$ \quad \textemdash \quad   It is true for $i=1$. Assume that the assertion holds for $i$. By Equation ~\eqref{induction} and Lemma \ref{formula}, we find that 
 	\begin{eqnarray*}
 		\iota_u \theta^{i+1,k-i-1}&=&\iota_{\rho(u)} \theta^{i,k-i}
 		= \frac{1}{i!} \iota_{\rho(u)}\Drhostarpush ^{i-1}\theta = \frac{1}{i!}\frac{1}{i+1} \iota_{u} \Drhostarpush ^i \theta= \frac{1}{(i+1)!} \iota_u \Drhostarpush ^i \theta,
 	\end{eqnarray*}
 	for  all $u\in \Gamma(A)$. Then we get $\theta^{i+1,k-i-1}= \frac{1}{(i+1)!} \Drhostarpush ^i \theta$.
 \end{proof}

 \begin{remark}
 	By Equations \eqref{Eqn:Drhostarjminusone}	and   \eqref{Btheta2}, we have
 	\begin{eqnarray*}\nonumber
 	\Theta_x(u_1,\cdots,u_j,X_{j+1},\cdots,X_{k}) 
 	&=& \frac{1}{j!} (\Drhostarpush ^{j-1} \theta)  (u_1,\cdots,u_j,X_{j+1},\cdots,X_{k})\\\label{Eqt:jthtermofTheta}
 	&=& \theta(u_1,\rho(u_2), \rho(u_3),\cdots,\rho(u_j),X_{j+1},\cdots,X_{k}).
 	\end{eqnarray*}
 	In fact, this relation  already appeared in \cite[Lemma 4.2]{C} (or \cite[Remark 2]{CMS1}). So  one can  treat Equation \eqref{Btheta2} as a compact form of their result. 	 
 \end{remark}  
  {
 \begin{remark}\label{alpha}
 For $\Theta\in \Omega^k_{\mult}(\Gpd)$, by definition of the source map $s$ of $T^*\Gpd\rightrightarrows A^*$ as in \eqref{stco} and Equation \eqref{Eqt:iotauTheta1},  we see that $s(\Theta_g)$ only depends on $\Theta_{s(g)}$. So by $(2)$ of Proposition \ref{Prop:ThetaalongM} and Formula \eqref{Btheta}, we have 
 \[s(\Theta)=s(\Theta|_M)=s(B\theta)=\frac{1}{k!}\Drhostarpush^{k-1}\theta\in \Gamma(\wedge^k A^*).\]
 Similarly, we have $t(\Theta)=\frac{1}{k!}\Drhostarpush^{k-1}\theta$. In particular, we have $s(\Theta)=t(\Theta)=\theta\in \Gamma(A^*)$ for $\Theta\in \Omega^1_{\mult}(\Gpd)$.
 \end{remark}}

\subsection{Groupoid $(0,k)$-characteristic pairs}

\subsubsection{Characterization of multiplicative forms}
We need a result due to Crainic,  Salazar, and Struchiner \cite{C} which we now recall. They used it to prove the correspondence of multiplicative forms and Spencer operators.  

 In \cite{C}, a special type of pairs $(e,\theta)$ was introduced, where   
 \begin{itemize}
 	\item[(1)] $e$ is a   $1$-cocycle of the jet groupoid $\jet \Gpd$ valued in $\wedge^k T^*M$, i.e. 
 	$e\in Z^1(\jet \Gpd , \wedge^k T^*M)$, and 
 	\item[(2)] $\theta\in \Gamma(A^*\otimes (\wedge^{k-1} T^*M)) $ is a $\rho$-compatible $(0,k)$-tensor.
 \end{itemize} Here $k\geqslant 1$ is an integer, and 
moreover, 
 $c$ and $\theta$  are subject to the following two conditions:\begin{itemize}
 	\item[(1)] 
 \begin{eqnarray}
 	&&\nonumber e[b'_g](\Ad_{[b'_g]} X_1,\cdots, \Ad_{[b'_g]} X_k)-e[b_g](\Ad_{[b_g]} X_1,\cdots, \Ad_{[b_g]} X_k)\\\label{C1} &=&\sum_{i=1}^k(-1)^{i+1}\theta\big((b'_g\ominus b_g)X_i,\Ad_{[b_g]}X_1,\cdots,\Ad_{[b_g]}X_{i-1},\Ad_{[b'_g]}X_{i+1},\cdots,\Ad_{[b'_g]}X_{k}\big),
 \end{eqnarray} 
 \item[(2)] 
 \begin{eqnarray}\nonumber && \mbox{ ~~~~~~~~~~~~~    }
 	\theta\big(\Ad_{[b_g]} u,\Ad_{[b_g]} X_1,\cdots,\Ad_{[b_g]} X_{k-1}\big)-\theta(u,X_1,\cdots, X_{k-1})\qquad \mbox{\qquad ~\qquad ~\qquad ~\qquad ~~\qquad~\qquad ~\qquad ~~\qquad ~\qquad ~\qquad ~\qquad ~\qquad~\qquad     }\qquad \\\label{C3} &=&e[b_g]\big(\Ad_{[b_g]} \rho(u),\Ad_{[b_g]}X_1,\cdots,\Ad_{[b_g]} X_{k-1}\big),
 \end{eqnarray} 
 for all   bisections $b_g$, $b'_g$ passing through $g$, $u\in A_{s(g)}$, and $X_1,\cdots,X_k\in T_{s(g)} M$.    Here $\Ad$ is the natural adjoint action of $\jet \Gpd $ on  $TM$ and
 \[b'_g\ominus b_g:=R_{g^{-1}*}{\littlecirc} (b'_{g*}-b_{g*}):~\ ~ T_{s(g)} M\to A_{t(g)}.\]
 
\end{itemize}

For convenience, we give such pairs a notion.
\begin{definition}\label{Def:groupoid0kcharpair} A pair $(e,\theta)$ as described above is called a  \textbf{$(0,k)$-characteristic pair}  on the Lie groupoid $\Gpd$. 

\end{definition}


We shall reinterpret the compatibility conditions of $c$ and $\theta$, namely Equations \eqref{C1} and \eqref{C3}, from a different aspect (see Proposition \ref{co cp}).

A multiplicative function  ($0$-form) on $\Gpd$ is exactly a Lie groupoid $1$-cocycle valued in the trivial  $\Gpd$-module $M\times \mathbb{R}$. In other words, we have $\OmegakmultG{0}=Z^1(\Gpd,M\times \mathbb{R})$. We wish to see what data characterize the space $\OmegakmultG{k}$ for $k\geqslant 1$. The result we state below gives an answer. Indeed, it is  a particular instance  of \cite[Proposition 4.1]{C}. 

\begin{proposition}\label{4.1}
	There is a one-to-one correspondence between multiplicative $k$-forms $\Theta\in \OmegakmultG{k}$ and $(0,k)$-characteristic pairs $(e,\theta)$ on $\Gpd$ such that for all bisections $b_g$ passing through a point $g\in \Gpd$, one has
\begin{eqnarray}\nonumber
	&&\Theta_g\bigl(R_{[b_g]}(u_1+X_1) ,\cdots, R_{[b_g]}(u_k+X_k)  \bigr)\\\label{Eqt:Thetaoriginal} &=&e[b_g](X_1,\cdots,X_k)+\sum_{j=1}^{k}\sum_{\sigma\in \mathrm{Sh}(j,k-j)} \sign(\sigma)~ \theta\bigl(u_{\sigma_1},\rho(u_{\sigma_2}),\cdots,\rho(u_{\sigma_j}),X_{\sigma_{j+1}},\cdots,X_{\sigma_k}\bigr),
\end{eqnarray}
where $u_i+X_i\in A_{t(g)}\oplus T_{t(g)}M $ ($i=1,\cdots, k$). Here we used the identification  $$T_g \Gpd = R_{[b_g]}(A_{t(g)}\oplus T_{t(g)}M) $$
and the notation $\mathrm{Sh}(p,q)$ stands for the set of $(p,q)$-shuffles. 
 \end{proposition}  
\begin{remark}\label{Lie group case}
	From this one-to-one correspondence,  we see that all  multiplicative $k$-forms with  $k\geqslant  \dim M + 2$ on $\Gpd$ are trivial. In particular, if the Lie groupoid happens to be a Lie group $G$ (i.e. $M$ is a single point), then all multiplicative $k$-forms on $G$ with  $k\geqslant  2$  are trivial. 
	

	For this reason, \textit{we will always assume that $1\leqslant k\leqslant \dim M+1$ in the sequel}.
\end{remark}

Upon the correspondence of this proposition, we   refer to $(e,\theta)$ as the 
 $(0,k)$-characteristic pair \textit{of} the multiplicative $k$-form $\Theta\in \OmegakmultG{k}$. In addition, we can reformulate the relation between $(e,\theta)$ and $\Theta$ in a compact manner:  
\begin{theorem}
	\label{main2}
	Let   $\Theta\in \Omega^k (\Gpd )$ be multiplicative and $(e,\theta)$ its corresponding $(0,k)$-characteristic pair. Then for any       bisection $b_g$ passing through $g\in \Gpd$, we have
	\begin{eqnarray}\label{formula theta}
		\Theta_g=R_{[b_g^{-1}]}^*\big(e[b_g]+B\theta_{t(g)}\big).\end{eqnarray}
  In this formula, $B\theta$ is defined by Equation ~\eqref{Btheta}  and $R^*_{[b_g^{-1}]}:~T^*_{t(g)}\Gpd\to T^*_g \mathcal{G} $ is the dual map of the right translation $R_{[b_g^{-1}]}$ $:~   T_g \mathcal{G} \to T_{t(g)}\Gpd$.
\end{theorem}
\begin{proof}  Equation \eqref{formula theta} is just a variation of     Equation \eqref{Eqt:Thetaoriginal}, a result by Crainic, Salazar, and  Struchiner \cite{C}.
{In fact, by Equation \eqref{Eqt:Bthetatemp3},
we can reformulate Equation \eqref{Eqt:Thetaoriginal}:}
	\begin{eqnarray*}
		&&\Theta_g\bigl(R_{[b_g]}(u_1+X_1) ,\cdots, R_{[b_g]}(u_k+X_k)  \bigr)\\ &=& e[b_g](X_1,\cdots,X_k)+B\theta( u_1+X_1,\cdots, u_k+X_k),
	\end{eqnarray*}
which proves the desired Equation \eqref{formula theta}.  
\end{proof}

\begin{corollary}\label{Cor:fromThetatocharpair} Given a multiplicative $k$-form $\Theta\in \OmegakmultG{k}$, its corresponding $(0,k)$-characteristic pair $(e,\theta)$ is determined by the following methods:
	\begin{itemize}
		\item[(1)]  As a   map $ {e}:$  $\jet\Gpd \rightarrow \wedge^{k } T^*M$, one has  
		\begin{eqnarray}\label{tildec}
			 {e}[b]=R_b^{!*} \Theta,
		\end{eqnarray}
		for all bisections $b: M\to   \Gpd$      of $\Gpd$. Here we treat $[b]: M\to  \jet\Gpd$ as a   section of the fibre bundle $\jet\Gpd \stackrel{s}{\to}M$,  and $R_b^{!}$  $:M\to \Gpd$ is defined by Equation \eqref{Eqt:bt}. 
		\item[(2)] As a section of the vector bundle $A^*\otimes (\wedge^{k-1} T^*M)$, we have \begin{equation}\label{Eqt:FromThetatotheta}
			\theta =\mathrm{pr}_{\Gamma(A^*\otimes (\wedge^{k-1} T^*M))} \Theta|_M \,.
		\end{equation}
		Moreover,   $\Theta|_M$ is identically $B\theta$.
	\end{itemize}

\end{corollary}
\begin{proof}In Equation \eqref{formula theta}, if we take  the trivial bisection $b=i_M: M\hookrightarrow \Gpd$,   then $e[b]=0$ and we get \begin{equation*}
	\label{Eqt:ThetaMBtheta} \Theta|_M=B\theta
	\end{equation*}  and hence Equation \eqref{Eqt:FromThetatotheta}. 
	Also, from Equation \eqref{formula theta}, we have
	$$
	e[b_g]=  R_{[b_g ]}^*\Theta_g -B\theta_{t(g)}.
	$$
	Taking  projection of both sides to $\wedge^k T^*_{t(g)}M$,   we obtain Equation \eqref{tildec}.
\end{proof}  

\begin{example}
{For a multiplicative $1$-form $\Theta\in \Omega^1_{\mult}(\Gpd)$, the corresponding  characteristic pair $(e,\theta)$ with $e\in Z^1(\jet \Gpd, T^*M)$ and $\theta\in \Gamma(A^*)$ is determined as follows:
\[\langle e[b_g], X_{t(g)}\rangle=\langle \Theta_g, b_{g*} \phi_{b_g*}^{-1} X_{t(g)}\rangle, \qquad \theta=\Theta|_M,\qquad \forall X_{t(g)}\in T_{t(g)} M.\]
We then have the relation $\Theta_g=R_{[b_g^{-1}]}^*(e[b_g]+\theta_{t(g)})$ for all bisections $b_g$ passing through $g\in \Gpd$.}
\end{example}

\begin{example}\label{Exm:dfchar}
	Given a multiplicative function  $f\in \OmegakmultG{0}$, we have $df\in \OmegakmultG{1}$. The corresponding  $(0,1)$-characteristic pair $(e,\theta)$ of $df$ is explained below. First, as $f$ can be regarded as a Lie groupoid $1$-cocycle $\Gpd\to M\times \mathbb{R}$, its infinitesimal $\hat{f}\colon A\to M\times \mathbb{R}$ is a Lie algebroid $1$-cocycle. Indeed, we can treat $\hat{f}$ as in $\Gamma(A^*)$ which reads
	$$
	\hat{f}(u)=- \overrightarrow{u}(f)|_M,\quad \forall u\in \Gamma(A).
	$$
	\begin{itemize}
		\item[(1)]The element $e\in Z^1(\jet  \Gpd, T^*M)$ is  determined by  
		\begin{eqnarray}\label{Eqt:eofdf}
		 {e}[b]=R_b^{!*}(df)=d(f\circ R_b^{!})\in \Omega^1(M)
		\end{eqnarray}
		for all bisections $b: M\to   \Gpd$      of $\Gpd$. Here we treat $[b]: M\to  \jet\Gpd$ as a   section of the fibre bundle $\jet\Gpd \stackrel{s}{\to}M$.
		
		\item[(2)]The element $\theta\in \Gamma(A^*)$ coincides with $(-\hat{f})$.
	\end{itemize}
	
\end{example}
{
\begin{example}\label{form on group}
Following Remark \ref{Lie group case}, on a Lie group $G$ with its Lie algebra $\mathfrak{g}:=T_eG$,    only   multiplicative $1$-forms could be nontrivial, and  $(0,1)$-characteristic pairs on $G$ are of the form $(0,\theta)$, where $\theta\in \mathfrak{g}^*$ is $G$-invariant, namely, $\Ad^\vee_g \theta=\theta$ for all $g\in G$. Any $\Theta\in \OmegakmultG{1}$ stems from such a $\theta$. In fact, from  $\Theta$ one may take $\theta:=\Theta|_e\in \mathfrak{g}^*$ which is necessarily $G$-invariant, and  from $\theta$ one can recover $\Theta$ as $\Theta(g)=R_{g^{-1}}^*\theta$ for all $g\in G$.
\end{example}}

\begin{lemma}\label{Lem:exactcharpairs}
	Let $\gamma\in \Omega^k(M)$ be a $k$-form on the base manifold $M$. It corresponds to an exact $k$-form on  the Lie groupoid $\Gpd $: \begin{equation*}\label{Eqt:exacte}
		J(\gamma)=s^*\gamma-t^*\gamma\in \Omega^k(\Gpd )\end{equation*} which is  multiplicative    and the corresponding $(0,k)$-characteristic pair is given by $$\bigl(e=d_{\jet \Gpd}(\gamma)  ,\theta=-\Drhostarpush  \gamma\bigr) ,$$ where $d_{\jet \Gpd}:$ $\Omega^k(M)$ $\to$  $Z^1(\jet \Gpd, \wedge^k T^*M)$ is the differential of $\jet \Gpd $ with respect to its  coadjoint action on $\wedge^k T^*M$. 	
\end{lemma}	\begin{proof}
In fact, by Corollary \ref{Cor:fromThetatocharpair} we have \[e[b]= R_b^{!*} (s^*\gamma-t^*\gamma)=  \Ad^\vee_{[b ]} \gamma-\gamma=( d_{\jet \Gpd}\gamma)[b ] \]
	for all bisections $b\in \mathrm{Bis}(\Gpd )$, and  
	\begin{eqnarray*}
		\theta(u,X_1,\cdots,X_{k-1})&=&(s^*\gamma-t^*\gamma)|_M (u,X_1,\cdots,X_{k-1})\\ &=&-\gamma(\rho(u),X_1,\cdots,X_{k-1})\\ &=&-(\Drhostarpush  \gamma)(u,X_1,\cdots,X_{k-1}),
	\end{eqnarray*}
	for all $u\in \Gamma(A)$ and $X_i\in \mathfrak{X}^1(M)$.
\end{proof}

\subsubsection{Another description of groupoid $(0,k)$-characteristic pairs}

\begin{proposition}\label{co cp}
	Let   $e\in Z^1(\jet \Gpd , \wedge^k T^*M)$  be a $1$-cocycle and $ \theta\in \Gamma(A^*\otimes (\wedge^{k-1} T^*M))$ be $\rho$-compatible. They form a $(0,k)$-characteristic pair $(e,\theta)$ on $\Gpd$ if and only if 
\begin{eqnarray}
\label{co1}&&e[h]=R_{[h]}^*(B\theta)-B\theta, \qquad \forall [h]\in \heat,\\
\label{coo3}\mbox{ and }&& \Drhostarpush\circ e =-d_{\jet \Gpd} \theta.\end{eqnarray}
Here    $d_{\jet \Gpd}: \Gamma(A^*\otimes (\wedge^{k-1} T^*M))\to C^1(\jet \Gpd,A^*\otimes (\wedge^{k-1} T^*M))$ is the differential of $\jet \Gpd$ with coefficient in $A^*\otimes (\wedge^{k-1} T^*M)$.\end{proposition}

  One notes that  the right hand side of Equation  ~\eqref{co1} lands in $\Gamma(\wedge^k (T^*M\oplus A^*))$. So, before we give the proof, we need to explain why it only has the $\Gamma(\wedge^k   T^*M) $-component,   or
\begin{equation}\label{Eqt:temp2temp}\iota_{v}(R_{[h]}^*(B\theta)-B\theta)
	=0,\quad \forall v\in A_x.\end{equation}
In fact, one can examine that, for all $u_i+X_i\in  A_x\oplus T_xM $,
\begin{eqnarray*}
	&& \bigl(\iota_{v}(R_{[h]}^*(B\theta)\bigr)(u_1+X_1,\cdots,u_{k-1}+X_{k-1})
	\\
	&=& (B\theta)\bigl(R_{[h]}v,R_{[h]}(u_1+X_1),\cdots,R_{[h]}(u_{k-1}+X_{k-1})\bigr),\\
	&=& (B\theta)\bigl( v,  u_1+H(\id+\rho H)^{-1}  X_1 + (\id+\rho H)^{-1}  X_1  ,\\&&\qquad \cdots,u_{k-1}+H(\id+\rho H)^{-1}  X_{k-1} + (\id+\rho H)^{-1}  X_{k-1}  \bigr)\qquad \mbox{(by Equation \eqref{Eqn:Rhx})}\\
	&=&  \theta\bigl( v,  \rho(u_1)+ X_1 ,\cdots, \rho(u_{k-1})+ X_{k-1}  \bigr), \qquad \mbox{ by Equation \eqref{Eqt:Bthetatemp4}.} 
\end{eqnarray*}
The last line is independent of $[h]$  and hence Equation \eqref{Eqt:temp2temp} is valid.

\begin{proof}[Proof of Proposition \ref{co cp}] We first show that Equation \eqref{coo3} is equivalent to Equation \eqref{C3}. In fact, 	Equation \eqref{coo3} reads $$\Drhostarpush  ( e  [b_g] )=\theta-\Ad^\vee_{[b_g]} \theta,\quad \forall [b_g]\in \jet_g \Gpd  .$$ When applied to   arguments $u\in A_{t(g)}$ and $X_i\in T_{t(g)} M$, the above equality becomes
	\begin{eqnarray}\nonumber 
		e[b_g](\rho(u),X_1,\cdots,X_{k-1})=\theta(u,X_1,\cdots,X_{k-1})-\theta\big(
		\Ad_{[b^{-1}_g]} \rho(u), \Ad_{[b^{-1}_g]}X_1,\cdots, \Ad_{[b^{-1}_g]} X_{k-1}\big)
	\end{eqnarray}
	which is clearly the same relation as  Equation  \eqref{C3}.

We then unravel Equation \eqref{co1} and it suffices to consider    an arbitrary $[h]=\id+H\in \heat_x$ (where  $H \in {\underline{\mathrm{Hom}}(T_xM,A_x)}$) and all $X_1,\cdots, X_k\in T_xM$. Substituting them into Equation \eqref{co1}, we get 
\begin{eqnarray}\nonumber
	&&e[h](X_1,X_2,\cdots,X_k)\\\nonumber
	&=&(R_{[h]}^*(B\theta)-B\theta)(X_1,\cdots, X_k)=(B\theta)(R_{[h]}X_1,\cdots, R_{[h]}X_k)\\\nonumber
	&=&(B\theta)(H(\id+\rho H)^{-1}  X_1 + (\id+\rho H)^{-1}  X_1 ,\cdots, H(\id+\rho H)^{-1}  X_k + (\id+\rho H)^{-1}  X_k )\\\nonumber 
		&=&\theta(H(\id+\rho H)^{-1}X_1, X_2,\cdots, X_k)\\\nonumber &&\quad +\theta((\id+\rho H)^{-1}X_1,H(\id+\rho H)^{-1}X_2,  X_3,\cdots, X_k)\\\label{CtoCL}
	&&\quad\quad+\cdots+\theta((\id+\rho H)^{-1}X_1,  \cdots, (\id+\rho H)^{-1}X_{k-1},H(\id+\rho H)^{-1}X_k). 
\end{eqnarray}
	Here in the last step, we used Equation \eqref{Eqt:Bthetatemp4}. 	
	So to finish the proof, it suffices to show the equivalence of Equations $\eqref{C1}\Leftrightarrow \eqref{CtoCL}$ (which $ \Leftrightarrow \eqref{co1} $): 
	 
\begin{itemize}
	 	\item[(1)]  $\eqref{C1}\Rightarrow \eqref{CtoCL}$: Consider the particular point $g=x\in M$. Take two bisections $b_x$ and $b_x'$ passing through $x$. where $b_x$ is the trivial identity section. So we know that $[b_x]=\id+0_x$  is the unit of the group  $\heat_x$, and we suppose that $[b'_x]=[h]=\id+H$ for some $H \in {\underline{\mathrm{Hom}}(T_xM,A_x)}$.    Then we have $\Ad_{[h]} X_i=(\id+\rho H) X_i$ (for $X_i\in T_xM$),   $e[b_x]=0$ (since $e$ is a $1$-cocycle), and \[b'_x\ominus b_x=R_{x^{-1}*}{\littlecirc} (b'_{x*}-b_{x*})= H, ~\ \mbox{as a map }~ T_{x} M\to A_{x}.\]
			Therefore, in this particular case, from Equation \eqref{C1} we have
\begin{eqnarray*}\nonumber &&e[h]((\id+\rho H)X_1, (\id+\rho H)X_2,\cdots,(\id+\rho H)X_k)\\\nonumber 
	&=&\theta(H(X_1),(\id+\rho H)X_2,\cdots,(\id+\rho H)X_k)\\\nonumber &&\quad +\theta(X_1,H(X_2),(\id+\rho H)X_3,\cdots,(\id+\rho H)X_k )\\\label{Eqt:CtoCLvariation}&&\quad\quad+\cdots+\theta(X_1,X_2,\cdots, X_{k-1},H(X_k)),
\end{eqnarray*} which is just a variation of Equation \eqref{CtoCL}.
	 
\item[(2)]    $\eqref{CtoCL}\Rightarrow \eqref{C1}$: 
For two bisections $b_g$ and $b'_g$ passing through $g$, there  exists some $[h]=\id+H\in \heat_{t(g)}$ such that $[b'_g]=[h]\cdot [b_g]$. Therefore, the left hand side of Equation \eqref{C1} becomes
\begin{eqnarray*}
&&e[b'_g](\Ad_{[b'_g]} X_1,\cdots, \Ad_{[b'_g]} X_k)-e[b_g](\Ad_{[b_g]} X_1,\cdots, \Ad_{[b_g]} X_k)\\
	&=& e([h][b_g])(\Ad_{[h]}\Ad_{[b_g]} X_1,\cdots, \Ad_{[h]}\Ad_{[b_g]} X_k)-e[b_g](\Ad_{[b_g]} X_1,\cdots, \Ad_{[b_g]} X_k).
\end{eqnarray*}
Using the cocycle condition $e([h][b_g])=e[h]+\Ad^\vee_{[h]} e[b_g]$ and $\Ad_{[h]}X=(\id+\rho H)X$ (for $X\in T_{t(g)}M$), the above computation continues: 
\begin{eqnarray*}
 	\mbox{LHS of Equation \eqref{C1}} 	&=&  e[h]((\id+\rho H)\Ad_{[b_g]} X_1,\cdots,(\id+\rho H)\Ad_{[b_g]} X_k ) .
	\end{eqnarray*}
On the other hand, we have \[b'_g\ominus b_g =R_{g^{-1}*}{\littlecirc} (b'_{g*}-b_{g*})=H\circ \Ad_{[b_g]},~\ \mbox{as a map }~ T_{s(g)} M\to A_{t(g)}.\]
So we get  
\begin{eqnarray*}
&&\mbox{RHS of Equation \eqref{C1}}\\&=&	\theta(H\Ad_{[b_g]}  X_1 ,(\id+\rho H) \Ad_{[b_g]}X_2,\cdots,(\id+\rho H) \Ad_{[b_g]}X_k)\\  &&\quad +\theta(\Ad_{[b_g]}X_1,H\Ad_{[b_g]} X_2 ,(\id+\rho H)\Ad_{[b_g]}X_3,\cdots,(\id+\rho H)\Ad_{[b_g]}X_k )\\ &&\quad\quad+\cdots+\theta(\Ad_{[b_g]}X_1,\Ad_{[b_g]}X_2,\cdots, \Ad_{[b_g]}X_{k-1},H\Ad_{[b_g]} X_k ).
	\end{eqnarray*}
From these we   see that 
if  Equation \eqref{CtoCL} holds, then the two sides of 
Equation \eqref{C1} match.
\end{itemize}  
\end{proof}

\subsection{Lie algebroid $(0,k)$-characteristic pairs  and IM-forms}

\begin{definition}\label{2.16}  Let $(A,[{\tobefilledin},{\tobefilledin}],\rho)$ be a Lie algebroid over $M$. Let $k\geqslant 1$ be an integer. 
A {\bf $(0,k)$-characteristic pair} on $A$ is a pair $(\mu,\theta)$  where
$\mu\in Z^1(\jet A,\wedge^k T^*M)$ is a Lie algebroid $1$-cocycle, $\theta\in \Gamma(A^*\otimes (\wedge^{k-1} T^*M))$ is  a $\rho$-compatible $(0,k)$-tensor, and they are subject to the following two conditions: 
\begin{itemize}
\item[\rm{(i)}] For all $H\in \core= \mathrm{Hom}(TM,A)$, we have
\begin{equation}\label{Eqt:edfu}\mu(H)=-(H^*\otimes \id_{T^*M}^{\otimes( k-1)})  \theta ;\end{equation}
\item[\rm{(ii)}] The identity 
\begin{eqnarray}\label{Drhomu}\Drhostarpush\circ \mu =-d_{\jet A} \theta ,
\end{eqnarray} i.e., \begin{equation}\label{Eqt:tempsecond}\iota_{\rho(v)} \mu(\liftingd u)= {\iota_{[u,v]} \theta}-\mathcal{L}_{\rho(u)} {\iota_v \theta}\end{equation}
holds for all   $u,v\in \Gamma(A)$. \end{itemize} Recall here   $d_{\jet A}: ~\Gamma(A^*\otimes (\wedge^{k-1} T^*M))\to C^1(\jet A,A^*\otimes (\wedge^{k-1} T^*M))$ is the Lie algebroid differential   with coefficient in the  $\jet A$-module $A^*\otimes (\wedge^{k-1} T^*M)$.
\end{definition}

Another way expressing Condition \eqref{Eqt:edfu} is that the following  relation:
 \begin{equation}\label{Eqt:edfu2}\mu(df\otimes u)=-df\wedge {\iota_u \theta}\end{equation}
holds for all $H=df\otimes u\in T^*M\otimes A$.

Indeed, such pairs $(\mu,\theta)$ are also   introduced in {\cite{C}}. We call them Lie algebroid characteristic pairs because  they can be viewed as the infinitesimal counterpart of groupoid $(0,k)$-characteristic pairs as in Definition \ref{Def:groupoid0kcharpair}.

\begin{proposition}\label{Prop:groupoidchartoalgebroidchar}
If $(e,\theta)$ is a $(0,k)$-characteristic pair on a Lie groupoid $\Gpd$, then the pair $(\widehat{e},\theta)$, where $\widehat{e}\in Z^1(\jet A,\wedge^k T^*M)$ is the infinitesimal $1$-cocycle of $e$,
is a  $(0,k)$-characteristic pair on the tangent Lie algebroid $A$ of $\Gpd$.
\end{proposition}
\begin{proof}We use Proposition \ref{co cp} which describes the conditions of $e$ and $\theta$. Clearly, Equation \eqref{coo3} together with the definition of infinitesimal  \eqref{Eqt:hatcdefinition} and relation  \eqref{Eqt:dGpddA} implies the second condition in Definition \ref{2.16}.  	
	
	We now derive Equation \eqref{Eqt:edfu} from Equation \eqref{co1}.  Consider $H\in  \mathrm{Hom}(TM, A)$ and a curve of isotropy jets $[h(\epsilon)]=\id+\epsilon H \in \heat$. 
	Then, we can compute $\widehat{e}(H)$ according to Equation  \eqref{Eqt:hatcdefinition}: 
	\begin{eqnarray}\nonumber
		\widehat{e}(H)&=&-\frac{d}{d\epsilon}|_{\epsilon=0} \Ad^\vee_{[h(\epsilon)]^{-1}} e[h(\epsilon)]\\\label{Eqn:ccheckHtemp1}
		&=& 
		-\frac{d}{d\epsilon}|_{\epsilon=0}(\id+\epsilon \rho H)^{*\otimes k}   (R_{( \id+\epsilon H)}^*B \theta-B\theta) .
		\end{eqnarray}
	Here we have used Equation \eqref{coadjoint}.  Observing the fact that $(R_{( \id+\epsilon H)}^*B \theta-B\theta))\in \wedge^k T^*M$ (see explanation after Proposition \ref{co cp} or Equation \eqref{Eqt:temp2temp}), we have
	\begin{eqnarray}\nonumber 
	 (\id+\epsilon \rho H)^{*\otimes k} (R_{(\id+\epsilon H)}^*B \theta-B\theta) &=&
	(\id+\epsilon \rho H)^{*\otimes k}\circ	\mathrm{pr}_{\Gamma(\wedge^k T^*M)} (R_{(\id+\epsilon H)}^*B \theta-B\theta)\\\nonumber &=&(\id+\epsilon \rho H)^{*\otimes k}\circ\mathrm{pr}_{\Gamma(\wedge^k T^*M)}  \circ R_{(\id+\epsilon H)}^*( B \theta )\\\label{Eqn:Rstartemp}&=&((\id+\epsilon \rho H)^{*}\circ \mathrm{pr}_{\Gamma(T^*M)}\circ R_{(\id+\epsilon H)}^*)^{\otimes k} B\theta.
	\end{eqnarray}
	We then recall from Equation   \eqref{Rhxichi} the formula of $R_{(\id+\epsilon H)}^* $ and derive  the following expression of composition of three maps:
	$$(\id+\epsilon \rho H)^{* }\circ \mathrm{pr}_{\Gamma( T^*M)}\circ R_{(\id+\epsilon H)}^*:~\begin{cases}
		A^*\to T^*M,\quad \chi\mapsto \epsilon (\id+\epsilon \rho H)^{-1*}\circ   H^* \chi\mapsto \epsilon    H^* \chi,\\T^*M\to T^*M, \quad \xi \mapsto (\id+\epsilon \rho H)^{-1*}\xi\mapsto \xi.
	\end{cases}
	$$

	So we are able to continue from \eqref{Eqn:ccheckHtemp1} and \eqref{Eqn:Rstartemp},  getting
		\begin{eqnarray}\nonumber
		\widehat{e}(H)&=&
	-\frac{d}{d\epsilon}|_{\epsilon=0}  (\id_{T^*M}\oplus \epsilon H^*)^{\otimes k} B\theta=-(H^*\otimes \id_{T^*M}^{\otimes( k-1)})  \theta.
\end{eqnarray}
In the last step, we have used the following relation:
$$ (\id_{T^*M}\oplus \epsilon H^*)^{\otimes k} B\theta \equiv \epsilon (H^*\otimes \id_{T^*M}^{\otimes( k-1)})  \theta \mod \epsilon^2,
$$	 
 which is easily seen.
\end{proof}
The following lemma is the Lie algebroid version of Lemma  \ref{Lem:exactcharpairs}.
 \begin{lemma}\label{Lem:Liealgebraspecialcharpair} Let $A$ be a Lie algebroid over $M$ with anchor map $\rho:~A\to TM$. 
	To every $\gamma\in \Omega^k(M)$, there is an  associated   $(0,k)$-characteristic pair $(\mu_{\gamma},\theta_{\gamma})$ on $A$, where $\mu_{\gamma}:\jet A\to \wedge^k T^*M$ and $\theta_{\gamma}\in \Gamma(A^*\otimes (\wedge^{k-1} T^*M))$ are defined respectively by
	\[\mu_\gamma=d_{\jet A} (\gamma)\qquad\mbox{ and }~ \theta_\gamma=-\Drhostarpush\gamma  . \]
\end{lemma}
 \begin{proof}By Example \ref{2.4}, we see that $\theta_\gamma$ is  $\rho$-compatible. Following the definition, we have
	$$\mu_\gamma(\liftingd u)=\mathcal{L}_{\rho(u)} \gamma,~ \forall~ u\in \Gamma(A).$$
	So we can examine: 
\[\mu_{\gamma}(df\otimes u)=\mu_{\gamma}(\liftingd(fu)-f\liftingd u)= \mathcal{L}_{f\rho(u)} \gamma-f\mathcal{L}_{\rho(u)}\gamma= df\wedge \iota_{\rho(u)}\gamma= - df\wedge \iota_u \theta_{\gamma},\]
which verifies \eqref{Eqt:edfu2}.
Also, we check that
\[\iota_{\rho(v)} \mu_\gamma(\liftingd u)= \iota_{\rho(v)}\mathcal{L}_{\rho(u)} \gamma=-\iota_{[\rho(u),\rho(v)]} \gamma+\mathcal{L}_{\rho(u)} \iota_{\rho(v)} \gamma=\iota_{[u,v]} \theta_\gamma-\mathcal{L}_{\rho(u)} \iota_v \theta_{\gamma},\]
which fulfils (ii) of Definition \ref{2.16}. This proves that $(\mu_{\gamma},\theta_{\gamma})$ is a $(0,k)$-characteristic pair on $A$.\end{proof}

It turns out that   
  Lie algebroid $(0,k)$-characteristic pairs  are variations of  the well-known notion of    IM-forms  (abbreviated from  \textit{infinitesimally multiplicative}).  
  
\begin{definition}\rm{(\cite{BC})} Let $k\geqslant 1$ be an integer.
An {\bf IM $k$-form} on a Lie algebroid $A\to M$ is a pair $(\nu,\theta)$ of vector bundle maps   $\nu: A\to \wedge^k T^*M$ and $\theta:A\to \wedge^{k-1} T^*M$ satisfying the following conditions:
\begin{itemize}
\item[\rm (1)] $\iota_{\rho(u)} \theta(v)=-\iota_{\rho(v)} \theta(u)$;
\item[\rm (2)] $\theta[u,v]=\mathcal{L}_{\rho(u)} \theta(v)-\iota_{\rho(v)} d\theta(u)-\iota_{\rho(v)}\nu(u)$;
\item[\rm (3)] $\nu[u,v]=\mathcal{L}_{\rho(u)} \nu(v)-\iota_{\rho(v)}d\nu(u)$,
\end{itemize}
for all $u,v\in \Gamma(A)$.
\end{definition}

\begin{example}Consider the $k=1$ case. An IM $1$-form is a pair $(\nu, \theta)$ formed by $\nu:A\to T^*M$  (seen as in $\Gamma(A^*\otimes T^*M)$) and $\theta\in \Gamma(A^*)$ 
		such that 
		\[(\id\otimes \rho^*)\nu=d_A \theta \qquad\mbox{and }~  \mathcal{L}_{\rho(u)}(\nu)= \Drhostarpush d\nu(u),\qquad u \in \Gamma(A),\]
		where $d_A:\Gamma(A^*)\to \Gamma(\wedge^2 A^*)$ is the differential of the Lie algebroid $A$.
\end{example}

\begin{proposition}\label{Prop:LiepaircharpairtoIMform} 
		There is a one-to-one correspondence between the set of  $(0,k)$-characteristic pairs $(\mu,\theta)$ on a Lie algebroid $A$ and the set of IM $k$-forms $(\nu,\theta)$  such that
	\[  \nu(u)=-\mu(\liftingd u)-d{\iota_u \theta},\qquad \forall u\in \Gamma(A).\]
\end{proposition} 
The proof of this proposition is a direct verification, which we omit.

 Note that we have assumed that $k\geqslant 1$ in the above discussions. For the extreme case of a multiplicative $0$-form, namely a multiplicative function $f\in \OmegakmultG{0}=Z^1(\Gpd,M\times \mathbb{R})$, its infinitesimal is the Lie algebroid $1$-cocycle $\hat{f}$ (see Example \ref{Exm:dfchar}).
So we can simply define  {\bf IM $0$-forms} on a Lie algebroid $A$ to be   Lie algebroid $1$-cocycles of $A$.


\subsection{The transitive case}
 {
 A Lie groupoid $\Gpd$ is called transitive if 
 given any two points in the base manifold, there is at least one element in $\Gpd$ connecting them. A Lie algebroid $A$ over $M$ is called transitive if its anchor map $\rho: A\to TM$ is  surjective. It is standard that the tangent Lie algebroid  of a transitive Lie groupoid is  transitive.}

 {
\begin{lemma}\label{transitiverho}
If $A$ is  transitive and $k\geqslant 2$, then any $\rho$-compatible $(0,k)$-tensor $\theta\in \Gamma(A^*\otimes (\wedge^{k-1}T^*M))$  is determined uniquely by some  $\gamma\in \Omega^k(M)$ such that  $\theta=\Drhostarpush \gamma$ (cf.  Example \ref{2.4}). 
\end{lemma}}
\begin{proof}As $\rho$ is surjective, given $\theta$ we can define $\gamma$ by the relation $$\iota_X \gamma=\iota_u\theta,\qquad ~\forall X\in TM \mbox{ and } u\in \Gamma(A)  \mbox{~such that~} ~ \rho(u)=X .$$ The $\rho$-compatibility property of $\theta$ guarantees that $\gamma$ is well-defined when $k\geqslant 2$. It is clear that the above relation is equivalent to  $\theta=\Drhostarpush \gamma$. Uniqueness of $\gamma$ is thus apparent. 
\end{proof}
 
\begin{proposition}\label{transgpd}
	Let $\Gpd$ be a transitive Lie groupoid over $M$. 
	\begin{itemize}
	\item[\rm (1)]If   $k\geqslant 2$, then  all   $(0,k)$-characteristic pairs on $\Gpd$   are of the form $(d_{\jet \Gpd}(\gamma), -\Drhostarpush \gamma)$ as described by Lemma \ref{Lem:exactcharpairs}.
		\item[\rm (2)]If   $k\geqslant 2$, then 
	all multiplicative $k$-forms on $\Gpd$   are exact, namely they are of the form $s^*\gamma-t^*\gamma$ for  $\gamma\in \Omega^k(M)$; 
	\item[\rm (3)] Given any $\theta\in \Gamma(A^*)$ satisfying the condition
	\begin{equation}\label{Eqt:conditionoftheta}\iota_v ( d_{\jet \Gpd} \theta)=0,\quad\forall   v\in \ker\rho,\end{equation}
	there exists a unique    $1$-cocycle $e_\theta\colon \jet \Gpd\to T^*M$ such that the pair $(e_\theta,\theta)$ is a $(0,1)$-characteristic pair on $\Gpd$.
		 Moreover, all $(0,1)$-characteristic pairs on $\Gpd$ arise from this construction.  
 \item[\rm (4)] Every $\theta$ satisfying Condition \eqref{Eqt:conditionoftheta} gives rise to a  multiplicative $1$-form $\Theta$ on $\Gpd$ such that
 $$
 \Theta_g(R_{[b_g]}(u+X))=\theta_{t(g)}(u+v)-\theta_{s(g)}(\Ad_{[b_g]^{-1}}v)
 $$
 for all    $u\in A_{t(g)}$, $X\in T_{t(g)} M$, bisection $b_g$ passing through $g\in \Gpd$, and $v\in A_{t(g)}$ satisfies $\rho(v)=X$. 
 All multiplicative $1$-forms $\Theta$ on $\Gpd$ are of this form.
 
  \end{itemize}
\end{proposition}
\begin{proof} (1) Let $(e,\theta)$ be a $(0,k)$-characteristic pair on $\Gpd$. Following Lemma \ref{transitiverho} and since $\theta$ is $\rho$-compatible and $\Gpd$ is transitive, we have $\theta=\Drhostarpush \gamma$ where $\gamma\in \Omega^k(M)$. Then the compatibility condition \eqref{coo3} becomes
	 $\Drhostarpush\circ e =-\Drhostarpush\circ d_{\jet \Gpd} \gamma$. So $e$ is indeed the negative of $d_{\jet \Gpd} \gamma$ for $\rho$ being surjective.
	 
	 Statement (2) is implied by (1) due to the one-to-one correspondence established by Proposition \ref{4.1}.


(3) First, as we have Condition \eqref{Eqt:conditionoftheta}, $d_{\jet \Gpd} \theta$ is indeed a map $\jet \Gpd\to \mathrm{Im} \rho^*$. Second, since $\rho^*$ is injective, we can define $e_\theta$ via the relation 
\begin{equation}\label{Eqt:k=1case2gpd}
\rho^*\circ e_\theta  =- d_{\jet \Gpd} \theta   .
\end{equation}
To see  that $e_\theta$ is a $1$-cocycle, one notices that the right hand side of \eqref{Eqt:k=1case2gpd} is a $1$-cocycle valued in $A^*$, and $T^*M\xrightarrow{\rho^*} A^*$ is compatible with the actions of $\jet\Gpd$ on $T^*M$ and $A^*$.
 
Now we examine that $(e_\theta,\theta)$ is a $(0,1)$-characteristic pair on $\Gpd$. The given $\theta\in \Gamma(A^*)$ is certainly $\rho$-compatible. One of the conditions we need is Equation \eqref{coo3} which now becomes Equation \eqref{Eqt:k=1case2gpd}.   

The other condition we need, namely Equation \eqref{co1}, now reads $$e_\theta [h]=R_{[h]}^*( \theta)- \theta, \qquad \forall [h]\in \heat.$$
Indeed, it is implied by Equation \eqref{Eqt:k=1case2gpd} as $\rho^*$ is injective. 

Finally, given any $(0,1)$-characteristic pair $(e,\theta)$ on $\Gpd$, one   easily finds that $e$ must be of the form $e_\theta$.

 Following the result of Statement (3),    one can use Proposition \ref{4.1} to show Statement (4) directly. 

\end{proof}
\begin{remark}
	Our earlier work \cite{CL1} presents some results on the structure of multiplicative multi-vector fields   on   transitive Lie groupoids.
\end{remark}

We now turn to   Lie algebroid characteristic pairs and IM forms.   The statements (2) and (4) in the following proposition already appeared in \cite[Remark 3.5]{BC}.
\begin{proposition}
Let $A$ be a transitive  Lie algebroid over $M$. 
\begin{enumerate}
	\item[\rm (1)]If $k\geqslant 2$, then all $(0,k)$-characteristic pairs  on $A$ are   of the form 
	$(\mu_{\gamma}=d_{\jet A}(\gamma),\theta_{\gamma}=-\Drhostarpush\gamma)$ as described by Lemma \ref{Lem:Liealgebraspecialcharpair}, where $\gamma\in \Omega^k(M)$. 
	\item[\rm (2)]If $k\geqslant 2$, then all  IM $k$-forms  are of the form $(\nu_\gamma,\theta_\gamma=-\Drhostarpush\gamma)$ where  $\gamma\in \Omega^k(M)$ and  $\nu_{\gamma}: A\to \wedge^k T^*M$ is determined  by the formula 
	\[\nu_\gamma(u)=\iota_{\rho(u)} d\gamma,\qquad \forall u\in \Gamma(A).\]
	\item[\rm (3)] Given any $\theta\in \Gamma(A^*)$ satisfying the condition 
\begin{equation}\label{Eqt:conditionoftheta2}\iota_v ( d_A \theta)=0,\quad\forall   v\in \ker\rho,\end{equation}	
	there exists a unique    $1$-cocycle $\mu_\theta\colon \jet A\to T^*M$ such that the pair $(\mu_\theta,\theta)$ is a $(0,1)$-characteristic pair on $A$.
	The element $\mu_\theta$ is defined by the relation: \[\rho^* \circ \mu_\theta (\liftingd  u) =-\mathcal{L}_u \theta,\qquad \forall u\in \Gamma(A).\] 

Moreover, all $(0,1)$-characteristic pairs on $A$ arise from this construction. 	
	 \item[\rm (4)] Given any $\theta$ satisfying Condition \eqref{Eqt:conditionoftheta2}, there exists a unique    $\nu_\theta: A\to \wedge^k T^*M$ such that   
  $(\nu_\theta,\theta)$ is an IM $1$-form.  The element $\nu_\theta$ is defined by the relation: 
  \[\rho^*\circ \nu_\theta(u)=\iota_u (d_A \theta),\qquad \forall u\in \Gamma(A).\]  
  
  Moreover, all IM $1$-forms of $A$ are of this form.
  \end{enumerate}

\end{proposition}

The proof of this proposition is completely similar to the   previous one, so we   omit it. 

\section{The complex of multiplicative forms}\label{Sec:CartanComplex}
\subsection{The de Rham differential}
It is easily verified that the standard de Rham differential $d:~\Omega^\bullet(\Gpd)\to \Omega^{\bullet+1}(\Gpd)$ maps multiplicative $k$-forms   to multiplicative $(k+1)$-forms. In plain terms,  
$(\OmegakmultG{\bullet},d)$ is a subcomplex of $(\Omega^\bullet(\Gpd),d)$.
As $\OmegakmultG{\bullet}$ corresponds to groupoid $(0,\bullet)$-characteristic pairs,  it is  tempting to describe $d$ in terms of   characteristic pairs as well.

First, given $f\in \OmegakmultG{0}$ (a multiplicative function on $\Gpd$), we have $df\in \OmegakmultG{1}$ which corresponds to the $(0,1)$-characteristic pair $(e,\theta)$ as described in Example \ref{Exm:dfchar}. 
 
Second, for all $k\geqslant 1$, we characterise $d:\OmegakmultG{k}\to \OmegakmultG{k+1}$ as follows.

\begin{proposition}\label{Prop:differentialofcharpairs}
	Given a multiplicative $k$-form $\Theta\in \OmegakmultG{k}$ which corresponds to the $(0,k)$-characteristic pair $(e,\theta)$,   the $(0,k+1)$-characteristic pair of $d\Theta\in \OmegakmultG{k+1}$, denoted by $(\tilde{e}, \tilde{\theta})$, is given as follows: 
	\begin{itemize}
		\item[(1)]  The map $\tilde{e}:$  $\jet\Gpd \rightarrow \wedge^{k+1} T^*M$ is determined by  
		\begin{eqnarray}\label{tildec2}
		\tilde{e}[b]=d(e[b]),
		\end{eqnarray}
		for all bisections $b: M\to   \Gpd$      of $\Gpd$. Here we treat $[b]: M\to  \jet\Gpd$ as a   section of the fibre bundle $\jet\Gpd \stackrel{s}{\to}M$,  $e[b]$ as in $\Omega^k(M)$, and $d: \Omega^k(M)\to \Omega^{k+1}(M)$ is the standard de Rham differential.
		\item[(2)] The section $\tilde{\theta}\in \Gamma(A^*\otimes (\wedge^k T^*M))$ is determined by
		\begin{eqnarray}\label{tildetheta}
		\iota_u \tilde{\theta} &=&-d ( \iota_u \theta ) -\widehat{e}(\liftingd u),\end{eqnarray}
		for all $u$$\in$$\Gamma(A)$. Here $\widehat{e}$$\in$$Z^1(\jet A,\wedge^k T^*M)$ is the infinitesimal of the $1$-cocycle  $e$$\in$  $Z^1(\jet \Gpd,\wedge^k T^*M)$, and   $\liftingd$$:\Gamma(A)\to \Gamma({\jet A})$ is the lifting map defined by Equation \eqref{Eqt:liftingd}. 
	\end{itemize}
\end{proposition}

\begin{proof}By Corollary \ref{Cor:fromThetatocharpair}, we have
	$$
	\tilde{e}[b]=R_b^{!*}(d\Theta)=d\big(  R_b^{!*}( \Theta)\big)=d(e[b]),
	$$
	which proves Equation \eqref{tildec2}. 
	
	Also by  Corollary \ref{Cor:fromThetatocharpair}, we have\[  \tilde{\theta}=\mathrm{pr}_{\Gamma(A^*\otimes (\wedge^{k } T^*M))} (d\Theta)|_M.\]
	Therefore, to find $\tilde{\theta}$ at $x\in M$, we need to consider arbitrary $u_x\in A_x$ and $X_{1x},\cdots, X_{kx}\in T_x M$, and to evaluate 
	\begin{equation}
	\label{Eqt:tildethetastarts}
	\tilde{\theta}_x(u_x, X_{1x},\cdots, X_{kx})=(d\Theta)(u_x, X_{1x},\cdots, X_{kx}).
	\end{equation} 
	
	We   extend $u_x$ to a smooth section $u\in \Gamma(A)$ which has its exponential $\exp{\epsilon u}\in \mathrm{Bis}(\Gpd )$ (for $|\epsilon|$ sufficiently small). 
	The flow of the right invariant vector field $\overrightarrow{u}\in \mathfrak{X}^1(\Gpd)$ is given by $L_{\exp{\epsilon u}}$. 
	The map $\phi(\epsilon)= t\circ \exp{\epsilon u}:M\to M$ is indeed the flow of $\rho(u)=t_* \overrightarrow{u}\in \mathfrak{X}^1(M)$.

	Using the flow $\phi(\epsilon)$,
	one is able to find extensions $X_i\in \mathfrak{X}^1(M)$ of $X_{ix}$ such that $X_i|_{\phi(\epsilon)x}=\phi(\epsilon)_* X_{ix}$.
	
	We then use the flow $L_{\exp{\epsilon u}}$ of $\overrightarrow{u}$ to extend   $X_i\in \mathfrak{X}^1(M)$ to ${\tilde{X_i}}\in \mathfrak{X}^1(\Gpd^{o})$, $\Gpd^{o}$ being a small neighbourhood of the identity section $M\subset \Gpd$, such that
	$$
	\tilde{X_i}|_{ \exp{\epsilon u}(y)}=L_{\exp{\epsilon u}*} X_{i}|_{ y},\qquad\forall y\in M.
	$$
	It follows that $[\overrightarrow{u},\tilde{X_i}]=0$ and 
	\begin{equation}
	\label{Eqt:tildeXiatflow}
	\tilde{X_i}|_{ \exp{\epsilon u}(x)}=L_{\exp{\epsilon u}*} X_{ix}=
	R_{ \exp{\epsilon u}*}^{!} X_{i}|_{\phi(\epsilon)x}.
	\end{equation}
	We are now ready to compute:
	\begin{eqnarray*}
		\mbox{Equation \eqref{Eqt:tildethetastarts}}&=&  (d\Theta)|_x(\overrightarrow{u}, \tilde{X}_{1} ,\cdots,\tilde{X}_{k})\\ &=&\sum_{i=1}^k (-1)^{i} X_{ix}\bigl(\Theta(\overrightarrow{u},\cdots, \widehat{ \tilde{X_i} },\cdots) \bigr ) -\sum_{i<j}(-1)^{i+j} \Theta_x(\overrightarrow{u}, [\tilde{X_i} ,\tilde{X_j}],\cdots) \\ &&\qquad +\overrightarrow{u}_x \bigl(\Theta( \tilde{X}_{1} ,\cdots,\tilde{X}_{k}) \bigr ) \\ &=&
		\sum_{i=1}^k (-1)^{i} X_{ix}\bigl((\iota_u \theta)(\cdots, \widehat{   X_i  },\cdots) \bigr ) -\sum_{i<j}(-1)^{i+j} (\iota_u \theta)_x(  [ {X_i} , {X_j}],\cdots) \\ &&\qquad +\overrightarrow{u}_x \bigl(\Theta( \tilde{X}_{1} ,\cdots,\tilde{X}_{k}) \bigr ). 
	\end{eqnarray*}
	The first two terms add to $(-d(\iota_u \theta))_x(X_1,\cdots,X_k)$, while the third one is
	\begin{eqnarray*} 
		&&\frac{d}{d\epsilon}|_{ \epsilon =0} \Theta_{\exp\epsilon u(x)}( \tilde{X}_{1} ,\cdots,\tilde{X}_{k})
		\\	&=&
		\frac{d}{d\epsilon}|_{ \epsilon =0} R_{ \exp{\epsilon u}}^{!*}\Theta_{\exp\epsilon u(x)}( X_{1}|_{\phi(\epsilon)x} ,\cdots,X_{k}|_{\phi(\epsilon)x}),\quad\mbox{(by Equation \eqref{Eqt:tildeXiatflow})}
		\\
		&=& \frac{d}{d\epsilon}|_{ \epsilon =0} e[\exp\epsilon u]( X_{1}|_{\phi(\epsilon)x} ,\cdots,X_{k}|_{\phi(\epsilon)x}),\quad\mbox{(by Corollary \ref{Cor:fromThetatocharpair})}
		\\
		&=&\frac{d}{d\epsilon}|_{ \epsilon =0} ((\Ad^\vee_{\exp\epsilon u})^{-1}e[\exp\epsilon u])( X_{1x}  ,\cdots,X_{kx} )\\
		&=&\frac{d}{d\epsilon}|_{ \epsilon =0} (\Ad^\vee_{\exp( \epsilon \liftingd u)})^{-1} \bigl(e\circ \exp(\epsilon \liftingd u) \bigr)( X_{1x}  ,\cdots,X_{kx} )\\
		&=& -\widehat{e}(\liftingd u)(X_{1x}  ,\cdots,X_{kx}),\end{eqnarray*}
 where $\widehat{e}\in Z^1(\jet A,\wedge^k T^*M)$, the infinitesimal   of $e\in Z^1(\jet \Gpd,\wedge^k T^*M)$, is computed according to its definition formula \eqref{Eqt:hatcdefinition}. This proves Equation \eqref{tildetheta}.

\end{proof}
 
	\begin{example}
		Let $\Theta\in \Omega^2(\Gpd)$ be a presymplectic structure, i.e., $\Theta$ is a closed and multiplicative  (but not necessarily nondegenerate) 2-form. Suppose that $\Theta$ corresponds to the $(0,2)$-characteristic pair $(e,\theta)$, where $e\in Z^2(\jet \Gpd, \wedge^2 T^*M)$ and $\theta\in \Gamma(A^*\otimes T^*M)$. Then by Proposition \ref{Prop:differentialofcharpairs}, we have 
		\[d(e[b])=0,\qquad \widehat{e}(\liftingd u)=-d(\iota_u \theta),\qquad \forall u\in \Gamma(A),\]
		where $\widehat{e}\in Z^1(\jet A,\wedge^2 T^*M)$ is the infinitesimal of $e$. Thus, a $(0,2)$-characteristic pair $(e,\theta)$ of a presymplectic structure $\Theta$ satisfies \[\theta(u,\rho(v))=-\theta(v,\rho(u)),\qquad u,v\in \Gamma(A)\]
		and the infinitesimal $\widehat{e}$ of $e$ is determined by $\theta$ via $\widehat{e}(\liftingd u)=-d(\iota_u \theta)$. Moreover, $\Theta$ and $(e,\theta)$ are related by the formula
		\[\Theta_g(R_{[b_g]}(u_1+X_1),R_{[b_g]}(u_2+X_2))=e[b_g](X_1,X_2)+\theta(u_1,X_2)-\theta(u_2,X_1)+\theta(u_1,\rho(u_2)).\]
	\end{example}

By taking infinitesimal,  multiplicative functions on $\Gpd$ correspond to Lie algebroid $1$-cocycles of  the tangent Lie algebroid $A$, and $(0,k)$-characteristic pairs on $\Gpd$ correspond  to $(0,k)$-characteristic pairs on   $A$ (see Proposition \ref{Prop:groupoidchartoalgebroidchar}), and also to IM $k$-forms  of $A$ (see Proposition \ref{Prop:LiepaircharpairtoIMform}). Then we are naturally led to consider the infinitesimal version of Proposition \ref{Prop:differentialofcharpairs}. 

\begin{proposition}Let $A$ be a Lie algebroid. Denote by $\mathrm{CP}^0(A)=Z^1(A,M\times \mathbb{R})$ the set of Lie algebroid $1$-cocycles   and $\mathrm{CP}^k(A)$ the set of $(0,k)$-characteristic pairs   on $A$  for   $k\geqslant  1$. Then $\mathrm{CP}^\bullet (A)=\oplus_{j=0}^{\dim M+1} \mathrm{CP}^j (A) $ admits a canonical cochain complex structure with the  differential   expressed as follows:
		\begin{enumerate}
		\item[\rm(1)] The differential $d:\mathrm{CP}^0(A)\to \mathrm{CP}^1(A)$ is simply  {$d(c )=(\mu_c,-c)$, for all $1$-cocycles $c\in Z^1(A,M\times \mathbb{R})$, where $\mu_c:\jet A\to T^*M$ is given by $\mu_c(\liftingd u)=d(c(u))$ (c.f. Example \ref{Exm:dfchar}).}
		\item[\rm(2)] For $1\leqslant k\leqslant \dim M$, the differential $d:\mathrm{CP}^k(A)\to \mathrm{CP}^{k+1}(A)$ is given by  {$d(\mu,\theta)=(\tilde{\mu},\tilde{\theta})$ } where
		\[\tilde{\mu}(\liftingd u):=d\mu(\liftingd u),\qquad\mbox{and~}~ \iota_u \tilde{\theta}:=-d(\iota_u \theta)-\mu(\liftingd u),\qquad \forall u\in \Gamma(A).\]
	\end{enumerate}

\end{proposition}

We have a direct  corollary following the one-to-one correspondence established by Proposition  \ref{Prop:LiepaircharpairtoIMform}.
\begin{corollary}\label{Cor:cochaincomplexofIMA}
	 Denote by $\mathrm{IM}^0(A)=Z^1(A,M\times \mathbb{R})$ the set of Lie algebroid $1$-cocycles   and $\mathrm{IM}^k(A)$ the set of IM $k$-forms of $A$  for   $k\geqslant  1$. Then $\mathrm{IM}^\bullet (A)=\oplus_{j=0}^{\dim M+1} \mathrm{IM}^j (A) $ admits a canonical cochain complex structure with the differential  as  described below:
	\begin{enumerate}
		\item [\rm(1)]The differential $d:\mathrm{CP}^0(A)\to \mathrm{CP}^1(A)$ is simply  {$d(c )=(0,-c)$, for all $1$-cocycles $c\in Z^1(A,M\times \mathbb{R})$.}
		\item[\rm(2)] For $1\leqslant k\leqslant \dim M$, the differential $d:\mathrm{CP}^k(A)\to \mathrm{CP}^{k+1}(A)$ is given by  {$d(\nu,\theta)=(0,\nu)$.}
	\end{enumerate}
\end{corollary}

We note that the above fact already appeared in \cite{BC}.

\subsection{The Cartan calculus}
In this part, useful formulas \textit{\'{a} la} Cartan are introduced to describe the interaction between multiplicative multi-vector fields and forms on a Lie groupoid $G\rightrightarrows M$. The notation used is consistent with the earlier sections.

\begin{lemma}\label{contraction}
	Let $\Pi\in \XkmultG{k}$ ($k\geqslant  1$) and $\alpha\in \OmegakmultG{1}$ be given.
	\begin{itemize}
		\item [(1)]
		Their contraction is also multiplicative, i.e.  $\iota_\alpha \Pi\in \XkmultG{k-1}$;	
		\item [(2)]For $\gamma\in \Omega^1(M)$, we have $\iota_{s^*\gamma} \Pi=\overleftarrow{\iota_\gamma \pi}$  and $\iota_{t^*\gamma} \Pi=\overrightarrow{\iota_\gamma \pi}$, where $\pi\in \Gamma(TM\otimes (\wedge^{k-1} A))$ is the leading term of $\Pi$;
		\item [(3)]For $u\in \Gamma(\wedge^k A)$, we have $\iota_{\alpha} \overleftarrow{u}=\overleftarrow{\iota_{a} u}$  and $\iota_{\alpha} \overrightarrow{u}=\overrightarrow{\iota_{a} u}$, where $a\in \Gamma(A^*)$ is the leading term of $\alpha$.
	\end{itemize} 
\end{lemma}
\begin{proof}  To show that $\iota_\alpha \Pi\in \XkmultG{k-1}$, we need two facts:
	\begin{itemize}
		\item[ (i)]   An $n$-vector  field $\Gamma\in \mathfrak{X}^n(\Gpd)$ is multiplicative if and only if 
		\[\Gamma_{gr}(\alpha^1_g\cdot \beta^1_r,\cdots, \alpha^n_g\cdot \beta^n_r)=\Gamma_g(\alpha^1_g,\cdots,\alpha^n_g)+\Gamma_r(\beta^1_r,\cdots, \beta^n_r),\qquad \forall (g,r)\in \Gpd^{(2)},\]
		for all composable pairs $(\alpha^i_g\in T_g^* \Gpd,\beta^i_r\in T_r^* \Gpd)$,  $i=1,\cdots, n$. See \cite[Proposition 2.7]{ILX}. \item[ (ii)] A $1$-form $\alpha\in \Omega^1(\Gpd)$ is multiplicative if and only if  $\alpha:\Gpd\to T^*\Gpd$ is a groupoid morphism.  This is explained in Section \ref{Sec:multkforms} (see also \cite{Kosmann}). 
	\end{itemize}
	Since our $\Pi$ and $\alpha$ are both multiplicative, we have  $\alpha_{gr}=\alpha_g\cdot \alpha_r$ for $(g,r)\in \Gpd^{(2)}$ by (ii) and 
	\begin{eqnarray*}
		(\iota_\alpha\Pi)_{gr}(\alpha^1_g\cdot \beta^1_r,\cdots, \alpha^{k-1}_g\cdot \beta^{k-1}_r)&=&\Pi_{gr}(\alpha_{gr}, \alpha^1_g\cdot \beta^1_r,\cdots, \alpha^{k-1}_g\cdot \beta^{k-1}_r)\\ &\stackrel{(ii)}{=}&\Pi_{gr}(\alpha_g\cdot \alpha_r, \alpha^1_g\cdot \beta^1_r,\cdots, \alpha^{k-1}_g\cdot \beta^{k-1}_r)\\ &\stackrel{(i)}{=}&\Pi_g(\alpha_g,\alpha^1_g,\cdots, \alpha^{k-1}_g)+\Pi_r(\alpha_r,\beta^1_r,\cdots, \beta^{k-1}_r)\\ &=&
		(\iota_\alpha \Pi)_g(\alpha^1_g,\cdots,\alpha^{k-1}_g)+(\iota_\alpha \Pi)_r(\beta^1_r,\cdots,\beta^{k-1}_r).
	\end{eqnarray*}
	By (i) again, the above relation proves the assertion  $\iota_\alpha \Pi\in \XkmultG{k-1}$.

	Next, we show the  equality $\iota_{s^*\gamma} \Pi=\overleftarrow{\iota_\gamma \pi}$. In fact, by $\Pi$ being multiplicative, we have a groupoid morphism  
	\begin{equation*}
	\xymatrix{
		\oplus^{k-1} T^*\Gpd\ar@{->}@< 2pt> [d]
		\ar @{ ->} @<-2pt> [d]  \ar[r]^{\quad \Pi^\sharp} &T\Gpd\ar@{->}@< 2pt> [d]
		\ar @{ ->} @<-2pt> [d] \\
		\oplus^{k-1} A^*\ar[r]^{\pi^\sharp} & TM \ .}
	\end{equation*}
 Therefore we have the relation
	$s_*\circ \Pi^\sharp=\pi^\sharp \circ (\oplus^{k-1} s)
	$. Using this, we can examine  the relation
	\begin{eqnarray*}
		(\iota_{s^*\gamma}\Pi)(\alpha^1,\cdots, \alpha^{k-1})&=&(-1)^{k-1}\langle \Pi^\sharp(\alpha^1,\cdots,\alpha^{k-1}),s^*\gamma\rangle\\ &=&(-1)^{k-1}\langle s_* \Pi^\sharp(\alpha^1,\cdots,\alpha^{k-1}),\gamma\rangle\\ &=&(-1)^{k-1}\langle \pi^\sharp(s(\alpha^1),\cdots,s(\alpha^{k-1})),\gamma\rangle=(\iota_{\gamma}\pi)( s(\alpha^1),\cdots,s(\alpha^{k-1}))\\ &=&\overleftarrow{\iota_\gamma \pi}(\alpha^1,\cdots, \alpha^{k-1}).
	\end{eqnarray*}
	The  last step is due to the definition  of $s:~T^*\Gpd\rightarrow A^*$ as in \eqref{stco}. The other equality $\iota_{t^*\gamma} \Pi=\overrightarrow{\iota_\gamma \pi}$ is approached similarly.
	

	We finally show $\iota_{\alpha} \overleftarrow{u}=\overleftarrow{\iota_{a} u}$ (the other one is similar). For this, we need the fact $s(\alpha)=a$ (see Remark \ref{alpha}). Then we have
	\begin{eqnarray*}
		(\iota_{\alpha}\overleftarrow{u})(\alpha^1,\cdots,\alpha^{k-1})&=&u(s(\alpha),s(\alpha^1),\cdots,s(\alpha^{k-1}))=u(a,s(\alpha^1),\cdots,s(\alpha^{k-1}))
		\\ &=&(\iota_a u)(s(\alpha^1),\cdots,s(\alpha^{k-1}))\\ &=&\overleftarrow{\iota_a u}(\alpha^1,\cdots,\alpha^{k-1}),
	\end{eqnarray*}
as desired.
\end{proof}

\begin{example}  
	For a Poisson Lie groupoid $(\Gpd,P)$ (see the next section for more details), the Hamiltonian vector field $X_f:=P^\sharp(df)$ of a multiplicative function $f\in C^\infty(\Gpd)$ is a multiplicative vector field, by (1) of Lemma \ref{contraction} and the fact that $df\in \OmegakmultG{1}$.\end{example}

\begin{example} 
	Consider an exact $2$-vector field $ P\in\XkmultG{2}$ which is of the form  $P=\overleftarrow{u}-\overrightarrow{u}$, where $u\in \Gamma(\wedge^2 A)$. Then for any $\alpha\in \OmegakmultG{1}$, by (3) of Lemma \ref{contraction}, we have 
	\[P^\sharp (\alpha)=\overleftarrow{\iota_a u}-\overrightarrow{\iota_a u}\in \XkmultG{1},\]
	where $a\in \Gamma(A^*)$ is the leading term of $\alpha$. In general, the contraction of $\alpha$ to every exact $k$-vector field  yields an exact $(k-1)$-vector field.   \end{example}

We have a   lemma parallel to the previous one.

\begin{lemma}\label{contraction2}
	Suppose that $\Theta\in \OmegakmultG{k}$ and $X\in  \XkmultG{1}$ are given.
	\begin{itemize}
		\item [(1)]
		Their contraction is also multiplicative, i.e.  $\iota_X \Theta\in \OmegakmultG{k-1}$;	
		\item [(2)]For $u\in \Gamma(A)$, we have $\iota_{\overleftarrow{u}}\Theta=s^*(\iota_u \theta)$  and $\iota_{\overrightarrow{u}}\Theta=t^*(\iota_u \theta)$, where $\theta\in \Gamma(A^*\otimes (\wedge^{k-1} T^*M))$ is the leading term   of  $\Theta $;
		\item [(3)]For $\gamma\in \Omega^k(M)$, we have $\iota_{X} s^*\gamma=s^*(\iota_x \gamma)$  and $\iota_{X}t^*\gamma=t^*(\iota_x \gamma)$, where $x\in \mathfrak{X}^1 (M)$ is the leading term of $X$.
	\end{itemize} 
\end{lemma}
 The proof is omitted. 
   {Note that part of Statement $\mathrm{(2)}$  is  reminiscent of Lemma \ref{lifting}.}

At this point, we see that the de Rham differential $d$ of forms, the contraction $\iota_X$ by a multiplicative vector  field $X\in  \XkmultG{1}$, and the Lie derivative via Cartan's formula
$$\mathcal{L}_{X}=\iota_X\circ d+d\circ \iota_X 
$$
all preserve  $\OmegakmultG{\bullet}$.

For a multiplicative $2$-vector field $P\in \XkmultG{2}$, we have seen that the map $P^\sharp$ maps a multiplicative $1$-form $\Theta$ to a multiplicative vector field  $P^\sharp(\Theta)$. Thereby,  if $\Theta$ corresponds to the $(0,1)$-characteristic pair $(e,\theta)$,  $P^\sharp(\Theta)$ should correspond to a $(1,0)$-characteristic pair $(c_{P^\sharp(\Theta)},\pi_{P^\sharp(\Theta)})$, where
\[e\in Z^1(\jet \Gpd, T^*M),\qquad \theta\in \Gamma(A^*),\qquad c_{P^\sharp(\Theta)}\in Z^1(\jet \Gpd, A), \quad\mbox{and}\quad \pi_{P^\sharp(\Theta)}\in \mathfrak{X}^1(M).\]

To find the explicit relations between these data,  we assume that $P\in \XkmultG{2}$ corresponds to the  $(2,0)$-characteristic pair  $(c_P, p)$, where $c_P\in Z^1(\jet \Gpd,\wedge^2 A)$ and $p\in \Gamma(TM\otimes A)$.

\begin{proposition}\label{cpmap}
	With assumptions as above,    the $1$-cocycle $c_{P^\sharp(\Theta)}:~\jet \Gpd\to A$ is given by
	\[c_{P^\sharp(\Theta)}([b_g])= (\Ad_{[b_g]} p)^\sharp \circ e([b_g])+\iota_\theta\circ  c_P([b_g]),\qquad  \]
	and $\pi_{P^\sharp(\Theta)}\in \mathfrak{X}^1(M)$ by $p^\sharp(\theta)$, i.e.
	$$\pi_{P^\sharp(\Theta)}(f)=-p  (df, \theta),\qquad \forall f\in \CinfM.
	$$
\end{proposition}
\begin{proof}     	The formula for $\pi_{P^\sharp(\Theta)}$ can be found by its definition: \[\pi_{P^\sharp(\Theta)}=\mathrm{pr}_{\Gamma(TM)} P^\sharp(\Theta)|_M=\mathrm{pr}_{\Gamma(TM)} P^\sharp|_M(\Theta|_M)=p^\sharp(\theta).\]

	Then, according to Equations \eqref{formula3} and \eqref{formula theta}, we have
	\[P_g=R_{g*}c_P([b_g])+L_{[b_g]}(p-\frac{1}{2}D_\rho p),\qquad P^\sharp(\Theta)_g=R_{g*}c_{P^\sharp \Theta}([b_g])+L_{[b_g]}(p^\sharp \theta),\]
	and 
	\[\Theta_g=R_{[b_g^{-1}]}^*(e[b_g]+\theta).\]
	Then we have 
	\begin{eqnarray*}
		c_{P^\sharp(\Theta)}([b_g])&=&R_{g^{-1}*}\big(P^\sharp(\Theta)_g-L_{[b_g]}(p^\sharp \theta)\big)
		\\ &=&R_{[b_g^{-1}]*} P^\sharp (R_{[b_g^{-1}]}^*(e[b_g]+\theta))-\Ad_{[b_g]}p^\sharp \theta
		\\ &=&(R_{[b_g^{-1}]*}P_g)^\sharp (e[b_g]+\theta)-\Ad_{[b_g]}p^\sharp \theta
		\\ &=&\bigl(c_P([b_g])+\Ad_{[b_g]}(p-\frac{1}{2}D_\rho p)\bigl)^\sharp(e[b_g]+\theta)-\Ad_{[b_g]}p^\sharp \theta
		\\ &=&\iota_\theta c_P([b_g])+(\Ad_{[b_g]} p)^\sharp e([b_g])+(\Ad_{[b_g]} p)^\sharp (\theta)-\frac{1}{2}(\Ad_{[b_g]} D_\rho p)^\sharp (e[b_g])-\Ad_{[b_g]}(p^\sharp \theta).
	\end{eqnarray*}
	To prove  the desired formula for $c_{P^\sharp(\Theta)}([b_g])$, it remains to show that the    last three terms in the last line above cancel out. In fact, by applying \eqref{C3},
	we have
	\[e[b_g](\Ad_{[b_g]} \rho(u))=\theta(\Ad_{[b_g]} u)-\theta(u),\qquad u\in \Gamma(A).\]
	We may write $p=X\otimes u$ for $X\in \mathfrak{X}^1(M)$ and $u\in \Gamma(A)$,  and then we have 
	\begin{eqnarray*}
		&&(\Ad_{[b_g]} p)^\sharp (\theta)-\frac{1}{2}(\Ad_{[b_g]} D_\rho p)^\sharp (e[b_g])-\Ad_{[b_g]}(p^\sharp \theta)\\ &=&
		\theta(\Ad_{[b_g]} u) \Ad_{[b_g]}X-e[b_g](\Ad_{[b_g]} \rho(u)) \Ad_{[b_g]} X-\theta(u) \Ad_{[b_g]} X\\ &=&0,
	\end{eqnarray*}
	where we have used the $\rho$-compatibility of $p$.  
\end{proof}

\begin{example} 
	For a Poisson Lie group $(G,P)$ (see \cite{LuW}), by Example \ref{form on group}, the characteristic pair of a multiplicative $1$-form $\Theta$ is $(0,\theta)$ such that $\Theta(g)=R_{g^{-1}}^*\theta$, where $\theta\in \mathfrak{g}^*$ is $G$-invariant. Also, the characteristic pair of $P$ is $(c_P,0)$, where $c_P\in Z^1(G,\wedge^2 \mathfrak{g})$ is subject to $c_P(g)=R_{g^{-1}*} P_g$. 
	The characteristic pair of $P^\sharp(\Theta)$ is $(c,0)$, where $c\in Z^1(G,\mathfrak{g})$ is determined by
	\[c(g)=R_{g^{-1}*}(P^\sharp \Theta)=R_{g^{-1}*} P^\sharp(R_{g^{-1}}^*\theta)=(R_{g^{-1}*} P_g)(\theta)=\iota_\theta c_P(g).\]
\end{example}

\begin{example} 
	By Lemma \ref{Lem:exactcharpairs}, the characteristic pair of an exact multiplicative $1$-form  $\Theta=s^*\gamma-t^*\gamma$ for $\gamma\in \Omega^1(M)$ is $(e=d_{\jet \Gpd}\gamma,\theta=-\rho^*\gamma)$, where $\rho^*$ is the dual map of the anchor $\rho$ of the Lie algebroid $A$. Based on Lemma \ref{contraction}, we know that \[P^\sharp(\Theta)=\overleftarrow{\iota_\gamma p}-\overrightarrow{\iota_\gamma p},\]
	and its 
	characteristic pair $(c,\pi)$ is given by
	\[c=d_{\jet \Gpd} (\iota_\gamma p)\in Z^1(\jet \Gpd, A),\quad\mbox{~and~}~ \pi=-\rho(\iota_\gamma p)\in \mathfrak{X}^1(M),\]
	by Example  \ref{exact case}.
	 One can also show these identities by utilization of Proposition \ref{cpmap}.  
	
\end{example}

\section{Multiplicative forms on Poisson groupoids}\label{Sec:PoissonGpd}

\subsection{Multiplicative $1$-forms on Poisson  groupoids}
Consider a smooth manifold $N$ and a bivector field $P\in \mathfrak{X}^2 (N)$. One can define a skew-symmetric bracket $[{\tobefilledin},{\tobefilledin}]_P$ on $\Omega^1(N)$  given by:
\begin{equation}\label{Eqt:Pbracket1forms}[\alpha,\beta]_P=\mathcal{L}_{P^\sharp \alpha} \beta -\mathcal{L}_{P^\sharp \beta} \alpha-dP(\alpha,\beta)
=d(\iota_{P^\sharp \alpha} \beta) + \iota_{P^\sharp \alpha}  d\beta-\iota_{P^\sharp \beta}  d\alpha,\qquad \forall \alpha,\beta\in \Omega^1(N),\end{equation}
and an anchor map   $P^\sharp:T^*N\to TN$, $\alpha\mapsto \iota_{\alpha}P$. We have two formulas  (see  \cite{Kosmann2}):
\begin{eqnarray}\label{quasiim}
[\alpha_1,[\alpha_2,\alpha_3]_P]_P+c.p.=-\frac{1}{2} L_{[P,P](\alpha_1,\alpha_2,\tobefilledin)} \alpha_3+c.p.+d([P,P](\alpha_1,\alpha_2,\alpha_3)),\qquad \forall \alpha_i\in \Omega^1(N),
\end{eqnarray} 
and  
\begin{equation}\label{Eqt:PsharpwrtPbracket}
P^\sharp[\alpha_1,\alpha_2]_P-[P^\sharp \alpha_1,P^\sharp \alpha_2]=\frac{1}{2}[P,P](\alpha_1,\alpha_2),\qquad \forall \alpha_i\in \Omega^1(N).\end{equation}

A Poisson manifold is a pair $(N,P)$ where $N$ is a smooth manifold and $P\in \mathfrak{X}^2 (N)$ is a   bivector field subject to $[P,P]=0$. Due to Equations \eqref{quasiim} and \eqref{Eqt:PsharpwrtPbracket}, we see that $T^*N$ is a Lie algebroid when equipped with the bracket $[\tobefilledin,\tobefilledin]_P$ and the anchor $P^\sharp$.

Recall that a {\bf Poisson groupoid} is a Lie groupoid $\Gpd$ with a multiplicative bivector field $P\in \XkmultG{2}$ such that $[P,P]=0$ (see  \cites{W1, MX1}).
In this section, we study the space $\OmegakmultG{1}$ of multiplicative $1$-forms on a Poisson groupoid. We first show that $\OmegakmultG{1}$ carries  a natural Lie algebra structure.

\begin{theorem}\label{Thm:Omega1multLiealgebra}
	For a Poisson Lie groupoid $(\Gpd,P)$,  the space of multiplicative $1$-forms  $\OmegakmultG{1} $ is a Lie subalgebra of the Lie algebra $(\Omega^1(\Gpd),[{\tobefilledin},{\tobefilledin}]_P)$. 
\end{theorem}
\begin{proof}
For $\Theta_1,\Theta_2\in \OmegakmultG{1}$, we wish to show that \begin{equation}\label{Eqt:PbracketTheta12}
		[\Theta_1,\Theta_2]_P=d(\iota_{P^\sharp \Theta_1} \Theta_2) + \iota_{P^\sharp \Theta_1}  d\Theta_2-\iota_{P^\sharp \Theta_2}  d\Theta_1\end{equation} is also multiplicative. In fact, this follows by (1) of
  Lemma \ref{contraction}, (1) of Lemma \ref{contraction2}, and the fact that the de Rham differentials of multiplicative forms are still multiplicative.  
\end{proof}
 
 In general, it is hard to explicitly calculate  the Lie bracket on $\OmegakmultG{1}$. However, when $M$ is a single point, we have the following fact.
 \begin{example}
  Suppose that we are working with a Poisson Lie group $(G,P)$. According to Example \ref{form on group}, $\Omega_{\mathrm{mult}}^1(G)$ is in one-to-one correspondence with the set of $G$-invariant elements $\theta\in \mathfrak{g}^*$. In specific, $\Theta\in \Omega_{\mathrm{mult}}^1(G)$ corresponds to $\theta:=\Theta|_e$ and conversely $\Theta(g)=R_{g^{-1}}^*\theta$$(=L_{g^{-1}}^*\theta)$ for all $g\in G$.  Let $\Theta_1$ and $ \Theta_2\in \Omega_{\mathrm{mult}}^1(G)$ be arising from $G$-invariant $\theta_1$ and $\theta_2\in  \mathfrak{g}^*$ respectively. To get  $[\Theta_1,\Theta_2]_P$, it suffices to compute $[\Theta_1,\Theta_2]_P|_e$. According to the defining Equation \eqref{Eqt:PbracketTheta12}, we have, for all  $u\in \mathfrak{g}=T_eG$,
  $$\langle [\Theta_1,\Theta_2]_P|_e,u\rangle=\overrightarrow{u}|_e(P(\Theta_1,\Theta_2))=\langle (\mathcal{L}_{\overrightarrow{u}}P)|_e, \theta_1\wedge \theta_2 \rangle.$$
   It is a standard fact that $\mathcal{L}_{\overrightarrow{u}}P$ coincides with $(-\overrightarrow{d_* u})$ where $d_*: \wedge^\bullet \mathfrak{g} \to \wedge^{\bullet+1}\mathfrak{g}$ stems from the Lie bialgebra $(\mathfrak{g},\mathfrak{g}^*)$ induced by the Poisson Lie group $(G,P)$. Therefore, we get 
   $$\langle [\Theta_1,\Theta_2]_P|_e,u\rangle=-\langle d_* u, \theta_1\wedge \theta_2 \rangle
   = \langle  u, [\theta_1, \theta_2]_* \rangle.$$
   Here $[\cdot,\cdot]_*$ denotes the Lie bracket on $\mathfrak{g}^*$.
   So we conclude that as Lie algebras,   $\Omega_{\mathrm{mult}}^1(G)$ is isomorphic to the Lie subalgebra of $\mathfrak{g}^*$ consisting of $G$-invariant elements.
  
  \end{example}

 Let us recall the notion of Lie algebra crossed modules.
 \begin{definition}\cite{Gerstenhaber}
 	A Lie algebra crossed module consists of a pair of Lie algebras
 	$\thetaalgebradegone$ and $\galgebradegzero$, and a morphism of Lie algebras 
 	$\phi:~\thetaalgebradegone\to \galgebradegzero$ such that $\galgebradegzero$ acts on
 	$\thetaalgebradegone$ by derivations and satisfies, for all $x\in\galgebradegzero$,
 	$u,v\in\thetaalgebradegone$,
 	\begin{itemize}
 		\item[(1)] ${\phi (u)} \moduleaction v=\ba{u}{v}$;
 		\item[(2)] $\phi (  x \moduleaction u)=\ba{ x }{\phi (u)}$,
 	\end{itemize}where $\moduleaction$ denotes the $\galgebradegzero$-action on $\thetaalgebradegone$.
 	
 \end{definition}
 We shall write $\crossedmoduletriple{\thetaalgebradegone}{\phi}{\galgebradegzero}$ to denote a Lie algebra crossed module.
 \begin{definition}\label{Def:liealgebracrossedmodulemorphism}
 	A morphism $(f,F)$: $\crossedmoduletriple{\thetaalgebradegone}{\phi}{\galgebradegzero}\to \crossedmoduletriple{\thetaalgebradegone'}{\phi' }{\galgebradegzero'}$ of Lie algebra crossed modules consists of two Lie algebra morphisms $f: \thetaalgebradegone\to \thetaalgebradegone'$ and $F:
 	\galgebradegzero\to \galgebradegzero'$ which fit into the following commutative diagram	
 	\begin{equation*}
 		\xymatrix{
 			\thetaalgebradegone\ar@{^{}->}[d]_{\phi} \ar[r]^{f} &\thetaalgebradegone'\ar@{^{}->}[d]^{\phi'} \\
 			\galgebradegzero\ar[r]^{F} & \galgebradegzero'},
 	\end{equation*}	
 	and satisfy that $f(x\triangleright u)=F(x)\triangleright' f(u)$ for $x\in \galgebradegzero$ and $u\in \thetaalgebradegone$.

 \end{definition}

A typical instance of Lie algebra crossed modules arising from Lie groupoids is illustrated in \cite{BCLX} (see also \cites{BEL,OW}). 
Given a Lie groupoid $\Gpd$ with its tangent Lie algebroid $A$, the triple  $(\Gamma(A)\xrightarrow{T} \mathfrak{X}^1_{\mult}(\Gpd))$ where $T: u\mapsto \overleftarrow{u}-\overrightarrow{u}$, consists a Lie algebra crossed module.  {Here the action of $\mathfrak{X}^1_{\mult}(\Gpd)$ on $\Gamma(A)$ is determined by $\overleftarrow{X\triangleright u}=[X,\overleftarrow{u}]$ (or $\overrightarrow{X\triangleright u}=[X,\overrightarrow{u}]$).}
 In analogy to  $T$,   we have a map of vector spaces (following Lemma \ref{Lem:exactcharpairs})
 \begin{equation*}\label{Eqt:crossedmoduleJmap}
 \Omega^{1}(M)\xrightarrow{~J~} \OmegakmultG{1},\qquad \gamma\mapsto s^*\gamma-t^*\gamma.\end{equation*}

In the case of a Poisson Lie groupoid   $(\Gpd,P)$, a canonical Lie algebra crossed module structure can be found that underlies both the aforementioned vector spaces and the linear map $J$. This key finding is attributed to Ortiz and Waldron \cite{OW}. Their investigation demonstrated that the set of multiplicative sections of any $\mathcal{LA}$-groupoid possesses the structure of a strict Lie $2$-algebra, which is presented using a crossed module. By applying this general outcome to the specific situation of the cotangent bundle of a Poisson groupoid (which is highlighted in \cite[Example 7.3]{OW}), we arrive at the desired structure. To provide further background, we   rephrase this fact and present an alternative approach to it.

 \begin{theorem}\cite{OW} \label{Thm:LiealgebracrossedmodulePoissonLiegroupoid}
 	Continue to use notations as above.  Endow $\Omega^1(M)$ with the Lie bracket $[{\tobefilledin},{\tobefilledin}]_{\piM}$ and $\OmegakmultG{1}$ with $[{\tobefilledin},{\tobefilledin}]_P$.  Then, 
 	\begin{itemize}
 		\item[(1)]There exists  a Lie algebra action    $\actionmapaction{\tobefilledin}{\tobefilledin}$:
 		$\OmegakmultG{1}\otimes\Omega^1 (M)\rightarrow  \Omega^1 (M)$ such that
 		$$[\Theta,s^*\gamma]_P=s^*(\actionmapaction{\Theta}{\gamma})\,,\quad\forall \Theta\in \OmegakmultG{1}, \gamma\in \Omega^1 (M);
 		$$
 		\item[(2)]The  triple $(\Omega^{1}(M)\xrightarrow{~J~} \OmegakmultG{1})$ forms a Lie algebra crossed module where   $J$ is defined by $J(\gamma)=s^*\gamma-t^*\gamma$.
 	\end{itemize}

 \end{theorem}


We will prove this theorem  more directly  by utilizing the theory of characteristic pairs explained in Section \ref{Section:multforms-charpairs}.
 In what follows,  for the Poisson structure $P\in \XkmultG{2}$ on $\Gpd$, we shall denote by $p\in \Gamma(TM\otimes A)$ the leading term of $P$.

 We need the following technical Lemma.

\begin{lemma}\label{Lem:Thetastartstar}Let $P$ be in $\XkmultG{2} $ and 
	  $p\in \Gamma(TM\otimes A)$ the leading term of $P$. 	
	For all $\Theta\in \OmegakmultG{1}$ and $\gamma\in \Omega^1(M)$, one has 
\begin{eqnarray}\label{co3}
[\Theta,s^*\gamma]_P=s^*(\widehat{e}(\liftingd \iota_\gamma p)+\iota_{p^\sharp(\theta)}(d\gamma)),\quad\mbox{and~}~ [\Theta,t^*\gamma]_P=t^*(\widehat{e}(\liftingd \iota_\gamma p)+\iota_{p^\sharp(\theta)}(d\gamma)),\end{eqnarray}
where $(e,\theta)$  with $e\in Z^1(\jet \Gpd, T^*M)$ and $\theta\in \Gamma(A^*)$   is the $(0,1)$-characteristic pair of $\Theta$ and $\widehat{e}\in Z^1(\jet A,T^*M)$ is the infinitesimal of $e$.  Moreover, if $(\Gpd,P)$ is a Poisson groupoid, then the   map $\actionmapaction{\tobefilledin}{\tobefilledin}$:
$\OmegakmultG{1}\otimes\Omega^1 (M)\rightarrow  \Omega^1 (M)$   defined   by 
\begin{equation}
\label{Eqt:Omega1multGonOmega1M}
\actionmapaction{\Theta}{\gamma}
= \widehat{e}(\liftingd \iota_\gamma p)+\iota_{p^\sharp(\theta)}(d\gamma),
\end{equation} 
  is a Lie algebra action.
\end{lemma}
\begin{proof}  The leading term   of $P^\sharp(\Theta)$ is  $p^\sharp(\theta)$ (by Proposition \ref{cpmap}).
By (1) and (2) of Lemma \ref{contraction}, we have $P^\sharp\Theta\in \XkmultG{1}$ and 
\begin{eqnarray*}
\iota_{P^\sharp \Theta} (s^*\gamma)=\iota_{s^*\gamma} P^\sharp \Theta=s^*(\iota_\gamma p^\sharp(\theta) )=-s^*(p(\gamma,\theta)).
\end{eqnarray*} Similarly,  from $s_* P^\sharp \Theta=p^\sharp(\theta)$,  we obtain
\[\iota_{P^\sharp \Theta} (s^*d\gamma)=s^*\big(   \iota_{s_* P^\sharp \Theta}(d\gamma)\big) =s^* \big(   \iota_{p^\sharp(\theta)}(d\gamma)\big) .\]
Using these two identities, we can compute   
\begin{eqnarray*}
[\Theta, s^*\gamma]_P&=&\mathcal{L}_{P^\sharp\Theta}(s^*\gamma)-\iota_{P^\sharp(s^*\gamma)} d\Theta\\ &=&
d\iota_{P^\sharp\Theta}(s^*\gamma)+\iota_{P^\sharp\Theta} (s^*d\gamma)-\iota_{\overleftarrow{\iota_{\gamma} p}} d\Theta \quad (\mbox{by (2) of Lemma \ref{contraction}})\\ &=&-s^*(d(p(\gamma,\theta)))+s^* \big(   \iota_{p^\sharp(\theta)}(d\gamma)\big)-s^*(\widehat{d\Theta}(\iota_\gamma p))\quad (\mbox{by (2) of Lemma \ref{contraction2}}).
\end{eqnarray*}
Here $\widehat{d\Theta}=\mathrm{pr}_{\Gamma(A^*\otimes T^*M)}(d\Theta)|_M$ is the leading term of $d\Theta$. By Proposition \ref{Prop:differentialofcharpairs}, we have  
\[\widehat{d\Theta}(\iota_\gamma p)=-d(\theta(\iota_\gamma p))-\widehat{e}(\liftingd \iota_\gamma p), \]
where $\widehat{e}\in Z^1(\jet A, T^*M)$ is the infinitesimal of $e\in Z^1(\jet \Gpd,T^*M)$.  So  we get the first equality of \eqref{co3}:
\begin{eqnarray*}
[\Theta, s^*\gamma]_P=s^*\bigl(\widehat{e}(\liftingd \iota_\gamma p)+   \iota_{p^\sharp(\theta)}(d\gamma)\bigr) =s^*(\actionmapaction{\Theta}{\gamma}).\end{eqnarray*}
The other one is proved in a similar manner.

Furthermore, if $P$ is Poisson, then by the Jacobi identity of $[\tobefilledin,\tobefilledin]_P$, we have 
\[
[\Theta',[\Theta, s^*\gamma]_P]_P+[\Theta,[s^*\gamma,\Theta']_P]_P+[s^*\gamma,[\Theta',\Theta]_P]_P=0,\qquad \forall \Theta, \Theta'\in \Omega^1_{\mult}(\Gpd),\gamma\in \Omega^1(M), \]
It follows immediately  	 that 
\[s^*\bigl(\Theta'\triangleright(\Theta\triangleright \gamma)-\Theta\triangleright(\Theta'\triangleright \gamma)-[\Theta',\Theta]_P\triangleright \gamma\bigr)=0.\]
Since $s^*$ is injective,  we proved that the map $\triangleright$ defines an action of $\Omega^1_{\mult}(\Gpd)$ on $\Omega^1(M)$.
\end{proof}

We also need a standard fact.
\begin{lemma}\label{Lem:W1}
	\cite{W1}  The source map $s: \Gpd\to M$ is a Poisson map and the target map $t: \Gpd\to M$ is an anti-Poisson map. Moreover, for any $\gamma,\eta\in \Omega^1(M)$, we have $[s^*\gamma,t^*\eta]_P=0$.  As a direct consequence, we have \begin{equation}
	\label{Eqt:ststarbracketclosed}
	[s^*\gamma-t^*\gamma,s^*\eta-t^*\eta]_P=s^*[\gamma,\eta]_{\piM}-t^*[\gamma,\eta]_{\piM},\qquad\forall \gamma,\eta\in \Omega^1(M).\end{equation}
\end{lemma}

Recall that the base manifold $M$ is equipped with an induced Poisson structure $\piM=s_* P\in \mathfrak{X}^2 (M)$  (see \cite{W1}). By our formula \eqref{Bpi}, we have $\piM=-\frac{1}{2}(1\otimes \rho)p$. 

We now finish the proof of Theorem \ref{Thm:LiealgebracrossedmodulePoissonLiegroupoid}.
\begin{proof} Statement (1) is proved by   Lemma \ref{Lem:Thetastartstar}. For  	(2), we note that   $J:~\Omega^{1}(M)\to  \OmegakmultG{1}$ is a morphism of Lie algebras (by Equation \eqref{Eqt:ststarbracketclosed}). 
			To prove that  $\Omega^{1}(M)\xrightarrow{~J~} \OmegakmultG{1}$ is a Lie algebra crossed module, it suffices to show  
	\[[\gamma,\gamma']_{\piM}=(J\gamma)\triangleright \gamma',\qquad\mbox{and}\quad J(\Theta\triangleright \gamma)=[\Theta,J(\gamma)]_P.\]
	In fact, the first one follows from     
	\[s^*[\gamma,\gamma']_{\piM}=[s^*\gamma,s^*\gamma']_P=[s^*\gamma-t^*\gamma,s^*\gamma']_P=s^*((J\gamma)\triangleright \gamma') \]
	(by  Lemma \ref{Lem:W1} and the definition of $\triangleright$). The second is a direct consequence of Lemma \ref{Lem:Thetastartstar}.  	
\end{proof}

 \begin{proposition}\label{morphismcm}
 	\label{pibracket} Let $(\Gpd,P)$ be a Poisson groupoid and 
 	$p\in \Gamma(TM\otimes A)$ the leading term of $P$.  
 	The map \[P^\sharp: (\OmegakmultG{1},[{\tobefilledin},{\tobefilledin}]_P)\to (\XkmultG{1},[{\tobefilledin},{\tobefilledin}]),\qquad \alpha\mapsto \iota_\alpha P\] is a Lie algebra morphism. Moreover, the pair $(P^\sharp,p^\sharp)$ constitutes a  morphism of Lie algebra crossed modules:
 	\begin{equation*}
 		\xymatrix{
 			\Omega^{1}(M)\ar@{^{}->}[d]_{J} \ar[r]^{p^\sharp} & \Gamma(A) \ar@{^{}->}[d]^{T} \\
 			\OmegakmultG{1}\ar[r]^{P^\sharp} & \XkmultG{1}},
 	\end{equation*}
 	where $T$ is defined by $u\mapsto \overleftarrow{  u}-\overrightarrow{  u}$ for $u\in \Gamma(A)$. 
\end{proposition}
 \begin{proof}
 We verify that $(P^\sharp,p^\sharp)$ is a morphism of Lie algebra crossed modules. First, $P^\sharp$ is a Lie algebra morphism. Second, by (2) of Lemma \ref{contraction},  
 we have 
 \[P^\sharp\circ J(\gamma)=P^\sharp(s^*\gamma-t^*\gamma)=\overleftarrow{p^\sharp \gamma}-\overrightarrow{p^\sharp \gamma}=T\circ p^\sharp(\gamma),\qquad \gamma\in \Omega^1(M).\]
 So $P^\sharp\circ J= T\circ p^\sharp$ holds true. 
Next we show the relation \[p^\sharp(\Theta\triangleright \gamma)=(P^\sharp\Theta)\triangleright (p^\sharp \gamma).\]
 In fact, using  (2) of Lemma \ref{contraction} again, we have
 \[ \overleftarrow{p^\sharp(\Theta\triangleright \gamma)}=P^\sharp s^*(\Theta\triangleright \gamma)=P^\sharp[\Theta,s^*\gamma]_P=[P^\sharp \Theta, \overleftarrow{p^\sharp\gamma}]=\overleftarrow{(P^\sharp\Theta)\triangleright (p^\sharp \gamma)}.\] Then we obtain the desired relation because   the left translation is injective. The fact that $p^\sharp$ is a Lie algebra morphism follows by direct verification: 
 \[p^\sharp[\gamma,\gamma']_{\piM}=p^\sharp(J\gamma\triangleright\gamma')=(P^\sharp J\gamma)\triangleright (p^\sharp \gamma')=(Tp^\sharp \gamma)\triangleright (p^\sharp \gamma')=[p^\sharp\gamma,p^\sharp \gamma']_A.\]
 \end{proof}
 
 {
 \begin{example}Suppose that a Poisson manifold $(M,\underline{P})$ admits a 
 	symplectic groupoid $(\mathcal{G},\omega)$ which integrates the Lie algebroid 
 	$T^*M$ arising from the Poisson structure $\underline{P}$.  In this case, the pair $((\omega^{\sharp})^{-1},\id)$ forms an isomorphism of Lie algebra crossed modules:
\begin{equation*}
 		\xymatrix{
 			\Omega^{1}(M)\ar@{^{}->}[d]_{J} \ar[r]^{\id} &\Omega^{1}(M) \ar@{^{}->}[d]^{T} \\
 			\OmegakmultG{1}\ar[r]^{(\omega^{\sharp})^{-1}} & \XkmultG{1}\,.}
 	\end{equation*} 
 \end{example}
 }

\subsection{The   DGLA  of multiplicative  forms on a Poisson groupoid} 
On a general Poisson manifold $(N,P)$, the space of all degree forms $\Omega^\bullet(N)=\oplus_{k=0}^n \Omega^k(N)$, where $n=\mathrm{dim} (N)$, admits a graded Lie bracket known as the Schouten-Nijenhuis bracket which   is extended by Leibniz rule from   the  Lie bracket \eqref{Eqt:Pbracket1forms} of $1$-forms $\Omega^1(N)$, and also denoted by $[{\tobefilledin},{\tobefilledin}]_P$. So we have a GLA $(\Omega^\bullet(N),[\tobefilledin,\tobefilledin]_P)$. {Equipped with the de Rham differential $d$, the triple $(\Omega^\bullet(N),[\tobefilledin,\tobefilledin]_P,d)$ is a DGLA. In fact, we have
\begin{eqnarray}\label{dg}
d[\alpha,\beta]_P=[d\alpha,\beta]_P+(-1)^{k-1}[\alpha,d\beta]_P,\qquad \forall \alpha\in \Omega^k(N), \beta\in \Omega^l(N).
\end{eqnarray}
See \cite{ST}.}
Also, the induced map 
\[\wedge^\bullet P^\sharp: (\Omega^\bullet(N),[\tobefilledin,\tobefilledin]_P,d)\to (\mathfrak{X}^\bullet (N),[\tobefilledin,\tobefilledin],[P,\tobefilledin])\] defined by 
\[(\wedge^k P^\sharp)(\alpha_1\wedge \cdots \wedge \alpha_k)=P^\sharp(\alpha_1)\wedge \cdots \wedge P^\sharp(\alpha_k)\]
is a morphism of DGLAs.   
 In other words, $\wedge^\bullet P^\sharp$ is   a morphism of {GLA}s and a cochain map:
\[(\wedge^{\bullet+1} P^\sharp)(d\alpha)=[P,(\wedge^\bullet P^\sharp)\alpha],\qquad \forall \alpha\in \Omega^\bullet(N).\]
Here we take the convention that when $\bullet=0$, $\wedge^0 P^\sharp$ reduces to the identity map $C^\infty(N)\to C^\infty(N)$. 

On a Poisson Lie groupoid $(\Gpd,P)$,  it is natural to expect that $\OmegakmultG{\bullet}$ also admits a DGLA structure $([{\tobefilledin},{\tobefilledin}]_P,d)$. We will  prove this fact and find  some  more interesting conclusions. 
Let us first recall the notion of {GLA} crossed modules (also known as $\mathbb{Z}$-graded Lie $2$-algebras).
\begin{definition} \cite{BCLX}
	A {GLA} crossed module  $\crossedmoduletriple{\thetaalgebradegone}{\phi}{\galgebradegzero}$ consists of a pair of {GLA}s
	$\thetaalgebradegone$ and $\galgebradegzero$, and a morphism of {GLA}s
	$\phi:~\thetaalgebradegone\to \galgebradegzero$ such that $\galgebradegzero$ acts on
	$\thetaalgebradegone$ and satisfies, for all $x,y\in\galgebradegzero$,
	$u,v\in\thetaalgebradegone$,
	\begin{itemize}
		\item[(1)] ${\phi (u)} \moduleaction v=\ba{u}{v}$;
		\item[(2)] $\phi (  x \moduleaction u)=\ba{ x }{\phi (u)}$,
	\end{itemize}where $\moduleaction$ denotes the $\galgebradegzero$-action on $\thetaalgebradegone$. 
\end{definition}

What we need is an enhanced version of this definition.
	 {
	\begin{definition} 
	A DGLA crossed module is a {GLA} crossed module $\crossedmoduletriple{\thetaalgebradegone}{\phi}{\galgebradegzero}$
	as defined above, where  $\thetaalgebradegone$ and $\galgebradegzero$ are both DGLAs, 
	$\phi:~\thetaalgebradegone\to \galgebradegzero$ is a morphism of DGLAs, and 	
	  the action $\triangleright$ of $\galgebradegzero$   on
	$\thetaalgebradegone$ is compatible with the relevant differentials:
	$$d_\thetaalgebradegone (x\triangleright  u)=(d_\galgebradegzero x)\triangleright u+(-1)^{|x|}x\triangleright d_\thetaalgebradegone u,\quad\forall  x\in\galgebradegzero, u\in \thetaalgebradegone.$$
	\end{definition}}
Morphisms of GLA and DGLA crossed modules are defined in the same fashions as that of Definition \ref{Def:liealgebracrossedmodulemorphism}.

\begin{example}\label{exBCLX}\cite{BCLX} Let $\Gpd$ be a Lie groupoid.
	{The space $\XkmultG{\bullet}$ of multiplicative multi-vector fields on $\Gpd$ is a graded vector space (not an algebra). It constitutes a {GLA} (after degree shifts),   the Schouten bracket being its structure map. Indeed, we have a {GLA} crossed module 
		\[\Gamma(\wedge^\bullet A) \xrightarrow{T} \XkmultG{\bullet} ,\qquad u\mapsto \overleftarrow{u}-\overrightarrow{u}.\]
		where $X\triangleright u\in \Gamma(\wedge^{k+l-1} A)$ is determined by the relation 
		\[\overleftarrow{X\triangleright u}=[X,\overleftarrow{u}],\qquad (\mathrm{or}, \qquad \overrightarrow{X\triangleright u}=[X,\overrightarrow{u}]),\qquad X\in \XkmultG{k}, u\in \Gamma(\wedge^l A).\]
	}
Note that we regard $\Gamma(\wedge^0 A)$ as $C^\infty(M)$ and $\XkmultG{0}$ as multiplicative functions on $\Gpd$. The action of $\XkmultG{0}$ on $\Gamma(\wedge^0 A)$ is simply trivial.
	\end{example}
	 {
\begin{example}Continuing the above example, if we are given a multiplicative Poisson bivector field $P$ on the Lie groupoid $\Gpd$,  then $(\mathfrak{X}^\bullet_{\mult}(\Gpd),[\tobefilledin,\tobefilledin],[P,\tobefilledin])$ becomes a DGLA equipped with the differential $[P,\tobefilledin]: \mathfrak{X}^\bullet_{\mult}(\Gpd)\to \mathfrak{X}^{\bullet+1}_{\mult}(\Gpd)$. It also induces a differential $\delta_P:\Gamma(\wedge^\bullet A)\to \Gamma(\wedge^{\bullet+1} A)$ defined by
	$$\overleftarrow{\delta_P u}=[P,\overleftarrow{u}],\qquad \forall u\in \Gamma(\wedge^\bullet A),  $$
	so that $(\Gamma(\wedge^\bullet A),[\tobefilledin,\tobefilledin]_A,\delta_P)$ is a DGLA.   Now, the {GLA} crossed module
 \[\Gamma(\wedge^\bullet A) \xrightarrow{T} \XkmultG{\bullet} ,\qquad u\mapsto \overleftarrow{u}-\overrightarrow{u}.\]
in Example \ref{exBCLX} is indeed a DGLA  crossed module. To see it, we need to show 
\[T (\delta_P u)=[P,Tu],\qquad \mbox{and}\qquad \delta_P(X\triangleright u)=[P,X]\triangleright u+(-1)^{k-1}X\triangleright \delta_P u,\quad \forall X\in \mathfrak{X}^k_{\mult}(\Gpd).\]
Let us examine these two equations: We have
\[T(\delta_P u)=\overleftarrow{\delta_P u}-\overrightarrow{\delta_P u} =[P,\overleftarrow{u}]-[P,\overrightarrow{u}]=[P,Tu],\]
and 
\begin{eqnarray*}
\overleftarrow{\delta_P(X\triangleright u)}&=&\overleftarrow{[P,X\triangleright u]}=[P,\overleftarrow{X\triangleright u}]=[P,[X,\overleftarrow{u}]]\\ &=&[[P,X],\overleftarrow{u}]+(-1)^{k-1}[X,[P,\overleftarrow{u}]]\\ &=&\overleftarrow{[P,X]\triangleright u}+(-1)^{k-1}\overleftarrow{X\triangleright (\delta_P u)},
\end{eqnarray*}
where   the graded Jacobi identity of the Schouten bracket is applied.
\end{example}
}

We present our main result, which notably improves upon the Ortiz-Waldron Theorem \ref{Thm:LiealgebracrossedmodulePoissonLiegroupoid}.
\begin{theorem}\label{multiform}
	Let  $(\Gpd,P)$ be a Poisson Lie groupoid. 
	\begin{itemize}
		\item[(1)] With respect to the graded Lie bracket $[{\tobefilledin},{\tobefilledin}]_P$ and  {the de Rham differential $d$},  the space $\OmegakmultG{\bullet}$ is a  {sub DGLA} of $\Omega^\bullet(\Gpd)$.
		\item[(2)] Endow $\Omega^\bullet(M)$ with the graded Lie bracket $[{\tobefilledin},{\tobefilledin}]_{\piM}$ and the de Rham differential $d$, where $\piM$ is the  Poisson structure on $M$ induced from $(\Gpd,P)$. The triple  $(\Omega^{\bullet}(M)\xrightarrow{~J~} \OmegakmultG{\bullet})$ consists a DGLA crossed module where $J$ is   defined by
		\begin{equation}\label{Eqt:crossedmoduleJmapalldegrees}
		 J(\gamma):= s^*\gamma-t^*\gamma,\quad \forall \gamma\in \Omega^{\bullet}(M)\end{equation}
	and the action map $\triangleright$ of $\OmegakmultG{\bullet}$ on $\Omega^{\bullet}(M)$ is uniquely determined by the relation $$ s^*(\Theta\triangleright \gamma)=[\Theta,s^*\gamma]_P,\quad \forall \Theta\in \Omega^k_{\mult}(\Gpd),\, \gamma\in \Omega^l(M).$$
		\item[(3)] The map $\wedge^\bullet P^\sharp$ sends multiplicative $k$-forms on $\Gpd$ to multiplicative $k$-vector fields, and thereby,  \[\wedge^\bullet P^\sharp:  (\OmegakmultG{\bullet},[{\tobefilledin},{\tobefilledin}]_P, {d})\to (\XkmultG{\bullet},[{\tobefilledin},{\tobefilledin}], {[P,\tobefilledin]})\]
		is a   morphism of DGLAs. (When $\bullet=0$, we treat $\wedge^0 P^\sharp$ as the identity map on the space of multiplicative functions on $\Gpd$.)
		\item[(4)] The map $\wedge^\bullet P^\sharp$ together with $\wedge^\bullet p^\sharp$ is a morphism of DGLA crossed modules:
		\begin{equation*}
			\xymatrix{
				\Omega^{\bullet}(M)\ar@{^{}->}[d]_{J} \ar[r]^{\wedge^\bullet p^\sharp} & \Gamma(\wedge^\bullet A) \ar@{^{}->}[d]^{T} \\
				\OmegakmultG{\bullet}\ar[r]^{\wedge^\bullet P^\sharp} & \XkmultG{\bullet}},
		\end{equation*}
		where $p=\mathrm{pr}_{\Gamma(TM\otimes A)} P|_M\in \Gamma(TM\otimes A)$ is the leading term of $P$.   (When $\bullet=0$, we treat $\wedge^0 p^\sharp$ as the identity map on   $C^\infty(M)$.)

	\end{itemize}

\end{theorem}

We should note that 
the wedge product of multiplicative forms  is not multiplicative in general. So one can \textit{not} deduce that the graded Lie bracket $[{\tobefilledin},{\tobefilledin}]_P$ on $\OmegakmultG{\bullet}$ is extended from the one on $\OmegakmultG{1}$.

To prove Theorem \ref{multiform}, we need to set up some basic formulas and facts. For a bivector field $P\in \mathfrak{X}^2(N)$,  we define $P^\sharp: \Omega^k(N)\to \Omega^{k-1}(N)\otimes \mathfrak{X}^1(N)$ by
$$
P^\sharp(\alpha_1\wedge\cdots\wedge \alpha_k):=\sum_{i=1}^k (-1)^{i+k} \alpha_1\wedge\cdots\wedge \widehat{\alpha_i}\wedge \cdots \wedge \alpha_k\otimes P^\sharp(\alpha_i).
$$ 
Then for $\alpha\in \Omega^k(N)$ and $\beta\in \Omega^l(N)$, define
$\iota_{P^\sharp \alpha} \beta\in \Omega^{k+l-2}(N)$ by
\begin{eqnarray}\label{complex}
&&\nonumber \iota_{P^\sharp (\alpha_1\wedge\cdots \wedge \alpha_k)} (\beta_1\wedge \cdots \wedge \beta_l)\\\nonumber &=& 
\sum_{i=1}^k (-1)^{i+k} \alpha_1\wedge\cdots\wedge \widehat{\alpha_i}\wedge \cdots \wedge \alpha_k\wedge \iota_{ P^\sharp(\alpha_i)}(\beta_1\wedge \cdots \wedge \beta_l) 
\\
&=&\sum_{ i,j}  (-1)^{i+k+j-1}  (\iota_{P^\sharp \alpha_i} \beta_j) \alpha_1\wedge\cdots\wedge \widehat{\alpha_i}\wedge \cdots \wedge \alpha_k \wedge \beta_1\wedge \cdots \wedge \widehat{\beta_j}\wedge \cdots \wedge \beta_l  .
\end{eqnarray}
For every $k$-form $\alpha\in \Omega^k(N)$, we denote by $\alpha^\sharp: \wedge^{k-1} TN\to T^*N$ the map 
\[\alpha^\sharp(X_1,\cdots,X_{k-1})=\alpha(X_1,\cdots,X_{k-1},\tobefilledin),\qquad X_i\in \mathfrak{X}^1(N).\]
 
With  notations as above,  we can verify the following identity. For all $X_1,\cdots,X_{k+l-3} \in \mathfrak{X}^1(N)$, one has 
\begin{eqnarray}\nonumber  
(\iota_{P^\sharp \alpha} \beta)^\sharp (X_1,\cdots,X_{k+l-3})&=&
 \sum_{\sigma\in \mathrm{Sh}(k-1,l-2)}(-1)^\sigma \beta^\sharp \big(P^\sharp\alpha^\sharp (X_{\sigma_1},\cdots, X_{\sigma_{k-1}}),X_{\sigma_{k},}\cdots, X_{\sigma_{k+l-3}}\big)\\\nonumber  &&-(-1)^{kl} \sum_{\tau\in \mathrm{Sh}(l-1,k-2)}(-1)^\tau\alpha^\sharp \big(P^\sharp\beta^\sharp (X_{\tau_1},\cdots, X_{\tau_{l-1}}),X_{\tau_{l}},\cdots, X_{\tau_{k+l-3}}\big).\\\label{important}&&
\end{eqnarray}

\begin{lemma}\label{multiPbr}
On a Poisson manifold $(N,P)$, for $\alpha\in \Omega^k(N)$ and $\beta\in \Omega^l(N)$, we have
\begin{equation}\label{Eqt:polyPbracket}[\alpha,\beta]_P= \iota_{P^\sharp \alpha} d\beta+(-1)^{k-1}d\iota_{P^\sharp \alpha} \beta-(-1)^{(k-1)(l-1)}\iota_{P^\sharp \beta} d\alpha,\end{equation}
where $\iota_{P^\sharp \alpha} \beta$ is defined by \eqref{complex}.
\end{lemma}
 
In the existing literature, a more common formula  of $[\tobefilledin,\tobefilledin]_P$ is of the form 
\[[\alpha,\beta]_{P}=(-1)^{k-1}(\mathcal{L}_{P}(\alpha\wedge \beta)-\mathcal{L}_{P}(\alpha)\wedge \beta)-\alpha\wedge \mathcal{L}_{P} \beta,\qquad \alpha\in \Omega^k(N),\beta\in \Omega^l(N).\]See  \cite{Koszul}.
Here $\mathcal{L}_{P}: \Omega^n(N)\to \Omega^{n-1}(N)$ is defined by $\mathcal{L}_{P}=\iota_P\circ d-d\circ \iota_P$,  and $\iota_P:\Omega^n(N)\to \Omega^{n-2}(N)$ is the contraction. The bracket $[\tobefilledin,\tobefilledin]_P$ is also known as the \textit{Koszul bracket}. From the formula as described above, one can prove Equation \eqref{Eqt:polyPbracket}.  
   For completeness, we sketch a direct proof of Equation \eqref{Eqt:polyPbracket}. 
\begin{proof}
If $\alpha=\alpha_1\wedge\cdots \wedge \alpha_k$ and $\beta=\beta_1\wedge \cdots \wedge \beta_l$, then by the Leibniz rule, we have
\begin{eqnarray*}
[\alpha,\beta]_P&=&[\alpha_1\wedge\cdots \wedge \alpha_k,\beta_1\wedge \cdots \wedge \beta_l]_P\\ &=&\sum_{i,j}(-1)^{i+j} [\alpha_i,\beta_j]_P\wedge \alpha_1\wedge \cdots \widehat{\alpha_i}\wedge \cdots\wedge \widehat{\beta_j}\wedge \cdots \wedge \beta_l\\ &=&\sum_{i,j}
(-1)^{i+j}(\iota_{P^\sharp \alpha_i} d\beta_j+d\iota_{P^\sharp \alpha_i} \beta_j -\iota_{P^\sharp \beta_j} d\alpha_i)\wedge \alpha_1\wedge \cdots \widehat{\alpha_i}\wedge \cdots\wedge \widehat{\beta_j}\wedge \cdots \wedge \beta_l.
\end{eqnarray*}By definition of $P^\sharp$, we have
\begin{eqnarray*}
\iota_{P^\sharp \alpha} d \beta&=&\sum_{i,j}(-1)^{i+k+j-1} \iota_{\alpha_1\wedge \cdots \widehat{\alpha_i}\wedge \cdots \wedge \alpha_k \otimes P^\sharp \alpha_i} (\beta_1\wedge \cdots \wedge d\beta_j\wedge \cdots\wedge \beta_l)\\ &=&\sum_{i,j}(-1)^{i+k+j-1}\big((-1)^{k-1}
(\iota_{P^\sharp \alpha_i} d\beta_j)\wedge \alpha_1\wedge \cdots \wedge\widehat{\alpha_i}\wedge \cdots \wedge  \widehat{\beta_j}\wedge \cdots \wedge \beta_l\\ &&+\sum_{p>j}(-1)^{p} (\iota_{P^\sharp \alpha_i} \beta_p) \alpha_1\wedge \cdots \wedge \widehat{\alpha_i}\wedge \cdots \wedge d\beta_j\wedge \cdots \wedge \widehat{\beta_p}\wedge  \cdots \wedge\beta_l\\ &&+\sum_{p<j} (-1)^{p-1} (\iota_{P^\sharp \alpha_i} \beta_p) \alpha_1\wedge \cdots \wedge \widehat{\alpha_i}\wedge \cdots  \wedge \widehat{\beta_p}\wedge \cdots \wedge d\beta_j\wedge  \cdots \wedge\beta_l\big),
\end{eqnarray*}
\begin{eqnarray*}
d\iota_{P^\sharp \alpha} \beta&=&\sum_{i,j} (-1)^{i+j+k-1}d((\iota_{P^\sharp \alpha_i} \beta_j)\wedge \alpha_1\wedge \cdots \wedge\widehat{\alpha_i}\wedge \cdots \wedge\widehat{\beta_j}\wedge \cdots \wedge \beta_l)\\ &=& 
\sum_{i,j} (-1)^{i+j+k-1}\big(d(\iota_{P^\sharp \alpha_i} \beta_j)\wedge \alpha_1\wedge \cdots \wedge \widehat{\alpha_i}\wedge \cdots \wedge \widehat{\beta_j}\wedge \cdots \wedge \beta_l
\\ &&+\sum_{p< i} (-1)^{p-1}(\iota_{P^\sharp \alpha_i} \beta_j)\alpha_1\wedge \cdots  \wedge d\alpha_p\wedge \cdots \wedge\widehat{\alpha_i}\wedge  \cdots \widehat{\beta_j}\wedge \cdots \wedge \beta_l
\\ &&+\sum_{p> i} (-1)^{p}(\iota_{P^\sharp \alpha_i} \beta_j)\alpha_1\wedge \cdots \wedge \widehat{\alpha_i}\wedge \cdots \wedge d\alpha_p \cdots \wedge\widehat{\beta_j}\wedge \cdots \wedge \beta_l
\\ &&+\sum_{p<j}(-1)^{p+k} (\iota_{P^\sharp \alpha_i} \beta_j) \alpha_1\wedge \cdots \wedge \widehat{\alpha_i}\wedge \cdots \wedge d\beta_p\wedge  \cdots  \wedge\widehat{\beta_j}\wedge \cdots \wedge \beta_l
\\ &&+\sum_{p>j}(-1)^{p+k-1} (\iota_{P^\sharp \alpha_i} \beta_j) \alpha_1\wedge \cdots \wedge \widehat{\alpha_i}\wedge  \cdots  \wedge\widehat{\beta_j}\wedge \cdots \wedge d\beta_p\wedge \cdots \wedge \beta_l
 \big),
\end{eqnarray*}and 
\begin{eqnarray*}
\iota_{P^\sharp \beta} d \alpha&=&\sum_{i,j}(-1)^{j+l+i-1} \iota_{\beta_1\wedge \cdots \widehat{\beta_j}\wedge \cdots \wedge \beta_l \otimes P^\sharp \beta_j} (\alpha_1\wedge \cdots \wedge d\alpha_i\wedge \cdots\wedge \alpha_k)\\ &=&\sum_{i,j}(-1)^{j+l+i-1}
\big((-1)^{l-1}(\iota_{P^\sharp \beta_j} d\alpha_i)\wedge \beta_1\wedge \cdots \wedge\widehat{\beta_j}\wedge \cdots \wedge  \widehat{\alpha_i}\wedge \cdots \wedge \alpha_k\\ &&
+\sum_{p<i} (-1)^{p-1} (\iota_{P^\sharp \beta_j} \alpha_p) \beta_1\wedge \cdots \wedge \widehat{\beta_j}\wedge \cdots \wedge \widehat{\alpha_p}\wedge \cdots\wedge d\alpha_i \wedge \cdots \wedge \alpha_k\\ &&+
\sum_{p>i} (-1)^p (\iota_{P^\sharp \beta_j} \alpha_p) \beta_1\wedge \cdots \wedge \widehat{\beta_j}\wedge \cdots \wedge d\alpha_i\wedge \cdots\wedge \widehat{\alpha_p}\wedge \cdots \wedge \alpha_k\big)
\\ &=&
\sum_{i,j}(-1)^{j+l+i-1}
\big((-1)^{k(l-1)}(\iota_{P^\sharp \beta_j} d\alpha_i)\wedge \alpha_1\wedge \cdots \wedge \widehat{\alpha_i}\wedge \cdots \wedge \beta_1\wedge \cdots \wedge\widehat{\beta_j}\wedge \cdots \wedge\beta_l  \\ &&
+\sum_{p<i} (-1)^{p-1+(l-1)k} (\iota_{P^\sharp \beta_j} \alpha_p)\alpha_1 \wedge \cdots\wedge \widehat{\alpha_p}\wedge \cdots \wedge d\alpha_i \wedge \cdots \wedge \widehat{\beta_j}\wedge \cdots \wedge \beta_l\\ &&+
\sum_{p>i} (-1)^{p+(l-1)k} (\iota_{P^\sharp \beta_j} \alpha_p) \alpha_1\wedge \cdots \wedge d\alpha_i\wedge \cdots \widehat{\alpha_p} \wedge \cdots \wedge \widehat{\beta_j}\wedge \cdots \wedge \beta_l\big).
\end{eqnarray*}
Taking the summation of these formulas, the second and third terms of $\iota_{P^\sharp \alpha} d\beta$  cancel out with the fourth and fifth terms of $(-1)^{k-1} d\iota_{P^\sharp \alpha} \beta$, and the  second and third terms of $-(-1)^{(k-1)(l-1)} \iota_{P^\sharp \beta} d\alpha$ cancel out with the third and second terms of $(-1)^{k-1} d\iota_{P^\sharp \alpha} \beta$. Combining these calculations  is the formula that we expect.   
\end{proof}


\begin{proposition}\label{multiPbr2}Let $\Gpd$ be a Lie groupoid. 
For all multiplicative forms $\alpha\in \OmegakmultG{k}$ and $\beta\in \OmegakmultG{l}$, if $P\in\XkmultG{2}$, then the contraction  $\iota_{P^\sharp\alpha} \beta\in \Omega^{k+l-2}(\Gpd)$ defined by \eqref{complex} is also multiplicative.
\end{proposition}
\begin{proof}  
To prove that the $(k+l-2)$-form $\iota_{P^\sharp\alpha} \beta$ is multiplicative, it suffices to show that
\[(\iota_{P^\sharp\alpha} \beta)^\sharp: \oplus^{k+l-3} T\Gpd\to T^*\Gpd\]
is a Lie groupoid morphism. 
As $\alpha$, $\beta$, and $P$ are all multiplicative, the three maps 
\[\alpha^\sharp: \oplus^{k-1} T\Gpd\to T^*\Gpd,\qquad \beta^\sharp: \oplus^{l-1} T\Gpd\to T^*\Gpd,\quad\mbox{~and~}~ P^\sharp: T^*\Gpd\to T\Gpd\]
are all groupoid morphisms. Thus the compositions
\[\oplus^{k+l-3} T\Gpd=\oplus^{k-1}T\Gpd ~\oplus~(\oplus^{l-2}T\Gpd)\xrightarrow{\alpha^\sharp\oplus Id}T^*\Gpd~\oplus~(\oplus^{l-2}T\Gpd)\xrightarrow{P^\sharp\oplus Id}T\Gpd~\oplus~(\oplus^{l-2}T\Gpd)\xrightarrow{\beta^{\sharp}} T^*\Gpd,\]
and 
\[\oplus^{k+l-3} T\Gpd=\oplus^{l-1}T\Gpd ~\oplus~(\oplus^{k-2}T\Gpd)\xrightarrow{\beta^\sharp\oplus Id}T^*\Gpd~\oplus~(\oplus^{k-2}T\Gpd)\xrightarrow{P^\sharp\oplus Id}T\Gpd~\oplus~(\oplus^{k-2}T\Gpd)\xrightarrow{\alpha^{\sharp}} T^*\Gpd,\]
are both Lie groupoid morphisms as well.
By Equation \eqref{important}, $(\iota_{P^\sharp\alpha} \beta)^\sharp$ is the summation of a series of the above two compositions. Based on the interchange law \eqref{interlaw} of $T^*\Gpd$, it is also a Lie groupoid morphism.
\end{proof}

 \begin{lemma}\label{for3}
For all integers $k$ and $\Theta\in \OmegakmultG{k}$, we have $(\wedge^k P^\sharp)  \Theta  \in \XkmultG{k}$.
\end{lemma}

\begin{proof} 
For any $\alpha_g^i\in T_g^*\Gpd$ and $\beta_r^i\in T_r^* \Gpd$ such that $s(\alpha^i_g)=t(\beta_r^i)$, since $P^\sharp: T^*\Gpd\to T\Gpd$ is a Lie groupoid morphism and $\Theta$ is multiplicative, we have 
\begin{eqnarray*}
( (\wedge^k P^\sharp)  \Theta)(\alpha_g^1\cdot \beta_r^1,\cdots, \alpha_g^k \cdot \beta_r^k)&=&(-1)^k \Theta(P^\sharp(\alpha_g^1\cdot \beta_r^1),\cdots,P^\sharp(\alpha_g^k\cdot \beta_r^k))\\ &=&(-1)^k\Theta(P^\sharp(\alpha_g^1)\cdot P^\sharp(\beta_r^1),\cdots,P^\sharp(\alpha_r^k)\cdot P^\sharp(\beta_r^k))\\ &=&(-1)^k\Theta(P^\sharp(\alpha_g^1),\cdots, P^\sharp(\alpha_g^k))+(-1)^k\Theta(P^\sharp(\beta_r^1),\cdots, P^\sharp(\beta_r^k))\\ &=&(\wedge^k P^\sharp) (\Theta)(\alpha_g^1,\cdots,\alpha_g^k)+(\wedge^k P^\sharp) (\Theta)(\beta_r^1,\cdots,\beta_r^k).\end{eqnarray*}
This property implies that $(\wedge^kP^\sharp)\Theta \in \XkmultG{k}$.
\end{proof}


We are ready to finish the proof of Theorem \ref{multiform}.
\begin{proof} We first prove  Statement $\mathrm{(1)}$.
By Lemma \ref{multiPbr}, Proposition \ref{multiPbr2} and the fact that the de Rham differential preserves  multiplicativity,    the graded Lie bracket $[{\tobefilledin},{\tobefilledin}]_P$ is closed on multiplicative forms. Thus $\OmegakmultG{\bullet}\subset \Omega^\bullet(\Gpd)$ is a graded Lie subalgebra  {and   a sub DGLA as well}.

For $\mathrm{(2)}$, 
we need a fact: for $\alpha\in \OmegakmultG{k}$ and $\gamma\in \Omega^l(M)$, there exists a unique $(k+l-1)$-form $\omega\in \Omega^{k+l-1}(M)$ such that 
\begin{equation}\label{Eqt:tempomega}[\alpha,s^*\gamma]_P(=\iota_{P^\sharp \alpha} ds^* \gamma+d\iota_{P^\sharp \alpha} s^*\gamma-\iota_{P^\sharp s^*\gamma} d\alpha)=s^*\omega.\end{equation}
To see it, we need to show that, for all $X\in \ker s_*$,  $Y_i\in \mathfrak{X}^1(\Gpd)$   ($i=1,\cdots,k+l-2$), 
\begin{eqnarray}\label{property}
[\alpha,s^*\gamma]_P(X,Y_1,\cdots, Y_{k+l-2})=0.
\end{eqnarray}
Indeed, by Equation \eqref{important}, we have 
\begin{eqnarray*}
&&(\Omega^1(\Gpd)\ni) ~(\iota_{P^\sharp \alpha} s^*\gamma)^\sharp(X,Y_1,\cdots, Y_{k+l-4})\\ &=&\sum_{\sigma\in \mathrm{Sh}(k-2,l-2)} (-1)^\sigma (s^*\gamma)^\sharp (P^\sharp \alpha^\sharp (X,Y_{\sigma_1},\cdots Y_{\sigma_{k-2}}),Y_{\sigma_{k-1}},\cdots,Y_{\sigma_{k+l-4}})\\ &&-(-1)^{kl}\sum_{\tau\in \mathrm{Sh}(l-1,k-3)} (-1)^\tau \alpha^\sharp (P^\sharp s^*\gamma(Y_{\tau_1},\cdots,Y_{\tau_{l-1}}),X,Y_{\tau_{l}},\cdots Y_{\tau_{k+l-4}}).
\end{eqnarray*}

We claim that all the terms above  vanish. For this we examine that 
\begin{eqnarray*}
&&(s^*\gamma)^\sharp (P^\sharp \alpha^\sharp (X,Y_{\sigma_1},\cdots Y_{\sigma_{k-2}}),Y_{\sigma_{k-1}},\cdots,Y_{\sigma_{k+l-4}})\\ &=&s^*\big(\gamma^\sharp(s_*P^\sharp \alpha^\sharp(X,Y_{\sigma_1},\cdots Y_{\sigma_{k-2}}),Y_{\sigma_{k-1}},\cdots,Y_{\sigma_{k+l-4}})\big)\\ &=&s^*\big(\gamma^\sharp(P^\sharp \alpha^\sharp(s_*X,s_*Y_{\sigma_1},\cdots s_*Y_{\sigma_{k-2}}),s_*Y_{\sigma_{k-1}},\cdots,s_*Y_{\sigma_{k+l-4}})\big)=0,
\end{eqnarray*}
where we have used the facts that $P^\sharp$ and $\alpha^\sharp$ are Lie groupoid morphisms, which commute with the source maps, and $s_*X=0$. Similarly, we can verify that
\begin{eqnarray*}
&&\alpha^\sharp (P^\sharp s^*\gamma(Y_{\tau_1},\cdots,Y_{\tau_{l-1}}),X,Y_{\tau_{l}},\cdots Y_{\tau_{k+l-4}})\\&=&(-1)^{k+l} \langle P^\sharp s^*\gamma(Y_{\tau_1},\cdots,Y_{\tau_{l-1}}),\alpha^\sharp(X,Y_{\tau_{l}},\cdots Y_{\tau_{k+l-4}})\rangle
\\&=&-(-1)^{k+l} \langle \gamma(Y_{\tau_1},\cdots,Y_{\tau_{l-1}}), s_*P^\sharp\alpha^\sharp(X,Y_{\tau_{l}},\cdots Y_{\tau_{k+l-4}})\rangle\\ &=& -(-1)^{k+l} \langle \gamma(Y_{\tau_1},\cdots,Y_{\tau_{l-1}}), P^\sharp\alpha^\sharp(s_*X,s_*Y_{\tau_{l}},\cdots s_*Y_{\tau_{k+l-4}})\rangle=0.\end{eqnarray*}
So we have \[(\iota_{P^\sharp \alpha} s^*\gamma)^\sharp(X,Y_1,\cdots, Y_{k+l-4})=0.\]

Now, due to the expression of $[\alpha,s^*\gamma]_P$, we obtain the desired  \eqref{property}.

Once we obtain $\omega$ which is subject to Equation \eqref{Eqt:tempomega}, we can  define the action of $\alpha$ on $\gamma$ by setting $\alpha\triangleright \gamma:=\omega$. Thanks to the graded Jacobi identity of $[\tobefilledin,\tobefilledin]_P$ and the injectivity of $s^*$, we see that $\triangleright$ defines an action of the {GLA} $\Omega^\bullet_{\mult}(\Gpd)$ on $\Omega^\bullet(M)$.   {Moreover by Equation \eqref{dg}, we have
\[d[\alpha,s^*\gamma]_P=[d\alpha,s^*\gamma]_P+(-1)^{k-1} [\alpha,s^*d\gamma]_P,\]
which implies   $$d(\alpha\triangleright \gamma)=(d\alpha)\triangleright \gamma+(-1)^{k-1} \alpha\triangleright (d\gamma)$$ as $s^*$ is injective. So $\triangleright$ is   compatible with the differentials. One further checks that  $(\Omega^{\bullet}(M)\xrightarrow{~J~} \OmegakmultG{\bullet})$ defines a DGLA  crossed module.}

 {The space of multiplicative forms $\OmegakmultG{\bullet}\subset \Omega^\bullet(\Gpd)$ is preserved by the de Rham differential and the space of multiplicative multi-vector fields $\XkmultG{\bullet}\subset \mathfrak{X}^\bullet(\Gpd)$ is closed under the Schouten bracket.} So Statement $\mathrm{(3)}$ follows from Lemma \ref{for3} and  {the fact that $\wedge^\bullet P^\sharp: (\Omega^\bullet(\Gpd),[\tobefilledin,\tobefilledin]_P,d)\to (\mathfrak{X}^\bullet(\Gpd),[\tobefilledin,\tobefilledin],[P,\tobefilledin])$ is a morphism of DGLAs.} Finally,  Statements $\mathrm{(1)}$-$\mathrm{(3)}$ together with Lemma \ref{contraction} imply  $\mathrm{(4)}$.


\end{proof}

 \begin{example} Let $M$ be a smooth manifold.
The cotangent bundle  $T^*M\to M$ is an abelian Lie group bundle, and can be regarded as a special Lie groupoid; its source and target maps are both the bundle projection  $T^*M\to M$ and the multiplication is the fiberwise addition. With the canonical symplectic form $\omega=d\alpha$, $\alpha$ being the canonical Liouville-Poincar\'{e} $1$-form, $T^*M\to M$ is a symplectic Lie groupoid.
Let us take the standard local coordinates $(q^i,p^j)$ of $T^*M$, where $q^i$ is the coordinate on $M$ and $p^j$ that of the fibre. 
Then one can write $\omega=dq^i\wedge dp_i$. 

{As the groupoid multiplication of $T^*M$ is given by the fiberwise addition,  multiplicative multi-vector fields and multiplicative forms are indeed  linear multi-vector fields (\cite{ILX}) and linear forms (\cite{BC}), respectively.}  So according to \cite{BC}, a multi-vector field $\Pi\in \mathfrak{X}^k(T^*M)$ is multiplicative  if  it is locally of the form
\[  
\Pi=\frac{1}{k!}\Pi^{i_1\cdots i_k}_{j}(q)p^j \frac{\partial}{\partial p^{i_1}}\wedge \cdots\wedge  \frac{\partial}{\partial p^{i_k}}+\frac{1}{(k-1)!}\Pi^{i_1\cdots i_{k-1},j}(q)\frac{\partial}{\partial p^{i_1}}\wedge\cdots \wedge\frac{\partial}{\partial p^{i_{k-1}}}\wedge \frac{\partial}{\partial q^j}.
\]
Similarly, a $k$-form $\Theta\in \Omega^k(T^*M)$ is multiplicative if it is of the form 
\[
\Theta=\frac{1}{k!}\Theta_{i_1\cdots i_k,j}(q)p^j dq^{i_1}\wedge \cdots\wedge dq^{i_k}+\frac{1}{(k-1)!}\Theta_{i_1\cdots i_{k-1},j}(q)dq^{i_1}\wedge\cdots dq^{i_{k-1}}\wedge dp^j.
\]
As $\omega^\sharp: T (T^*M)\to T^*(T^*M)$ maps $\frac{\partial}{\partial p^{i}}$ to $-dq^{i}$ and 
$\frac{\partial}{\partial q^{i}}$ to $dp^i$, we see  that  $\omega^\sharp$ establishes an isomorphism between $\mathfrak{X}^k_{\mult}(T^*M)$ and $\Omega^k_{\mult}(T^*M)$.

 {Next, we find the Lie algebra crossed module and the {GLA} crossed module structures stemming from  the Poisson Lie groupoid $\Gpd=T^*M\to M$. The Poisson structure is $P=\frac{\partial}{\partial q^i}\wedge \frac{\partial}{\partial p^i}$ corresponding to the earlier symplectic structure $\omega$.
\begin{itemize}
\item [\rm (1)] Since source and target maps $s$ and $t$ are one and the same,   the map $J$ is just trivial: $\Omega^1(M)\xrightarrow{J=0} \Omega^1_{\mult}(\Gpd)$. The Lie bracket on $\Omega^1(M)$ is also trivial, whereas the Lie bracket on $\Omega^1_{\mult}(\Gpd)$ is listed below:
\begin{eqnarray*}
[\Theta_{i,j}(q) p^j dq^i, \Theta'_{a,b}(q) p^b dq^a]_P&=&\Theta_{i,j}(q)\Theta'_{a,i}(q) p^j dq^a-\Theta'_{a,b}(q) \Theta_{i,a}(q) p^b dq^i;\\ 
{[\Theta_{i,j}(q) p^j dq^i,\Theta'_l(q)dp^l]_P}&=&\Theta'_l(q) \frac{\partial \Theta_{i,j}(q)}{\partial q^l}p^j dq^i;\\ 
{[\Theta_k(q)dp^k,\Theta'_l(q)dp^l]_P}&=&-\Theta_k(q)\frac{\partial \Theta'_l(q)}{\partial q^k} dp^l+\Theta'_l(q)\frac{\partial \Theta_k(q)}{\partial q^l} dp^k.
\end{eqnarray*}
The action of $\Omega^1_{\mult}(\Gpd)$ on $\Omega^1(M)$ is given by
\[\big(\Theta_{i,j}(q) p^j dq^i+\Theta_l(q)dp^l\big)\triangleright (\gamma_k(q)dq^k)=-\gamma_k(q)\Theta_{i,k}(q)dq^i-\Theta_l(q)\frac{\partial \gamma_k(q)}{\partial q^l} dq^k.\]
\item [\rm (2)] For the same reasons,  we have the trivial map $\Omega^\bullet(M)\xrightarrow{J=0} \Omega^\bullet_{\mult}(\Gpd)$ and trivial  graded Lie bracket on $\Omega^\bullet(M)$. The graded Lie bracket on $\Omega^\bullet_{\mult}(\Gpd)$ is   as described below. 

Let $I=\{i_1,i_2,\cdots,i_k\}$ be a multi-index and  $dq^I=dq^{i_1}\wedge \cdots \wedge dq^{i_k}$ a $k$-form on $M$. Similarly let  $A=\{a_1,a_2,\cdots,a_l\}$  and $dq^A=dq^{a_1}\wedge \cdots dq^{a_l}$ be an $l$-form. Denote by $I_s$ be the multi-index by removing $i_s$ from $I$. The notation $A_t$ is similar. We have computed the following:
\begin{eqnarray*}
&&[\Theta_{I,j}(q) p^j dq^{I}, \Theta'_{A,b}(q) p^b dq^{A}]_P\\ &=&(-1)^{k-s}\Theta_{I,j}(q)\Theta'_{A,i_s}(q) p^j dq^{I_s}\wedge dq^A-(-1)^{l-t}\Theta'_{A,b}(q) \Theta_{I,a_t}(q) p^b dq^{A_t}\wedge dq^I;\\ 
&&[\Theta_{I,j}(q) p^j dq^{I},\Theta'_{L,b}(q)dq^L\wedge dp^b]_P\\ &=&(-1)^{l-s}\Theta'_{L,b}(q)\Theta_{I,l_s}(q)dq^{L_s}\wedge dp^b\wedge dq^I-\Theta'_{L,b}(q) \frac{\partial \Theta_{I,j}(q)} {\partial q^{b}} p^j dq^L\wedge dq^I; \\ 
&&[\Theta_{K,a}(q)dq^K\wedge dp^a,\Theta'_{L,b}(q)dq^L\wedge dp^b]_P\\ &=&-\Theta_{K,a}(q)\frac{\partial \Theta'_{L,b}(q)}{\partial q^{a}} dq^{K}\wedge dq^L\wedge dp^b+\Theta'_{L,b}(q)\frac{\partial \Theta_{K,a}(q)}{\partial p^b} dq^L\wedge dq^K\wedge dp^a.
\end{eqnarray*}
The action of $\Omega^k_{\mult}(\Gpd)$ on $\Omega^l(M)$ is given by
\begin{eqnarray*}
&&\big(\Theta_{I,j}(q) p^j dq^{I}+\Theta_{K,a}(q)dq^K\wedge dp^a\big)\triangleright (\gamma_L(q)dq^L)\\ &=&-(-1)^{l-s}\gamma_L(q)\Theta_{I,l_s}(q)dq^{L_s}\wedge dq^I-\Theta_{K,a}(q)\frac{\partial \gamma_L(q)}{\partial q^a}dq^K\wedge dq^L.
\end{eqnarray*}
\end{itemize}
We can also write explicitly the Schouten bracket on multiplicative multi-vector fields,  which are omitted.
}
\end{example}

 Finally, as the infinitesimal counterpart   of (1) of Theorem \ref{multiform}, we know that  $\mathrm{IM}^\bullet (A)$, the space of IM-forms of    the tangent Lie algebroid of $\Gpd$, carries a graded Lie bracket structure. Hence  $\mathrm{IM}^\bullet (A)$ is a DGLA provided that the groupoid $\Gpd$ is Poisson (cf. Corollary \ref{Cor:cochaincomplexofIMA}).  The correspondence between Poisson Lie groupoids and Lie bialgebroids \cite{MX1,MX3} suggests a canonical DGLA structure on $\mathrm{IM}^\bullet (A)$ can be derived from any Lie bialgebroid $(A,A^*)$. This will be explored in future research.


\end{document}